\documentclass{amsart}

\usepackage{latexsym}
\usepackage{amssymb}
\usepackage{amsmath}

\newtheorem{lemma}{Lemma}[section]
\newtheorem{proposition}[lemma]{Proposition}
\newtheorem{corollary}[lemma]{Corollary}

\newtheorem{remark}[lemma]{Remark}
\newtheorem{theorem}[lemma]{Theorem}
\newtheorem{examples}[lemma]{Examples}
\newtheorem{example}[lemma]{Example}

\newcommand{\Hmm}[1]{\leavevmode{\marginpar{\tiny%
$\hbox to 0mm{\hspace*{-0.5mm}$\leftarrow$\hss}%
\vcenter{\vrule depth 0.1mm height 0.1mm width \the\marginparwidth}%
\hbox to 0mm{\hss$\rightarrow$\hspace*{-0.5mm}}$\\\relax\raggedright #1}}}

\begin{document}

\title[Pseudogroups]{Pseudogroups and their \'etale groupoids}

\author[M.~V.~Lawson]{Mark V.  Lawson$^1$}
\address{$^1$Department of Mathematics
and the
Maxwell Institute for Mathematical Sciences,
Heriot-Watt University,
Riccarton,
Edinburgh~EH14~4AS,
Scotland}
\email{markl@ma.hw.ac.uk }
\thanks{The first author was supported by an EPSRC grant (EP/F004184, EP/F014945, EP/F005881).
Both authors would like to thank Pedro Resende for his comments and suggestions on Section~2 of this paper.}

\author[ D.~H.~Lenz]{Daniel H. Lenz$^2$}
\address{$^2$ Mathematisches Institut, Friedrich-Schiller Universit\"at Jena, Ernst-Abb\'{e} Platz~2, 07743 Jena, Germany}
\email{ daniel.lenz@uni-jena.de }



\begin{abstract}
A pseudogroup is a complete infinitely distributive inverse monoid.
Such inverse monoids bear the same relationship to classical pseudogroups of transformations as frames do to topological spaces.
The goal of this paper is to develop the theory of pseudogroups motivated by applications to group theory, $C^{\ast}$-algebras and aperiodic tilings.
Our starting point is an adjunction between a category of pseudogroups and a category of \'etale groupoids
from which we are able to set up a duality between spatial pseudogroups and sober \'etale groupoids.
As a corollary to this duality, we deduce a non-commutative version of Stone duality involving what we call
boolean inverse semigroups and boolean \'etale groupoids,
as well as a generalization of this duality to distributive inverse semigroups.
Non-commutative Stone duality has important applications in the theory of $C^{\ast}$-algebras:
it is the basis for the construction of Cuntz and Cuntz-Krieger algbras
and in the case of the Cuntz algebras it can also be used to construct the Thompson groups. 
We then define coverages on inverse semigroups and the resulting presentations of pseudogroups.
As applications, we show that Paterson's universal groupoid is an example of a booleanization,
and reconcile Exel's recent work on the theory of tight maps with the work of the second author.
\end{abstract}

\maketitle

\section{Preliminaries}

A {\em frame} is a complete infinitely distributive lattice.
The theory of frames can be viewed as an approach to spaces in which open sets, and not points, are taken as basic.
This theory is interesting in its own right \cite{J} and a key ingredient in topos theory \cite{MM}.
Johnstone discusses the origins of frame theory in his notes to Chapter~II of his classic book \cite{J}.
One sentence there is significant, and somewhat surprising.
He writes on page~76:
\begin{quote}
It was Ehresmann \ldots and his student B\'enabou \ldots who first took the decisive step in regarding complete
Heyting algebras as `generalized topological spaces'.
\end{quote}
What Johnstone does not say is why Ehresmann was led to his frame-theoretic viewpoint of topological spaces.
In fact, Ehresmann's motivation was to develop an abstract theory of local structures in geometry.
Local structures, such as differential manifolds, are defined by means of atlases whose changes of co-ordinates belong to a pseudogroup of transformations.
The nature of the pseudogroup determines the nature of the local structure and it was for this reason that Ehresmann needed 
an abstract formulation of pseudogroups of transformations.
He used ordered groupoids but it subsequently became clear that a simpler, but equivalent, formulation was possible using inverse semigroups.
Thus a {\em complete abstract pseudogroup}, to use the terminology of Resende \cite{R1,R2}, is a complete, infinitely distributive inverse monoid;
we shall explain what this definition means later in this section.
In this paper, we shall simply call them {\em pseudogroups}.
Transformation pseudogroups are just pseudogroups of partial homeomorphisms of a topological space
and the idempotents of a transformation pseudogroup are just the partial identities on the open subsets of the space.
We can now see how frames arise: they are the partially ordered sets of idempotents of pseudogroups.
It was perhaps natural to disentangle frames from their roots and study them on their own terms.
But the premise of this paper is that we now need to go back and generalize the foundations of the classical theory of frames to pseudogroups.
This is not an empty exercise because it has become clear that the resulting theory provides the setting for the significant applications 
of inverse semigroup theory to $C^{\ast}$-algebras.
In fact, the theory in this paper arose out of a detailed analysis of the relationships that exists between inverse semigroups, 
topological groupoids and $C^{\ast}$-algebras \cite{Exel1,Exel2,Kel1,Kel2,Law4,LMS,L,P,Ren}.
We now recall some key definitions and establish some basic results.

A  {\em groupoid} is a set $G$
together with a partially defined associative
multiplication, denoted by concatenation, and an
involution $x\mapsto x^{-1}$, satisfying the following conditions:
\begin{itemize}
\item[(G1)] $(x^{-1})^{-1}=x$.
\item[(G2)]  If $xy$ and $yz$ exist, then $xyz$ exists as well.
\item[(G3)]  $x^{-1}x$ exists and if $xy$ exists as well then $x^{-1}xy=y$.
\item[(G4)] $xx^{-1}$ exists and if $zx$ exists as well then $zxx^{-1}=z$.
\end{itemize}
Elements of the form $x x^{-1}$ are called \textit{identities} of $G$ and the set of all identities of $G$ is denoted by $G_{o}$. 
Each groupoid comes with the maps
$\mathbf{d} \colon G\rightarrow G_{o}$ and $\mathbf{r} \colon G\rightarrow G_{o}$ defined by $\mathbf{d}(x) = x^{-1}x$ and $\mathbf{r}(x) = xx^{-1}$.
If $G$ carries a topology making the multiplication and inversion continuous, it is called a \textit{topological groupoid}. 
The most important class of topological groupoids are the {\em \'etale groupoids}.
Classically, an etale groupoid is a topological groupoid in which the domain map is a local homeomorphism.
However, in Theorem~5.18 of \cite{R2}, Resende characterizes them as those topological groupoids whose
frames of open sets form a semigroup under subset multiplication with an identity formed by the open set of all identities.
This is a fundamental observation in understanding how a topological structure such as an \'etale groupoid
can be related to algebraic structures such as quantales and pseudogroups.
For more background on topological groupoids, we refer the reader to \cite{R1,R2}.

Let $(P,\leq)$ be a poset.
A minimum element in $P$ is called  {\em zero} denoted by 0.
We shall usually assume that our posets have zeros.
For $x \in P$ define
$$x^{\downarrow} = \{y \in E \colon y \leq x \},$$
the {\em principal order ideal generated by $x$},
and
$$x^{\uparrow} = \{y \in E \colon y \geq x \},$$
the {\em principal filter generated by $x$}.
We extend this notation to subsets $A \subseteq P$
and define $A^{\downarrow}$ and $A^{\uparrow}$.
A subset $A$ such that $A = A^{\downarrow}$ is called an {\em order ideal}.
Observe that the intersection of order ideals is always an order ideal.
If $A$ is a finite set then $A^{\downarrow}$ is said to be a {\em finitely generated} order ideal.
A subset $A$ of $P$ is said to be {\em directed} if for each $a,b \in A$ there exists $c \in A$ such that $c <  a,b$.
A {\em filter} in $P$ is a directed subset $A$ such that $A = A^{\uparrow}$.
If $A$ is any directed subset then $A^{\uparrow}$ is a filter.
If $a$ and $b$ are elements of $P$ we write $a^{\downarrow} \cap b^{\downarrow} \neq 0$
to mean that there is some non-zero element below both $a$ and $b$.
If $a^{\downarrow} \cap b^{\downarrow} = 0$ then we say $a$ and $b$ are {\em orthogonal}.
If $X \subseteq P$ then $X^{\perp}$ denotes all the elements of $P$ orthogonal to every element of $X$.
Clearly the zero always belongs to this set.

For background on inverse semigroups we shall refer to \cite{Law2}.
One piece of notation we shall sometimes use is that we write $\mathbf{d}(s) = s^{-1}s$ and $\mathbf{r}(s) = ss^{-1}$.
It is worth stressing that the order used with inverse semigroups is always the natural partial order.
Observe that a necessary condition for a subset of an inverse semigroup to have a join is that
the elements in the set be pairwise compatible where elements $s$ and $t$ are {\em compatible} if $s^{-1}t$ and $st^{-1}$ are idempotents.
If these idempotents are in fact both zeros then $s$ and $t$ are said to be {\em orthogonal}.
Thus for inverse semigroups {\em complete} always means that every {\em compatible} subset has a join.
Compatibility plays a crucial role throughout this paper.  
An inverse semigroup $S$ is said to be {\em distributive} if it has joins of all finite
compatible subsets and multiplication distributes over the finite joins that exist. 
If this does not only holds for finite sets but for arbitrary sets the semigroup is called {\em infinitely distributive}.
Distributivity is heavily used in the sequel.
{\em Pseudogoup morphisms} will be semigroup homomorphisms that preserve compatible joins.
An inverse semigroup is said to be an {\em inverse $\wedge$-semigroup} if all binary meets exist.
It is an easy exercise to show, or see \cite{R1,R2}, that pseudogroups are inverse $\wedge$-semigroups.
A homomorphism between such semigroups that preserves the meet operation is called a {\em $\wedge$-homomorphism}.

The inverse semigroup $S$ is said to satisfy the {\em weak meet condition} if the intersection of any two principal order ideals is finitely generated as an order ideal.
This condition was introduced by Steinberg \cite{Stei} who called an inverse semigroup satisfying the condition a {\em weak semilattice}.
If $S$ is an inverse $\wedge$-semigroup then in fact $a^{\downarrow} \cap b^{\downarrow} = (a \wedge b)^{\downarrow}$.
Thus inverse semigroups satisfying the weak meet condition generalize inverse $\wedge$-semigroups.
In a weak semilattice, the intersection of any finite number of principal order ideals is finitely generated as an order ideal.

A {\em filter} in an inverse subsemigroup $S$ is a subset $F$ that is closed upwards under the natural partial order and directed.
The set of all filters $\mathsf{L}(S)$ on an inverse semigroup forms a groupoid as we shall now describe.
If $A$ is a filter define $\mathbf{d}(A) = (A^{-1}A)^{\uparrow}$ and $\mathbf{r}(A) = (AA^{-1})^{\uparrow}$.
Define a partial product on $\mathsf{L}(S)$ by
$A \cdot B = (AB)^{\uparrow}$ iff $\mathbf{d}(A) = \mathbf{r}(B)$.
By  \cite{L,LMS} we then have that $(\mathsf{L}(S),\cdot)$ is a groupoid  and is in fact the groupoid
underlying Paterson's universal groupoid of an inverse semigroup $S$ \cite{P}.
Filters also play an important role in Exel's work \cite{Exel1,Exel2} on relating inverse semigroups and $C^{\ast}$-algebras.
A filter is said to be {\em proper} if it does not contain zero.

\begin{remark} 
{\em The filters in this paper will always be assumed to be proper.} 
\end{remark}

We shall have occasion in this paper to study a variety of different kinds of filters but particularly important are the maximal or {\em ultrafilters}.
The set of ultrafilters forms a subgroupoid of the groupoid of all filters of an inverse semigroup. 
The groupoid of ultrafilters has attracted attention in various guises.  
In the case of inverse semigroups arsing from locally finite tilings, it is just the tiling groupoid \cite{Kel1, L}. 
In the case of inverse semigroups arising from locally finite graphs, 
it is  shown in \cite{L} to agree with the graph groupoid introduced by  Kumjian, Pask, Raeburn and  Renault in  \cite{KPRR}.

Ultrafilters in inverse semigroups satisfying the weak meet condition can be handled rather easily as we now show.
Observe first that by Zorn's Lemma every non-zero element of an inverse semigroup belongs to an ultrafilter.
We shall apply this observation frequently in what follows.

A finite non-empty subset $A \subseteq S$ is said to be {\em consistent} if there is a non-zero element $b \in S$
such that $b \leq a$ for all $a \in A$.
We write $b \leq A$ and call it a {\em lower bound} of $A$.
Observe that we do not require that $b \in A$.
An arbitrary subset of $S$ is said to be consistent if every finite non-empty subset is consistent.
Filters are examples of consistent subsets since they are directed.
If $A$ is a consistent subset and $a,b \in A$ then clearly $a^{\downarrow} \cap b^{\downarrow} \neq 0$.
We shall use this observation without further comment when working with consistent subsets.
The following is a consequence of Zorn's Lemma that we shall use repeatedly.

\begin{lemma} Let $S$ be a weak semilattice.
Every consistent subset is contained in a maximal consistent subset.
\end{lemma}
\begin{proof} Let $A$ be a consistent subset of an inverse semigroup $S$.
Consider the poset of all consistent subsets that contain $A$.
This is non-empty since $A$ itself belongs to this set.
The union of a chain of consistent subsets is consistent since consistency is defined in terms of finite subsets.
Thus the post of all consistent sets containing $A$ has a maximal member.
Thus every consistent subset is contained in a maximally consistent one.
\end{proof}

\begin{lemma} Let $S$ be a weak semilattice.
Let $A$ be a maximal consistent subset.
Let $a,b \in A$ and suppose that $a^{\downarrow} \cap b^{\downarrow} = \{c_{1}, \ldots, c_{m} \}^{\downarrow}$.
Then $c_{j} \in A$ for some $j$.
\end{lemma}
\begin{proof} Suppose that $\{c_{1}, \ldots, c_{m} \} \cap A = \emptyset$.
Since $A$ is a maximal consistent subset, the sets $\{c_{i} \} \cup A$ are inconsistent for each $i$.
It follows that for each $i$ there is finite subset $B_{i} \subseteq A$ such that $\{c_{i}\} \cup B_{i}$ is inconsistent.
Put $B = \bigcup_{i=1}^{m} B_{i} \cup \{a,b \}$.
Then $B \subseteq A$ and so $B$ is a finite consistent set.
Let $z$ be a non-zero lower bound of $B$
and
let $F$ be any ultrafilter containing $z$.
In particular, $a,b \in F$.
It follows that there is $c \in F$ such that $c \leq a,b$.
Thus for some $j$ we have that $c_{j} \in F$.
However, $\{c_{j} \} \cup B_{j} \subseteq F$ which implies that $\{c_{j} \} \cup B_{j}$ is a consistent set.
This is a contradiction.
It follows that $c_{j} \in A$ for some $j$, as claimed.
\end{proof}

The following proposition provides an alternative way of viewing ultrafilters.

\begin{proposition}  Let $S$ be a weak semilattice.
Every maximal consistent subset is an ultrafilter, and every ultrafilter is a maximal consistent subset.
\end{proposition}
\begin{proof}
Let $A$ be a maximal consistent subset.
We prove first that $A$ is a directed subset.
Let $a,b \in A$ where $a^{\downarrow} \cap b^{\downarrow} = \{c_{1}, \ldots, c_{m} \}^{\downarrow}$.
Then by Lemma~1.3, we have that $c_{j} \in A$ for some $j$ and clearly $c_{j} \leq a,b$.
It follows that $A^{\uparrow}$ is a filter.
But filters are consistent subsets and $A \subseteq A^{\uparrow}$ and so $A = A^{\uparrow}$.
We have therefore shown that $A$ is a filter and since $A$ is maximal consistent and all filters are 
consistent it follows that $A$ is an ultrafilter.
Conversely, let $A$ be an ultrafilter.
Then $A$ is certainly a consistent subset.
If it is not maximal consistent then by Lemma~1.2 it must be contained in a subset that is.
But by the above, such a subset would be an ultrafilter which is a contradiction.
It follows that $A$ is a maximal consistent subset.
\end{proof}

The next result is really a corollary to the above, but it is such an important result that
we promote it to a theorem.

\begin{theorem} Let $F$ be a proper filter in an inverse semigroup satisfying the weak meet condition.
Then $F$ is an ultrafilter if and only if it satisfies the following condition:
if $b$ is such that $b^{\downarrow} \cap a^{\downarrow} \neq 0$ for all $a \in F$ then $b \in F$.
\end{theorem}
\begin{proof} Let $F$ be an ultrafilter.
We prove that it satisfies the given condition.
Suppose that $b \notin F$.
By assumption $\{b\} \cup F$ is a consistent subset.
By Lemma~1.2, it is contained in a maximal consistent subset.
By Proposition~1.4, such a subset is an ultrafilter.
But this contradicts the fact that $F$ is an ultrafilter.
It follows that $b \in F$.
Conversely, let $F$ be a filter that satisfies the condition.
We prove that $F$ is an ultrafilter.
Suppose not.
Then $F \subseteq G$ where $G$ is any filter properly containing $F$.
Let $b \in G \setminus F$.
Then $\{b \} \cup F$ is a consistent set and so $b$ satisfies the premiss of the condition.
It follows that $b \in F$ which is a contradiction.
It follows that $F$ is an ultrafilter. 
\end{proof}

The distinction between semigroups and monoids leads to a distinction we shall have to make between
{\em distributive lattices} and {\em unital distributive lattices},
and {\em boolean algebras} and {\em unital boolean} algebras.
Thus for us a distributive lattice is one that does not necessarily have a top element,
and what we call boolean algebras are often referred to as {\em generalized boolean algebras}.
A {\em boolean space} is a hausdorff topological space with a basis of compact-open subsets.
A continuous map between spaces is {\em proper} if the inverse image of every compact set is compact.
Classical Stone duality states that the category of boolean algebras is dual to the category of boolean spaces.

In a distributive lattice $D$ a filter $F$ is said to be {\em prime} if $a \vee b \in F$ implies that either $a \in F$ or $b \in F$.
Generalizations of prime filters will play an important role in this paper.
The following result is well-known in the unital case, but we shall need it in the non-unital case as well.

\begin{proposition} \mbox{} 
\begin{enumerate}

\item In a distributive lattice every ultrafilter is prime.

\item A distributive lattice is boolean if and only if every prime filter is an ultrafilter.

\end{enumerate}
\end{proposition}
\begin{proof} (1) Let $F$ be an ultrafilter in a distributive lattice $D$.
Suppose that $c = a \vee b \in F$ but $a,b \notin F$.
By Theorem~1.5, there exists $f_{a} \in F$ such that $a \wedge f_{a} = 0$ and $f_{b} \in F$ such that $f_{b} \wedge b = 0$.
But $F$ is a filter and so $f = f_{a} \wedge f_{b} \in F$ which means that $f \wedge a = 0 = f \wedge b$.
But $f \wedge c = (f \wedge a) \vee (f \wedge b)$.
The lefthand-side is non-zero since it is a meet of elements in a filter,
but the righthand-side is zero.
This is a contradiction and so either $a \in F$ or $b \in F$.

(2) We prove first the unital case of our result.
Let $D$ be a unital boolean algebra and let $P$ be a prime filter.
Let $a \in D$ be such that $a \wedge p \neq 0$ for all $p \in P$.
Since $D$ is a unital boolean algebra we have that $1 = a \vee a'$.
But $1 \in P$ and $P$ is a prime filter.
Thus either $a \in P$ or $a' \in P$.
We cannot have $a' \in P$ because $a \wedge a' = 0$ which contradicts our assumption.
Thus $a \in P$ and so by Theorem~1.5, we have proved that $P$ is an ultrafilter.
Conversely, let $D$ be a unital distributive lattice such that every prime filter is an ultrafilter.
We prove that $D$ is boolean. 
We use Corollary~4.9 of \cite{J}.
Thus we shall prove that every prime ideal is maximal.
Let $I$ be a prime ideal in $D$.
Suppose that it is not a maximal proper ideal.
Then $I \subseteq J$ a maximal proper ideal by Lemma~2.3 of \cite{J}.
By Corollary~2.4 of \cite{J}, the ideal $J$ is also prime.
We have that $D \setminus J \subseteq D \setminus I$ where by Proposition~2.2 of \cite{J},
both $D \setminus J$ and $D \setminus I$ are prime filters.
By assumption, both must be ultrafilters and so $D \setminus J = D \setminus I$ giving $I = J$.
It follows that $I$ is a maximal proper ideal, as required.

We now turn to the non-unital case.
Let $D$ be a distributive lattice in which every prime filter is maximal.
We prove that $D$ is a boolean algebra.
To do this we have to prove that $e^{\downarrow}$ is a boolean algebra for each $e \in D$.
This can be achieved by showing that every prime filter in $e^{\downarrow}$ is an ultrafilter.
Let $F \subseteq e^{\downarrow}$ be a filter in $e^{\downarrow}$.
Then it is easy to check that $F^{\uparrow}$ is a filter in $D$.
In addition, if $F \subseteq G \subseteq e^{\downarrow}$ are filters then $F^{\uparrow} \subseteq G^{\uparrow}$,
and if $F^{\uparrow} = G^{\uparrow}$ then $F = G$.
Let $F \subseteq e^{\downarrow}$ be a prime filter in $e^{\uparrow}$.
Then $F^{\uparrow}$ is a prime filter in $D$.
If $F$ is not an ultrafilter in $e^{\downarrow}$ then there is a filter $G$ such that $F \subseteq G \subseteq e^{\downarrow}$.
But $F^{\uparrow}$ implies by assumption that $F^{\uparrow}$ is an ultrafilter and so $F^{\uparrow} = G^{\uparrow}$ giving $F = G$.
We have therefore proved that in the unital distributive lattice $e^{\downarrow}$ every prime filter is an ultrafilter and so $e^{\downarrow}$
is a unital boolean algebra, as required. 

Conversely, let $D$ be a boolean algebra.
We prove that every prime filter is an ultrafilter.
Let $P$ be a prime filter and let $a \in D$ be an element such that $a \wedge p \neq 0$ for all $p \in P$.
We shall prove that $a \in P$ from which it follows that $P$ is an ultrafilter by Theorem~1.5.
Choose $e \in P$ and put $P' = \{e \wedge p \colon p \in P \}$.
Then $P'$ is a prime filter in the unital boolean algebra $e^{\downarrow}$.
Observe that the element $a \wedge e$ has a non-empty meet with every element of $P'$.
But in a unital boolean algebra we have seen that every prime filter is an ultrafilter and so $P'$ is an ultrafilter
and thus by Theorem~1.5 we have that $a \wedge e \in P'$.
It follows that $a \in P$, as required. 
\end{proof}

\noindent
{\bf Notation }In this paper, we shall use a number of different groupoids constructed from an inverse semigroup $S$ by means of filters.
It may be helpful to list them here:
\begin{itemize}

\item $\mathsf{L}(S)$ the groupoid of all proper filters on the inverse semigroup $S$. When the non-proper filter $S$ is adjoined we in fact
obtain an inverse semigroup \cite{L,LMS}.

\item $\mathsf{G}(S)$ is the groupoid of all completely prime filters when $S$ is a pseudogroup.

\item $\mathsf{G}_{P}(S)$ is the groupoid of all prime filters when $S$ is a distributive inverse semigroup.

\item $\mathsf{G}_{M}(S)$ is the groupoid of all ultrafilters.

\item $\mathsf{G}_{\mathcal{C}}(S)$ is the groupoid of all $\mathcal{C}$-filters where $\mathcal{C}$ is a coverage on $S$.

\item  $\mathsf{G}_{u}(S)$ is Paterson's universal groupoid that consists of all filters (for us all proper filters) with the patch topology.

\item  $\mathsf{G}_{t}(S)$ is Exel's tight groupoid that consists of all tight filters with the patch topology.

\item $\mathsf{G}_{m}(S)$ is the groupoid of ultrafilters equipped with the patch topology.\\

\end{itemize}

We would like to acknowledge one other author who has carried out pioneering work in this area: Pedro Resende \cite{R1,R2,MR}.
In many ways our approaches are complementary since our interest is primarily in deriving dualities involving classes of inverse semigroups.
Our work also differs from his in a number of other ways:
Resende works with localic groupoids whereas we work with topological groupoids;
he uses quantales whereas our work is connected more to the theory of $C^{\ast}$-algebras;
and finally, his theory works at the level of objects only whereas we have introduced suitable morphisms.
However, the duality Theorem~2.23, at the level of objects, was explicitly proved in \cite{MR}.
We shall make direct links to Resende's work where appropriate.

\section{The adjunction}

We shall be interested in two categories which at this stage we only define at the level of objects:
the category $\mathbf{Inv}$ of pseudogroups and the category $\mathbf{Etale}$ of \'etale groupoids.
Our goal is to define a functor
$\mathsf{G} \colon \mathbf{Inv}^{op} \rightarrow \mathbf{Etale}$
that takes a pseudogroup and delivers an \'etale groupoid,
and a functor
$\mathsf{B} \colon \mathbf{Etale} \rightarrow \mathbf{Inv}^{op}$
that takes an \'etale groupoid and delivers a pseudogroup.
We begin by constructing both functors at the level of objects only in Section~2.1
and then explore how the constructions can be lifted to maps in Section~2.2.

\subsection{Objects only}

The easiest functor to define is the second.
Recall that a {\em local bisection} $A$ in a groupoid $G$ is a subset satisfying the two
conditions $A^{-1}A,AA^{-1} \subseteq G_{o}$.
We shall usually refer to `local bisections' simply as `bisections'.
Here $A^{-1}$ means the inverses of all elements of $A$, and the product is
just the usual multiplicaton of subsets of a category.
It is well-known, and in any event easy to prove, that the set of all bisections of a groupoid
forms a pseudogroup.
The idempotents of this semigroup are just the subsets of $G_{o}$ and the natural partial order is given by subset-inclusion.
If $G$ is now an \'etale groupoid, then we define $\mathsf{B}(G)$ to be the set of all open bisections.
We now have the object part of our functor.

\begin{proposition} Let $G$ be an \'etale groupoid.
Then $\mathsf{B}(G)$, the set of all open bisections of $G$ under subset multiplication, is a pseudogroup.
\end{proposition}

The description of the object part of our second functor is more complex and depends on an important class of filters.
A filter $F$ in a pseudogroup $S$ is said to be {\em completely prime} if $\bigvee a_{i} \in F$ implies that $a_{i} \in F$ for some $i$.
Such filters were defined in \cite{R1} where they were called {\em compatibly prime} and are a generalization of a concept important in frame theory \cite{J}.
Given a pseuodgroup $S$, we denote the set of all completely prime filters on $S$ by $\mathsf{G}(S)$.

\begin{lemma} Let $A$ be a filter in a pseudogroup $S$.
\begin{enumerate}

\item $A$ is completely prime if and only if $A^{-1}$ is completely prime.

\item $A$ is completely prime if and only if $\mathbf{d}(A)$ is completely prime.

\end{enumerate}
\end{lemma}
\begin{proof} (1) This is straightforward.

(2) Suppose that $A$ is completely prime.
We prove that $\mathbf{d}(A) = (A^{-1}A)^{\uparrow}$ is completely prime.
Let $x = \bigvee_{i} x_{i} \in A^{-1} \cdot A$.
Then $a^{-1}a \leq x$ for some $a \in A$;
this is always possible since if $a,b \in A$ then $(a \wedge b)^{-1}(a \wedge b) \leq a^{-1}b$.
Clearly $a^{-1}a = xa^{-1}a$ and so by infinite distributivity we have that $a^{-1}a = \bigvee_{i} x_{i}a^{-1}a$.
Thus again by infinite distributivity,
$a = \bigvee_{i} ax_{i}a^{-1}a$.
By assumption, $A$ is completely prime and so $ax_{i}a^{-1}a \in A$ for some $i$.
Thus $ax_{i} \in A$ since $A$ is upwardly closed.
However $a^{-1}ax_{i} \in A^{-1} \cdot A$.
Thus $x_{i} \in A^{-1} \cdot A$, again by upward closure, as required.

Suppose now that $\mathbf{d}(A)$ is completely prime.
We prove that $A$ is completely prime.
Let $a = \bigvee a_{i} \in A$.
Then $\mathbf{d}(a) = \bigvee \mathbf{d}(a_{i}) \in \mathbf{d}(A)$.
Where we use standard properties of compatible joins \cite{Law2}
By assumption $\mathbf{d}(a_{i}) \in \mathbf{d}(A)$.
It follows that $a_{i} = a\mathbf{d}(a_{i}) \in A$, as required.
\end{proof}

Clearly, $\mathbf{d}(A)$ and $\mathbf{r}(A)$ each contain idempotents.
The following is stated as Lemma~2.9 of \cite{Law3}.

\begin{lemma} Let $F$ be a filter.
Then $F$ contains an idempotent if and only if it is also an inverse subsemigroup.
\end{lemma}

We shall call filters of the above type {\em idempotent filters}.

\begin{lemma} Let $A$ and $B$ be completely prime filters such that $\mathbf{d}(A) = \mathbf{r}(B)$.
Then $(AB)^{\uparrow}$ is a completely prime filter such that $\mathbf{d}((AB)^{\uparrow}) = \mathbf{d}(B)$ and
$\mathbf{r}((AB)^{\uparrow}) = \mathbf{r}(A)$.
\end{lemma}
\begin{proof} By Lemma~2.7 of \cite{Law3}, we have that $(AB)^{\uparrow}$ is a filter.
We show that it is completely prime.
Let $x = \bigvee_{i} x_{i} \in A \cdot B$.
Then $ab \leq x$ for some $a \in A$ and $b \in B$.
Thus $ab = \bigvee_{i} x_{i} (ab)^{-1}ab$.
By infinite distributivity we have that
$a^{-1}ab = \bigvee_{i} a^{-1}x_{i} (ab)^{-1}ab$.
Since $A^{-1} \cdot A = B \cdot B^{-1}$ we have that $a^{-1}ab \in B$.
But $B$ is completely prime and so $a^{-1}x_{i}(ab)^{-1}ab \in B$.
Thus $aa^{-1}x_{i}(ab)^{-1}ab \in aB$ and so $x_{i} \in A \cdot B$, as required.

For the final claims, observe that
$(A \cdot B)^{-1} \cdot (A \cdot B) = ( (B^{-1}A^{-1})^{\uparrow}(AB)^{\uparrow})^{\uparrow}$
and that
$((B^{-1}A^{-1})^{\uparrow}(AB)^{\uparrow})^{\uparrow} = (B^{-1}A^{-1}AB)^{\uparrow}$.
Under our assumption that $A^{-1} \cdot A = B \cdot B^{-1}$,
we deduce that
$(B^{-1}A^{-1}AB)^{\uparrow} = B^{-1} \cdot B$.
The dual result is proved similarly.
\end{proof}

On $\mathsf{G}(S)$ define a partial binary operation by
$$A \cdot B = (AB)^{\uparrow} \text{ iff } \mathbf{d}(A) = \mathbf{r}(B).$$
The proof of the following is now straightforward:
essentially only associativity is left to prove and that is easy.

\begin{lemma}
For each pseudogroup $S$, the structure $(\mathsf{G}(S),\cdot)$ is a groupoid.
\end{lemma}

For each $s \in S$ define $X_{s}$ to be the set of all completely prime filters that contains $s$.
Clearly $X_{0} = \emptyset$ although other sets $X_{s}$ could also be empty.

\begin{lemma} Let $S$ be a pseudogroup.
\begin{enumerate}

\item $X_{s}$ is a bisection.

\item $X_{s}^{-1} = X_{s^{-1}}$.

\item $X_{s}X_{t} = X_{st}$.

\item $X_{s} \cap X_{t} = X_{s \wedge t}$.

\item If $s = \bigvee_{i} s_{i}$ then $\bigcup_{i} X_{s_{i}} = X_{s}$.

\end{enumerate}
\end{lemma}
\begin{proof} (1) Let $F,G \in X_{s}$ such that $\mathbf{d}(F) = \mathbf{d}(G)$.
By Lemma~2.11 of \cite{Law3}, this implies that $F = G$.
The dual result can be proved similarly.

(2) Immediate from the properties of the natural partial order.

(3) It is clear that $X_{s}X_{t} \subseteq X_{st}$.
Let $F \in X_{st}$.
Put $H = F^{-1} \cdot F$.
Then $F = (stH)^{\uparrow}$.
Put $A = (s(tHt^{-1})^{\uparrow} )^{\uparrow}$ and $B = (tH)^{\uparrow}$.
Then by Lemma~2.2(2), we have that $A \in X_{s}$ and $B \in X_{t}$, and $A \cdot B = F$.

(4) This is straightforward since filters are closed under binary meets.

(5) This is immediate from the definition of completely prime filters.
\end{proof}

Put $\tau = \{X_{s} \colon s \in S \}$.
By the lemma above, $\tau$ is a basis for a topology on $\mathsf{G}(S)$
and in what follows we shall always regard $\mathsf{G}(S)$ equipped with this topology.

\begin{lemma}
$\mathsf{G}(S)$ is a topological groupoid.
\end{lemma}
\begin{proof}
By the above lemma the inversion map is continuous.
Denote the set of composable elements in the groupoid $G$ by $G \ast G$ and denote the
multiplication map by $m \colon G \ast G \rightarrow G$.
We observe that
$$m^{-1}(X_{s}) = \left( \bigcup_{0 \neq ab \leq s} X_{a} \times X_{b} \right) \cap (\mathsf{G}(S) \ast \mathsf{G}(S))$$
for all $s \in S$.
The proof is straightforward and the same as step~3 of the proof of Proposition~2.22 of \cite{Law3}
and shows that $m$ is a continuous function.
\end{proof}

We can now state our second main result.

\begin{proposition} Let $S$ be a pseudogroup.
Then $\mathsf{G}(S)$, the set of all proper completely prime filters, is an \'etale groupoid.
\end{proposition}
\begin{proof} It remains to show that $\mathsf{G}(S)$ is \'etale.
There are a number of ways to prove this.
We could follow step~4 of the proof of Proposition~2.22 of \cite{Law3}.
We give a different proof here.

We show first that  $\mathsf{G}(S)_{o}$ is an open subspace of $\mathsf{G}(S)$.
Let $F$ be an identity in $\mathsf{G}(S)$.
Then by Lemma~2.3, $F$ is an inverse subsemigroup and so contains idempotents.
Let $e \in F$.
Then $F \in X_{e}$.
But every completely prime filter in $X_{e}$ contains an idempotent and so is an identity in the groupoid.
Thus $F \in X_{e} \subseteq \mathsf{G}(S)_{o}$ and is an open set.
Thus $\mathsf{G}(S)_{o}$ is an open set.

Next we show that the product of two open sets is an open set.
Let $X$ and $Y$ be any open sets.
By the definition of the topology, we may write
$X = \bigcup_{i} X_{s_{i}}$ and $Y = \bigcup_{j} X_{t_{j}}$.
Then we have
$$XY = \bigcup_{i,j} X_{s_{i}t_{j}}$$
by Lemma~2.6(3).
Thus the product of open sets is always open.
\end{proof}


We shall now describe the relationships between
$$S \text{ and } \mathsf{B}(\mathsf{G}(S)),
\text{ and }
G \text{ and } \mathsf{G}(\mathsf{B}(G)).$$

Let $S$ be a pseudogroup.
The set $X_{s}$, a set of completely prime filters, is a bisection by Lemma~2.6(1) and it is by definition open.
Thus $X_{s} \in \mathsf{B}(\mathsf{G}(S))$.
Define $\varepsilon \colon S \rightarrow \mathsf{B}(\mathsf{G}(S))$ by $s \mapsto X_{s}$.

\begin{proposition} The function $\varepsilon \colon S \rightarrow \mathsf{B}(\mathsf{G}(S))$ has the following properties:
\begin{enumerate}

\item It is a pseudogroup $\wedge$-morphism.

\item Every element of $\mathsf{B}(\mathsf{G}(S))$ is a compatible join of elements in the image of $\varepsilon$.

\item The map $\varepsilon$ is an isomorphism of monoids if and only if the pseudogroup $S$ has the additional property that for all
$s,t \in S$ we have that $X_{s} = X_{t}$ implies that $s = t$.

\end{enumerate}
\end{proposition}
\begin{proof} (1) The map is a semigroup homomorphism by Lemma~2.6(3),
it preserves binary meets by Lemma~2.6(4) and it preserves compatible joins by Lemma~2.6(5).
It is a monoid map essentially by Lemma~2.3.

(2) Each element of $\mathsf{B}(\mathsf{G}(S))$ is an open bisection and every open set, by definition,
is a union of open sets of the form $X_{s}$.

(3) Only one direction needs proving.
We assume that $X_{s} = X_{t}$ implies that $s = t$ for all $s,t \in S$.
It remains to prove that $\varepsilon$ is surjective.
This will follow if we can prove that $X_{s}$ and $X_{t}$ compatible implies that $s$ and $t$ are compatible.
In turn, this will follow is we can prove that if $X_{s}$ contains only idempotent filters then $s$ must be an idempotent.
Observe that the assumption implies that $X_{s} = X_{s \wedge s^{-1}s}$.
But then since $\varepsilon$ is injective we deduce that $s = s \wedge s^{-1}s$ and so $s$ is an idempotent.
\end{proof}

A pseudogroup $S$ is said to be {\em spatial} if $X_{s} = X_{t}$ implies that $s = t$ for all $s,t \in S$. 
By  part (3) of  the preceeding proposition, a pseudogroup is spatial if and only if $\varepsilon$ is an isomorphism of monoids.

The proof of the following is straightforward.

\begin{lemma} Let $G$ be an \'etale groupoid.
For each $g \in G$ define $F_{g}$ to be the set of all open bisections that contain $g$.
Then $F_{g}$ is a completely prime filter in the pseuodgroup $\mathsf{B}(S)$.
\end{lemma}

Let $G$ be a \'etale groupoid.
Define $\eta \colon G \rightarrow \mathsf{G}(\mathsf{B}(G))$ by $g \mapsto F_{g}$.
By the above lemma this is a well-defined map.
A functor $\alpha \colon G \rightarrow H$ is called a {\em cover} if it satisfies two conditions:
$\mathbf{d}(g) = \mathbf{d}(g')$ and $\alpha (g) = \alpha (g')$ implies that $g = g'$,
and if $\alpha (e) = \mathbf{d}(h)$ where $e$ is an identity then there exists $g \in G$
such that $\mathbf{d}(g) = e$ and $\alpha (g) = h$.

\begin{proposition}
The function  $\eta \colon G \rightarrow \mathsf{G}(\mathsf{B}(G))$ is a continuous covering functor.
\end{proposition}
\begin{proof}
Recall that in an \'etale groupoid the open bisections form a basis for the topology \cite{R1}.
We may deduce from this, and the fact that the multiplication function is continuous,
that if $O_{g}$ is an open bisection containing $g \in G$ and $g = hk$
then there are open bisections $h \in O_{h}$ and $k \in O_{k}$ such that $O_{h}O_{k} \subseteq O_{g}$.
It now readily follows that $\eta$ is a functor.

We now prove that $\eta$ is a covering functor.
Suppose that $\mathbf{d}(g) = \mathbf{d}(h)$ and $\eta (g) = \eta (h)$.
Then there is an open bisection $O$ that contains both $g$ and $h$.
But it then follows immediately from the definition of bisection that $g = h$.
Now suppose that $\eta (e) = \mathbf{d}(F)$ where $e$ is an identity and $F$
is a completely prime filter in $\mathsf{B}(G)$.
By definition $\mathbf{d}(F) = F_{e}$.
Let $b \in F$ be any open bisection.
By assumption $e \in b^{-1}b$.
Thus we may find $g \in b$ such that $e = g^{-1}g$.
Consider $F_{g}$.
Then $\mathbf{d}(F_{g}) = \mathbf{d}(F)$ and $b \in F_{g} \cap F$.
By Lemma~2.11 of \cite{Law3}, we have that $F_{g} = F$, as required.

It remains to show that $\eta$ is continuous.
Let $b \in \mathsf{B}(G)$ be an open bisection.
Then
$$g \in \eta^{-1} (X_{b}) \Leftrightarrow F_{g} \in X_{b} \Leftrightarrow b \in F_{g} \Leftrightarrow g \in b.$$
Thus $\eta^{-1}(X_{b}) = b$.
\end{proof}

An \'etale groupoid is said to be {\em sober} if the map $\eta$ is a homeomorphism.

\begin{proposition} \mbox{}
\begin{enumerate}

\item For every \'etale groupoid $G$ the pseudogroup $\mathsf{B}(G)$ is spatial.

\item For every pseudogroup $S$ the \'etale groupoid $\mathsf{G}(S)$ is sober.

\end{enumerate}
\end{proposition}
\begin{proof} (1) Let $U$ and $V$ be two distinct open bisections in $\mathsf{B}(G)$.
Without loss of generality, there exists $g \in U$ and $g \notin V$.
But then $F_{g}$ is a completely prime filter in $\mathsf{B}(G)$ that contains $U$ and omits $V$.

(2) Let $S$ be a pseudogroup.
We show that every completely prime filter in  $\mathsf{B}(\mathsf{G}(S))$
is of the form $F_{f}$ where $f \in \mathsf{G}(S)$ is a uniquely determined element.
We show first that such an $f$ exists.
Define $f = \{s \in S \colon X_{s} \in F \}$.
From the fact that $F$ is completely prime and that the sets $X_{s}$ form a basis of open bisections for $\mathsf{G}(S))$
it follows that $f$ is non-empty.
Using Lemma~2.6, it is routine to verify that $f$ is a completely prime filter and by construction $F_{f} \subseteq F$.
Let $O \in F$.
Then $O$ can be written as a union of open bisections of the form $X_{s}$ for some $s$.
It follows that $O \in F_{f}$.

Now suppose that $F_{f} = F_{g}$ for completely prime filters $f$ and $g$ in $S$.
Let $s \in f$.
Then $f \in X_{s}$ and so by assumption $X_{s} \in F_{g}$ which gives $s \in g$.
It follows that $f \leq g$.
The reverse inclusion follows by symmetry.

It remains to show that $\eta$ is an open map.
Let $X_{s}$ be a basic open bisection in $\mathsf{G}(S)$.
Then $\eta (X_{s})$ consists of all $F_{f}$ where $f \in X_{s}$.
But this is precisely the set $\{F_{f} \colon X_{s} \in F_{f} \}$
which is the basic open set $X_{X_{s}}$.\end{proof}

\subsection{Adding the morphisms}

What we would like to do is prove that, for suitable definitions of morphisms,
the functor
$$\mathsf{G} \colon \mathbf{Inv}^{op} \rightarrow \mathbf{Etale}$$
is right adjoint to the functor
$$\mathsf{B} \colon \mathbf{Etale} \rightarrow \mathbf{Inv}^{ op}.$$
Since the idempotent pseudogroups are the frames and the \'etale groupoids in which every element is an identity are just the
topological spaces, this would generalize the classical adjunction between categories of these structures;
see Theorem~1, page~476 of \cite{MM} and Theorem~1.4, page~42 of \cite{J}.
The problem is in defining appropriate morphisms.
Frames have top elements preserved by frame morphisms but this is not true of general pseudogroups.
This means that the inverse images of completely prime filters might be empty.
We do, however, have the following.

\begin{lemma} Let $\theta \colon S \rightarrow T$ be a pseudogroup $\wedge$-morphism.
If $F$ is a completely prime filter in $T$ and $\theta^{-1}(F)$ is non-empty then it is a completely prime filter.
\end{lemma}
\begin{proof} Let $a,b \in \theta^{-1}(F)$.
Then $\theta (a),\theta (b) \in F$.
But $F$ is a filter in a pseudogroup and so $\theta (a) \wedge \theta (b) \in F$.
We have assume that $\theta$ is a $\wedge$-morphism and so $\theta (a \wedge b) = \theta (a) \wedge \theta (b)$.
Thus $a \wedge b \in \theta^{-1}(F)$.
It is clear that $\theta^{-1}(F)$ is closed upwards, and it is completely prime because $\theta$ is a pseudogroup morphism.
\end{proof}

A function $\theta \colon S \rightarrow T$ between pseudogroups will be called {\em callitic}
if it satisfies two conditions:
\begin{enumerate}

\item it is a $\wedge$-morphism of pseudogroups, and

\item for each completely prime filter $F$ in $T$,  we have that $F \cap \mbox{im}(\theta) \neq \emptyset$.

\end{enumerate}

\begin{lemma} Let $\theta \colon S \rightarrow T$ be a callitic morphism of pseudogroups.
Then
$$\theta^{-1} \colon \mathsf{G}(T) \rightarrow \mathsf{G}(S)$$
is a continuous covering functor.
\end{lemma}
\begin{proof} The assumption that $\theta$ is callitic simply ensures that for each completely prime filter $F$ the set
$\theta^{-1}(F)$ is non-empty.
It follows by Lemma~2.13, that $\theta^{-1} \colon \mathsf{G}(T) \rightarrow \mathsf{G}(S)$ is a well-defined function.
The bulk of the proof is taken up with showing that $\theta^{-1}$ is a functor.
Let $F$ be an identity completely prime filter in $T$.
Then $F$ contains idempotents by Lemma~2.4.
In particular, it must contain the top idempotent in the frame $E(T)$ by upward closure.
Since $\theta$ is a frame morphism when restricted to the semilattice of idempotents
it follows that $\theta^{-1}(F)$ contains the top element of $E(S)$.
Thus $\theta^{-1}(F)$ is a completely prime filter containing idempotents and so it is an identity in the groupoid.

We prove that if $F$ and $G$ are completely prime filters such that $F^{-1} \cdot F = G \cdot G^{-1}$ then
$$( \theta^{-1} (F) \theta^{-1} (G) )^{\uparrow} = \theta^{-1} ( (FG)^{\uparrow}  ).$$
We prove first that
$$\theta^{-1} (F) \theta^{-1} (G) \subseteq \theta^{-1} (FG).$$
Let $s \in \theta^{-1} (F) \theta^{-1} (G)$.
Then $s = ab$ where $a \in \theta^{-1} (F)$ and $b \in \theta^{-1} (G)$.
Thus $\theta (s) = \theta (a) \theta (b) \in FG$.
It follows that $s \in \theta^{-1} (FG)$.
Observe that $\theta^{-1} (X)^{\uparrow} \subseteq \theta^{-1} (X^{\uparrow})$.
It follows that
$$( \theta^{-1} (F) \theta^{-1} (G) )^{\uparrow} \subseteq \theta^{-1} ( (FG)^{\uparrow}  ).$$
We now prove the reverse inclusion.
Let $s \in \theta^{-1} ( (FG)^{\uparrow} ).$
Then $\theta (s) \in F \cdot G$ and so
$fg \leq \theta(s)$ for some $f \in F$ and $g \in G$.
The map $\theta$ is assumed callitic and so there exists $v \in S$ such that $\theta (v) \in G$.
Consider the product $\theta (s) \theta (v)^{-1}$.
Since $\theta (s) \in F \cdot G$ and $\theta (v)^{-1} \in G^{-1}$ we have that
$\theta (s)\theta (v)^{-1} \in F \cdot G \cdot G^{-1} = F \cdot F^{-1} \cdot F = F$.
Thus $\theta (sv^{-1}) \in F$,  and we were given $\theta (v) \in G$, and clearly $(sv^{-1})v \leq s$.
Put $a = sv^{-1}$ and $b = v$.
Then $ab \leq S$ where $\theta (a) \in F$ and $\theta (b) \in G$.
It follows that $s \in (\theta^{-1} (F) \theta^{-1} (G) )^{\uparrow}$.

We may now show that $\theta^{-1}$ is a functor.
Let $F$ be any completely prime filter.
Observe that $\theta^{-1}(F)^{-1} = \theta^{-1} (F^{-1})$.
We have that
$$(\theta^{-1} (F^{-1}) \theta^{-1} (F))^{\uparrow}
=
(\theta^{-1} (F)^{-1} \theta^{-1} (F))^{\uparrow}
=
\mathbf{d} ( \theta^{-1} (F))$$
and
$$\theta^{-1} ( (F^{-1}F)^{\uparrow} ) = \theta^{-1} (\mathbf{d} (F)).$$
Hence
$$\theta^{-1} (\mathbf{d} (F))
=
\mathbf{d} ( \theta^{-1} (F)).$$
A dual result also holds and so $\theta^{-1}$ preserves the domain and codomain operations.
Suppose that $\mathbf{d}(F) = \mathbf{r} (G)$ so that $F \cdot G$ is defined.
By our calculation above $\mathbf{d} (\theta^{-1}(F)) = \mathbf{r} (\theta^{-1} (G))$
and so the product $\theta^{-1} (F) \cdot \theta^{-1} (G)$ is defined.
By our main result above we have that
$$\theta^{-1} (F \cdot G) = \theta^{-1} (F) \cdot \theta^{-1} (G),$$
as required.

The proof that $\theta^{-1}$ is a covering functor follows the same lines as the proof of Proposition~2.15 of \cite{Law3}.
It remains to show that it is continuous.
A basic open set of $\mathsf{G}(S)$ has the form $X_{s}$ for some $s \in S$.
It is simple to check that this is pulled back to the set $X_{\theta (s)}$.
\end{proof}

We also have the following whose proof is straightforward.

\begin{lemma}
The composition of callitic maps is callitic.
\end{lemma}

Identity maps are callitic and so pseudogroups and their callitic maps form a category.

A morphism $\theta \colon S \rightarrow T$ is called {\em hypercallitic} if for each $t \in T$
we have that
$$t = \bigvee_{s \in S} (t \wedge \theta (s)).$$
Observe that since $t \wedge \theta (s) \leq t$ the righthand-side is well-defined.
Frame morphisms are always hypercallitic because tops are mapped to tops.

\begin{lemma}
Hypercallitic maps are callitic.
\end{lemma}
\begin{proof} Let $\theta \colon S \rightarrow T$ be a hypercallitic map.
Let $F$ be a completely prime filter in $T$ and let $t \in T$.
By assumption, we may write
$t = \bigvee_{s \in S} (t \wedge \theta (s))$.
But $F$ is completely prime and so $t \wedge \theta (s) \in F$ for some $s \in S$.
It then follows that $\theta (s) \in F$, as required.
\end{proof}

The proof of the following is immediate from the definition of hypercallitic map.

\begin{lemma} Let $\theta \colon S \rightarrow T$ be a pseudogroup morphism such that each element of $T$ is the join
of a compatible subset of the image of $\theta$.
Then $\theta$ is hypercallitic.
\end{lemma}

The following is an important consequence.

\begin{corollary}
The map $\varepsilon \colon S \rightarrow \mathsf{B} (\mathsf{G} (S))$ is hypercallitic.
\end{corollary}

The following result explains the real reason for our interest in hypercallitic maps.

\begin{lemma} Let $\theta \colon G \rightarrow H$ be a continuous covering functor between two \'etale groupoids.
Then $\theta^{-1} \colon \mathsf{B}(H) \rightarrow \mathsf{B}(G)$ is hypercallitic.
\end{lemma}
\begin{proof}
The proof that we have a $\wedge$-morphism of pseudogroups basically follows the proof of Proposition~2.17 of \cite{Law3}.
It remains to show that $\theta^{-1}$ is hypercallitic.
From Exercise~I.1.8, Question~6 of \cite{R1}, we learn that if $G$ is a topological groupoid in which $G_{o}$ is an open subspace
then $G$ has a basis consisting of open bisections.
Let $B$ be a non-empty open bisection in $G$ and let $g \in B$.
Then $\theta (g) \in H$.
Clearly $H$ is an open set containing $\theta (g)$ but not a bisection.
However, since $H$ is \'etale, it follows that $H$ is a union of open bisections and so $\theta (g) \in C_{g}$ an open bisection $C_{g}$ in $H$.
Since $\theta$ is continuous $g \in \theta^{-1}(C_{g})$ is open and because $\theta$ is a covering functor $\theta^{-1} (C_{g})$ is a bisection.
Thus $g \in B \cap \theta^{-1}(C_{g})$ an open bisection in $G$.
It follows that we may write
$$B = \bigcup_{g \in B} (B \cap \theta^{-1} (C_{g})).$$
\end{proof}

There is another way of seeing the importance of hypercallitic maps using a construction of Resende \cite{R1,R2}.
Let $S$ be a pseudogroup.
Define $\mathcal{L}^{\vee} (S)$ to be the set of all order ideals of $S$ that are closed under compatible joins.
This is called the {\em enveloping quantale} of $S$.
It is, in particular, a frame with top element $S$.
Let $\theta \colon S \rightarrow T$ be a morphism of pseudogroups.
Then we may define a function $\bar{\theta} \colon \mathcal{L}^{\vee} (S) \rightarrow \mathcal{L}^{\vee} (T)$
by $\bar{\theta} (A) = [A^{\downarrow}]^{\vee}$
which means the downward closure of $A$ followed by the closure under compatible joins.
The following lemma grew out of conversations with Resende.

\begin{lemma}
The map $\bar{\theta}$ defined above is a frame map if and only if $\theta$ is hypercallitic.
\end{lemma}
\begin{proof} The map $\bar{\theta}$ is a frame map if and only if $\bar{\theta}(S) = T$.
That is if and only if $[\theta(S)^{\downarrow}]^{\vee} = T$.
This means that for each $t \in T$ we may find $t_{i} \in \theta (S)^{\downarrow}$ such that $t = \bigvee t_{i}$.
But $t_{i} \in \theta (S)^{\downarrow}$ means that $t_{i} \leq \theta (s_{i})$ for some $s_{i} \in S$.
In particular we have that $t_{i} = t_{i} \wedge \theta (s_{i})$.
Thus $t = \bigvee t_{i} \wedge \theta (s_{i})$.
It is now easy to show that $t = \bigvee_{s \in S} t \wedge \theta (s)$, as required.
\end{proof}

The following lemma answers a question that the reader might ask themselves
when reading the statement and proof of the Adjunction Theorem below.

\begin{lemma} Let $\theta \colon S \rightarrow T$ be a callitic morphism of pseudogroups where $T$ is spatial.
Then $\theta$ is in fact hypercallitic.
\end{lemma}
\begin{proof} Let $t \in T$ where $t \neq 0$.
Then the set of all completely prime filters $X_{t}$ containing $t$ cannot be empty
because then $X_{t} = X_{0}$ would imply that $t = 0$.
We prove that
$$X_{t} = X_{\vee_{s \in S} (t \wedge \theta (s))}$$
from which the result follows by the spatiality of $T$.
Let $F \in X_{t}$.
Since $\theta$ is callitic there exists $\theta (s) \in F$ for some $s \in S$.
Thus $t \wedge \theta (s) \in F$.
It follows that $F \in  X_{\vee_{s \in S} (t \wedge \theta (s))}$.
Conversely, if $F \in  X_{\vee_{s \in S} (t \wedge \theta (s))}$
then clearly $F \in X_{t}$.
\end{proof}

We may now define our categories.
We define the category $\mathbf{Inv}$ to have pseudogroups as objects and callitic pseudogroup maps as morphisms.
In what follows we work with the dual category $\mathbf{Inv}^{op}$ for convenience.
We define $\mathbf{Etale}$ to have \'etale groupoids as objects and continuous covering functors as morphisms.
There is a functor $\mathsf{G} \colon \mathbf{Inv}^{op} \rightarrow \mathbf{Etale}$
that takes a pseudogroup $S$ to the \'etale groupoid $\mathsf{G}(S)$ by Proposition~2.8
and takes the callitic map $\theta \colon S \rightarrow T$ to the continuous covering functor
$\theta^{-1} \colon \mathsf{G}(T) \rightarrow \mathsf{G}(S)$ by Lemma~2.14.
There is a functor $\mathsf{B} \colon \mathbf{Etale} \rightarrow  \mathbf{Inv}^{op}$
that takes an \'etale groupoid $G$ to a pseudogroup $\mathsf{B}(G)$ by Proposition~2.1
and takes the continuous covering functor $\theta \colon G \rightarrow H$ to the
callitic morphism of pseudogroups $\theta^{-1} \colon \mathsf{B}(H) \rightarrow \mathsf{B}(G)$ by Lemma~2.19.
Our main theorem is the following.

\begin{theorem}[Adjunction]
The functor
$$\mathsf{G} \colon \mathbf{Inv}^{op} \rightarrow \mathbf{Etale}$$
is right adjoint to the functor
$$\mathsf{B} \colon \mathbf{Etale} \rightarrow \mathbf{Inv}^{ op}.$$
\end{theorem}
\begin{proof}
Given a continuous covering functor $\alpha \colon G \rightarrow \mathsf{G}(S)$,
we may construct the callitic map $\alpha^{-1} \varepsilon \colon S \rightarrow \mathsf{B}(G)$.
This is just the map
$$s \mapsto \alpha^{-1}(X_{s}).$$
Given a callitic map $\beta \colon S \rightarrow \mathsf{B}(G)$ we may construct the
continuous covering functor $\beta^{-1} \eta \colon G \rightarrow \mathsf{G}(S)$.
This is just the map
$$g \mapsto \beta^{-1} (F_{g}).$$
We shall that these two constructions are mutually inverse.

Let $\beta \colon S \rightarrow \mathsf{B}(G)$ be a callitic morphism of pseudogroups.
Define $\alpha (g) = \beta^{-1}(G_{g})$.
Then the map we get from $S$ to $\mathsf{B}(G)$
after applying the above procedures twice is the map $s \mapsto \alpha^{-1} (X_{s})$.
We have that
$g \in \alpha^{-1} (X_{s})
\Leftrightarrow \alpha (g) \in X_{s}
\Leftrightarrow \beta^{-1} (F_{g}) \in X_{s}
\Leftrightarrow s \in \beta^{-1} (F_{g})
\Leftrightarrow \beta (s) \in F_{g}
\Leftrightarrow g \in \beta (s)$.
It follows that $\beta (s) = \alpha^{-1}(X_{s})$, as required.

Let $\alpha \colon G \rightarrow \mathsf{G}(S)$ be a continuous covering functor.
Define $\beta (s) = \alpha^{-1} (X_{s})$.
Then then map we get from $G$ to $\mathsf{G}(S)$
after applying the above procedures twice is the map $g \mapsto \beta^{-1} (F_{g})$.
We have that
$s \in \beta^{-1} (F_{g})
\Leftrightarrow \beta (s) \in F_{g}
\Leftrightarrow \alpha^{-1}(X_{s}) \in F_{g}
 \Leftrightarrow g \in \alpha^{-1} (X_{s})
\Leftrightarrow \alpha (g) \in X_{s}
\Leftrightarrow s \in \alpha (g)$, as required.

Naturality is straightfoward to prove here and so we have an adjunction.
\end{proof}

The map $\eta \colon G \rightarrow \mathsf{G}(\mathsf{B} (G))$ of Proposition~2.11 is the
{\em unit} of the adjunction.
The map $\varepsilon \colon S \rightarrow \mathsf{B} (\mathsf{G} (S))$ of Proposition~2.9
is the {\em counit} of the adjunction.

Let $\mathbf{Inv}_{sp}$ be the category of spatial pseudogroups and callitic pseudogroup morphisms.
Let $\mathbf{Etale}_{so}$ be the category of sober \'etale groupoids and continuous covering functors.
From the above theorem, Proposition~2.12 and general category theory we have proved the following.

\begin{theorem}[Duality]
The category  $\mathbf{Inv}_{sp}^{op}$ is equivalent to the category  $\mathbf{Etale}_{so}$.
\end{theorem}

We have already remarked in the Introduction that the above theorem, at the level of objects, was first proved in \cite{MR}.

\section{Non-commutative Stone dualities}

The main goal of this section is to derive the duality theorems for boolean inverse semigroups proved in \cite{Law3,Law4} from Theorem~2.23,
however we shall begin by proving a duality theorem for a wider class of inverse semigroups.
An inverse semigroup $S$ is said to be {\em distributive} if it has joins of all {\em finite}
compatible subsets and multiplication distributes over the finite joins that exist.
A {\em morphism} of distributive inverse semigroups is a semigroup morphism
that preserves the joins of finite compatible subsets.
It is worth noting that distributive inverse monoids are of independent interest \cite{KM} 
and that they have a distributive lattice of idempotents.

A filter $F$ in a distributive inverse semigroup is said to be {\em prime} if
$a \vee b \in F$ implies that $a \in F$ or $b \in F$.
The class of prime filters in distributive inverse semigroups will play a key role in this section.

\subsection{$\mbox{Idl}$-completions}

The results of Section~2 were proved only for pseudogroups.
In order to obtain dualities for a wider class of inverse semigroups we shall need a way of completing inverse semigroups to pseudogroups.
The basis of this is a construction due to Boris Schein \cite{S,Law2}.
Let $S$ be an inverse semigroup.
Define $C(S)$ to be the set of all compatible order ideals of $S$ with subset multiplication as the operation.
Then $C(S)$ is a pseudogroup and the map $\iota \colon S \rightarrow C(S)$, given by $s \mapsto s^{\downarrow}$, is a semigroup homomorphism.
In addition, $C$ is left adjoint to the forgetful functor from the category
of pseudogroups and pseudogroup morphisms to the category of inverse semigroups and semigroup homomorphisms.
If $X$ and $Y$ are compatible subsets of $S$ then $X^{\downarrow}$ and $Y^{\downarrow}$ are compatible order ideals and $X^{\downarrow}Y^{\downarrow} = (XY)^{\downarrow}$.
Given a semigroup homomorphism $\theta \colon S \rightarrow T$, we may therefore define a function $\Theta \colon C(S) \rightarrow C(T)$
by $\Theta (A) = \theta (A)^{\downarrow}$ which is a semigroup homomorphism.
Observe that if $A = \bigcup_{i} A_{i}$ then $\Theta (A) = \bigcup_{i} \Theta (A_{i})$.
Thus $\Theta$ is a morphism of pseudogroups and gives an explicit description of the induced 
functor from the category of inverse semigroups and semigroup homomorphisms
to the category of pseudogroups and pseudogroup morphisms. 

When $S$ is a distributive inverse semigroup, 
we shall work with a cut down version of $C(S)$ which uses a finitary version of a construction used by Rinow \cite{R}.
An element $A$ of $C(S)$ is said to be {\em $\vee$-closed} if it is closed under joins of its finite subsets;
such subsets are necessarily compatible since $A$ is a compatible order ideal.
We denote by
 $\mbox{Idl} (S)$
the set of all $\vee$-closed elements of $C(S)$.
Observe that if $A$ is an order ideal
then it becomes a $\vee$-closed order ideal when we
include all the joins of finite subsets of $A$.
We denote this set by $A^{\vee}$ and call it the {\em $\vee$-closure of $A$}.
If $A$ is a compatible order ideal then $A^{\vee}$ is a $\vee$-closed compatible order ideal as can easily be verified.
The operation $A \mapsto A^{\vee}$ satisfies the following conditions:
\begin{description}

\item[{\rm (Cl1)}] $A \subseteq A^{\vee}$.

\item[{\rm (Cl2)}] If $A \subseteq B$ then $A^{\vee} \subseteq B^{\vee}$.

\item[{\rm (Cl3)}] $A^{\vee} = (A^{\vee})^{\vee}$.

\item[{\rm (Cl4)}] $A^{\vee}B^{\vee} = (AB)^{\vee}$.

\end{description}
The proofs are all straightforward except for (Cl4) which needs some comment.
The proof of the inclusion $A^{\vee}B^{\vee} \subseteq (AB)^{\vee}$ follows from the fact that multiplication distributes over compatible joins.
The proof of the reverse inclusion uses the fact that $A^{\vee}B^{\vee}$ is an order ideal. 

\begin{proposition} Let $S$ be a distributive inverse semigroup.
Then $\mbox{\rm Idl} (S)$ is a pseudogroup and the homomorphism $\iota \colon S \rightarrow \mbox{\rm Idl} (S)$
given by $s \mapsto s^{\downarrow}$
preserves binary joins of compatible pairs of elements.

In addition, $\mbox{\rm Idl}$ is left adjoint to the forgetful functor from the category
of pseudogroups and pseudogroup morphisms to the category of distributive inverse semigroups and their morphisms.

If $\theta \colon S \rightarrow T$ is a morphism of distributive inverse semigroups then
$\Theta \colon \mbox{\rm Idl}(S) \rightarrow \mbox{\rm Idl}(T)$ defined by $\Theta (A) = [\theta (A)^{\downarrow}]^{\vee}$
is the induced morphism of pseudogroups.

\end{proposition}
\begin{proof} It is clear that $\mbox{\rm Idl} (S)$ is closed under inverses, and it is closed under multiplication by (C4) above.
It follows that $\mbox{\rm Idl} (S)$ is an inverse semigroup.
Observe that $\mbox{\rm Idl} (S)$ is actually an inverse subsemigroup of $C(S)$ and so the natural partial orders agree.
A compatible set of elements in  $\mbox{\rm Idl} (S)$ has a join in $C(S)$ and this can be reflected into  $\mbox{\rm Idl} (S)$
using the operation $A \mapsto A^{\vee}$.
Thus every compatible subset of $\mbox{\rm Idl} (S)$ has a join.
It is now easy to prove using the properties of the $\vee$-closure operation that $\mbox{\rm Idl} (S)$ is infinitely distributive.
Observe that the idempotents of $\mbox{\rm Idl} (S)$ are the $\vee$-closed order ideals in the meet semilattice $E(S)$
and that there is a maximum idempotent $E(S)$ and so the semilattice of idempotents of $\mbox{\rm Idl} (S)$ forms a frame.
In the monoid case, this can be deduced from Corollary in Section~2.11 of \cite{J}.
It follows that $\mbox{\rm Idl} (S)$ is a pseudogroup.

The map $\iota \colon S \rightarrow \mbox{\rm Idl} (S)$ is a homomorphism.
Suppose that $c = a \vee b$ in $S$.
Clearly $\iota (a), \iota (b) \subseteq \iota (c)$.
But any {\em $\vee$-closed} element of $C(S)$ that contains $a$ and $b$ must contain $c$.
It follows that $\iota (c) = \iota (a) \vee \iota (b)$.

Let $\alpha \colon S \rightarrow T$ be a homomorphism to a pseudogroup that preserves finite compatible joins.
Then there is a unique morphism of pseudogroups $\bar{\alpha} \colon \mbox{\rm Idl} (S) \rightarrow T$
such that $\bar{\alpha} \iota = \alpha$ defined by $\bar{\alpha} (A) = \bigvee A$.

The proof of the last claim is routine.
\end{proof}

We call the pseudogroup $\mbox{\rm Idl}(S)$ the {\em $\mbox{\rm Idl}$-completion} of $S$.

Prime filters in distributive inverse semigroups and completely prime filters in their $\mbox{Idl}$-completions are related as follows.

\begin{lemma} \label{prime-vs-completelyprime} Let $S$ be a distributive inverse semigroup.

If $P$ is a prime filter in $S$ define
$$P^{u} = \{ A \in  \mbox{\rm Idl} (S) \colon  A \cap P \neq \emptyset \}.$$
Then $P^{u}$ is a completely prime filter in $\mbox{\rm Idl} (S)$.

If $F$ is a completely prime filter in $\mbox{\rm Idl} (S)$
define
$$F^{d} = \{s \in S \colon s^{\downarrow} \in F \}.$$
Then $F^{d}$ is a prime filter in $S$.

The above two operations are mutually inverse and set up an order
isomorphism between the poset of prime filters on $S$ and the poset of
completely prime filters on  $\mbox{\rm Idl} (S)$.
\end{lemma}
\begin{proof}
Clearly the set $P^{u}$ is closed upwards.
Let $A,B \in P^{u}$.
Then $s \in A \cap P$ and $t \in B \cap P$.
But $P$ is a filter and so there exists $p \in P$ such that $p \leq s,t$.
But then $p \in A \cap B$ and so $P^{u}$ is closed under binary intersections.
Suppose that $\bigvee A_{i} \in P^{u}$.
Thus there exists $p \in P$ such that $p \in \bigvee_{i} A_{i}$.
By definition $p = \vee_{j=1}^{m} a_{j}$ for some finite set of elements $a_{j}$ in the $A_{i}$.
But $P$ is a prime filter and so $a_{k} \in P$ for some $k$.
Thus one of the $A_{i}$, the one containing $a_{k}$, belongs to $P^{u}$ as required.
Thus $P^{u}$ is a completely prime filter.

We now show that $F^{d}$ is a prime filter.
Let $s,t \in F^{d}$.
Then $s^{\downarrow},t^{\downarrow} \in F$.
Thus $A = s^{\downarrow} \cap t^{\downarrow} \in F$.
Now $A = \bigvee_{a \in A} a^{\downarrow} \in F$ and so $a^{\downarrow} \in F$ for some $a \in A$.
Thus $a \in F^{d}$ and $a \leq s,t$.
It is clear that $F^{d}$ is closed upwards.
It remains to show that $F^{d}$ is a prime filter.
Let $s \vee t \in F^{d}$.
Then $(s \vee t)^{\downarrow} \in F$.
But $(s \vee t)^{\downarrow} = s^{\downarrow} \vee t^{\downarrow} \in F$.
It follows that $s^{\downarrow} \in F$ or $t^{\downarrow} \in F$.
Thus $s \in F$ or $t \in F$.

It is now routine to check that these two operations are mutually inverse and order-preserving.
\end{proof}

\subsection{Coherent pseudogroups}

We shall now characterize the pseudogroups that arise as  $\mbox{\rm Idl}$-completions.
To do this we need the following definition.
We say that the compatible subset $X$ of a pseudogroup $S$ is a {\em covering} of the element $a$ if $a \leq \bigvee X$.
An element $a \in S$ in a pseudogroup $S$ is said to be {\em finite} if for any compatible subset $X \subseteq S$ such that
$a \leq \bigvee X$ there exists a finite subset $Y$ of $X$ such that $a \leq \bigvee Y$.
In other words, every covering has a finite subcovering.
In the case of the frames of open sets of a topological space the finite elements are just the compact ones.
It is worth noting that the inequalities can be replaced by equalities; see page~63 of \cite{J}.
We denote the set of finite elements of a pseudogroup $S$ by $K(S)$.

\begin{lemma} Let $S$ be a pseudogroup.
\begin{enumerate}

\item If $a$ is finite then $a^{-1}$ is finite.

\item If $a$ is any element and $e$ is a finite idempotent such that $e \leq a^{-1}a$ then $ae$ is finite.

\item If $a$ is finite then $a^{-1}a$ is finite, and dually.

\item If $a$ and $b$ are finite and $a^{-1}a = bb^{-1}$ then $ab$ is finite.

\end{enumerate}
\end{lemma}
\begin{proof} (1) Straightforward.

(2) Let $a$ be any element and $e$ a finite idempotent $e \leq a^{-1}a$.
We prove that $ae$ is finite.
Suppose that $ae \leq \bigvee x_{i}$.
Then $e = a^{-1}ae \leq \bigvee a^{-1}x_{i}$.
But $e$ is finite and so $e \leq \bigvee_{i=1}^{m} a^{-1}x_{i}$.
Thus $ae \leq \bigvee_{i=1}^{m} aa^{-1}x_{i} \leq \bigvee_{i=1}^{m} x_{i}$.
It follows that $ae$ is finite.

(3) Let $a$ be any finite element.
Suppose that $a^{-1}a \leq \bigvee x_{i}$.
Then $a \leq \bigvee ax_{i}$.
Thus $a \leq \bigvee_{i=1}^{m} ax_{i}$ since $a$ is finite.
Hence $a^{-1}a \bigvee_{i=1}^{m} a^{-1}ax_{i} \leq  \bigvee_{i=1}^{m} x_{i}$.
Thus $a^{-1}a$ is finite.

(4) Let $a$ and $b$ be any finite elements where $a^{-1}a = bb^{-1}$.
We prove that $ab$ is finite.
Suppose that $ab \leq \bigvee x_{i}$.
Then $a^{-1}abb^{-1} \leq \bigvee a^{-1}x_{i}b^{-1}$.
By assumption $a^{-1}abb^{-1}$ is a finite idempotent.
Thus we may write $a^{-1}abb^{-1} \leq \bigvee_{i=1}^{m} a^{-1}x_{i}b^{-1}$.
Hence $ab  \leq \bigvee_{i=1}^{m} aa^{-1}x_{i}b^{-1}b \leq \bigvee_{i=1}^{m} x_{i}$.
It follows that $ab$ is finite.
\end{proof}

The above lemma tells us that the finite elements in a pseudogroup always form an ordered groupoid \cite{Law2}.

\begin{lemma} Let $S$ be a pseudogroup.
\begin{enumerate}

\item The finite elements of $S$ form an inverse subsemigroup if and only if the finite idempotents form a subsemigroup.

\item If the finite elements form an inverse subsemigroup they form a distributive inverse semigroup.

\item Every element of $S$ is a join of finite elements if and only if every idempotent is a join of finite idempotents.

\end{enumerate}
\end{lemma}
\begin{proof} (1) Let $a$ and $b$ be arbitrary finite elements.
Then $a^{-1}a$ and $bb^{-1}$ are both finite and so $e = a^{-1}abb^{-1}$ is finite
and consequently $ab = (ae)(eb)$ is finite.

(2) Observe that if $a$ and $b$ are compatible finite elements then $a \vee b$ is finite.

(3) Only one direction needs proving.
Suppose that every idempotent is a join of finite idempotents.
Let $a$ be an arbitrary element.
By assumption we may write $a^{-1}a = \bigvee e_{i}$ where $e_{i} \leq a^{-1}a$ and are finite.
Thus $a = \bigvee ae_{i}$ and by Lemma~3.18(1) the elements $ae_{i}$ are all finite.
\end{proof}

A pseudogroup $S$ is said to be {\em coherent} if the set of its finite elements forms a distributive inverse subsemigroup
and if every element of $S$ is a join of finite elements.
  
\begin{proposition} A pseudogroup $S$ is coherent if and only if there exists a distributive inverse semigroup $T$ such
that $S$ is isomorphic to $\mbox{\rm Idl} (T)$.  
In fact, any coherent pseudogroup $S$ is canonically isomorphic to $\mbox{\rm Idl} (K(S))$.
\end{proposition}
\begin{proof} 
Let $T$ be a distributive inverse semigroup.
We prove first  that the finite elements of $\mbox{\rm Idl} (T)$ are precisely the elements of the form $t^{\downarrow}$ where $t \in T$.
Observe that $t^{\downarrow} = \overline{t^{\downarrow}}$.
Let $t^{\downarrow} = \bigvee A_{i}$.
Then $t$ is in the $\vee$-closure of  $\bigcup A_{i}$.
Thus there is a finite set of elements $a_{1}, \ldots, a_{m} \in \bigcup A_{i}$ such that $t = \vee a_{j}$.
But this implies that $t^{\downarrow}$ is the join of only finitely many of the $A_{i}$.
Thus $t^{\downarrow}$ is finite.
Suppose now that $A$ is a finite element.
We have that $A = \bigvee_{a \in A} a^{\downarrow}$.
By assumption there are finitely many elements $a_{1}, \ldots, a_{m} \in A$ such that
$A = \bigvee a_{i}^{\downarrow}$.
But if $a = \vee a_{i}$ then $A = a^{\downarrow}$, as required.
Clearly, every element of  $\mbox{\rm Idl} (T)$ is a compatible join of finite elements.
It follows that $\mbox{\rm Idl} (T)$ is coherent and that its finite elements form a distributive
inverse semigroup isomorphic to $T$. 

Now suppose that $S$ is a coherent pseudogroup.
Put $T = K(S)$, a distributive inverse semigroup by assumption.
Define $\theta \colon \mbox{\rm Idl} (T) \rightarrow S$ by $\theta (A) = \bigvee A$.
This is surjective since every element of $S$ is the join of finite elements.
Suppose that $\theta (A) = \theta (B)$.
Let $a \in A$.
Then $a \leq \bigvee A$.
Thus $a \leq \bigvee B$.
But $a$ is a finite element and so there is a finite subset $b_{1}, \ldots, b_{m} \in B$ such that
$a \leq \vee_{i} b_{i}$.
But $B$ is  $\vee$-closed and so $\bigvee_{i} b_{i} \in B$ that implies $a \in B$.
We have proved that $A \subseteq B$.
The reverse inclusion is proved similarly.
It follows that $\theta$ is a bijection.
It is clearly a homomorphism.
We have proved that $K(\mbox{\rm Idl} (T))$ is isomorphic to $S$.
\end{proof}

A pseudogroup morphism between coherent pseudogroups is said to be {\em coherent} if it preserves finite elements.
The proof of the following is now straightforward. In fact, the previous lemma just gives the object-part of the statement.

\begin{lemma}
The category of distributive inverse semigroups and their morphisms is equivalent to the category of coherent pseudogroups
and coherent pseudogroup morphisms.
\end{lemma}

Let $\theta \colon S \rightarrow T$ be a morphism of distributive inverse semigroups and let
$\Theta \colon \mbox{Idl}(S) \rightarrow \mbox{Idl}(T)$ be the induced morphism of pseudogroups.
Then $\Theta$ is a pseudogroup $\wedge$-morphism if and only if $\theta$ satisfies the following condition:
\begin{description}

\item[{\rm (DC1)}] If $t \leq \theta (s_{1}), \theta (s_{2})$ then there exists $s \leq s_{1},s_{2}$ such that $t \leq \theta (s)$.

\end{description}
Consider now the following condition on $\theta$:
\begin{description} 

\item[{\rm (DC2)}] For each prime filter $P$ in $T$ the inverse image $\theta^{-1}(P)$ is non-empty

\end{description}
Observe that assuming (DC1), condition (DC2) implies that the inverse images of prime filters under $\theta$ are prime filters.
We claim that $\Theta$ is callitic if and only if $\theta$ satisfies (DC1) and (DC2).
This essentially follows by Proposition~3.5 and the fact that $\mbox{Idl}(S)$ and $\mbox{Idl}(T)$ are coherent.
We say that a morphism of distributive inverse semigroups is {\em callitic} if its satisfies (DC1) and (DC2).
We may now refine Lemma~3.6 as follows.

\begin{proposition}
The category of distributive inverse semigroups and their callitic morphisms is equivalent to the category of coherent pseudogroups
and callitic coherent pseudogroup morphisms.
\end{proposition}

\subsection{Non-commutative Stone duality for distributive inverse semigroups}

In distributive inverse semigroups, we assume that finite compatible joins exist but we make no assumption about the existence of meets.
However, for those finite subsets where meets do exist the following lemma shows that they behave as expected with respect to joins.
It is just the finitary case of \cite{R3}.

\begin{lemma} Let $S$ be a distributive inverse semigroup.
Suppose that $a \vee b$ exists and that $c \wedge (a \vee b)$ exists.
Then $c \wedge a$ and $c \wedge b$ both exist, the join $(c \wedge a) \vee (c \wedge b)$ exists and
$$c \wedge (a \vee b)
=
(c \wedge a) \vee (c \wedge b).$$
\end{lemma}

It can easily be verified that the union of a totally ordered set of $\vee$-closed order ideals
of an inverse semigroup is again a $\vee$-closed order ideal.
The proof of the following result now follows from Zorn's Lemma.

\begin{lemma} Let $S$ be an inverse semigroup.
Let $I$ be a $\vee$-closed order ideal of $S$ and let $F$ be a filter disjoint from $I$.
Then there is a $\vee$-closed order ideal $J$ maximal with respect to the two conditions:
(1) $I \subseteq J$ and (2) $J \cap F = \emptyset$.
\end{lemma}

An order ideal $P$ of an inverse semigroup $S$ is said to be {\em prime}
if $a^{\downarrow} \cap b^{\downarrow} \subseteq P$ implies that either $a \in P$ or $b \in P$.

\begin{lemma} Let $S$ be an inverse semigroup.
Then a subset $F$ is a prime filter if and only if $S \setminus F$ is a $\vee$-closed prime order ideal.
\end{lemma}
\begin{proof}
Suppose that $F$ is a prime filter.
We prove that $P = S \setminus F$ is a $\vee$-closed order ideal.
Let $a \in P$ and $b \leq a$.
Suppose that $b \notin P$.
Then $b \in F$ and so $a \in F$, which is a contradiction.
Thus $P$ is an order ideal.
Suppose that $a^{\downarrow} \cap b^{\downarrow} \subseteq P$ and that $a,b \in F$.
Then since $F$ is a filter there exists $c \in F$ such that $c \leq a,b$.
But $c \in P$ which is a contradiction.
Finally, suppose that $a,b \in P$ and that $a$ and $b$ are compatible.
If $a \vee b \in F$ then either $a$ or $b$ is in $F$.
It follows that $a \vee b \in P$ and so $P$ is a $\vee$-closed prime ideal.

Conversely, suppose that $P$ is a $\vee$-closed prime ideal.
We prove that $F = S \setminus P$ is a prime filter.
Let $a \in F$ and $a \leq b$.
If $b \in P$ then $a \in P$ and so $b \in F$.
Let $a,b \in F$.
Then if $a^{\downarrow} \cap b^{\downarrow} \subseteq P$ then either $a$ or $b$ is in $P$.
It follows that there must exist $c \leq a,b$ such that $c \in F$.
Finally, suppose that $a \vee b \in F$.
If $a,b \in P$ then $a \vee b \in P$ so at least one of $a$ or $b$ belongs to $F$.
Thus $F$ is a prime filter.
\end{proof}

\begin{lemma} Let $S$ be a distributive inverse semigroup.
\begin{enumerate}

\item Let $F$ be a filter in $S$ and let $P$ be a $\vee$-closed order ideal of $S$ maximal amongst all $\vee$-closed order ideals disjoint from $F$.
Then $P$ is a prime $\vee$-closed order ideal.

\item Let $a,b \in S$ be such that $b \nleq a$.
Then there exists a prime filter that contains $b$ and omits $a$. 

\end{enumerate}
\end{lemma}
\begin{proof} (1) Assume that $a^{\downarrow} \cap b^{\downarrow} \subseteq P$.
Define
$$P_{1} = [P \cup \{ a \}^{\downarrow}]^{\vee}
\text{ and }
P_{2} = [P \cup \{ b \}^{\downarrow}]^{\vee}.$$
Both are well-defined $\vee$-closed order ideals that contain $P$.
Assume, for the sake of argument, that both intersect the filter $F$ in the elements $f_{1}$ and $f_{2}$ respectively.
We may write
$$f_{1} = p_{1} \vee x_{1}
\text{ and }
f_{2} = p_{2} \vee y_{1}$$
where $p_{1},p_{2} \in P$ and $x_{1} \leq a$ and $y_{1} \leq b$.
Since $F$ is a filter there is an element $f \in F$ such that $f \leq f_{1},f_{2}$.
Thus we may write
$$f = (p_{1} \vee x_{1}) \mathbf{d}(f)
\text{ and }
f = (p_{2} \vee y_{1}) \mathbf{d}(f).$$
By distributivity
$$f = p_{1} \mathbf{d}(f) \vee x_{1}\mathbf{d}(f)
\text{ and }
f = p_{2} \mathbf{d}(f)  \vee y_{1}\mathbf{d}(f).$$
Now $f = f \wedge f$.
Thus by Lemma~3.8, we have that
$$f =
(p_{1} \mathbf{d}(f) \wedge p_{2} \mathbf{d}(f))
\vee
(p_{1}\mathbf{d}(f) \wedge y_{1} \mathbf{d}(f))
\vee
(x_{1}\mathbf{d}(f) \wedge p_{2} \mathbf{d}(f))
\vee
(x_{1} \mathbf{d}(f) \wedge y_{1} \mathbf{d}(f)).
$$
Each term belongs to $P$, the final term by assumption.
Hence $f \in P$ which is a contradiction.
Thus either $P_{1}$ or $P_{2}$ is disjoint from $F$.
Without loss of generality we may assume that $P_{1}$ is disjoint from $F$.
But then we must have that $P_{1} = P$ and so $a \in P$.
It follows that $P$ is a prime $\vee$-closed order ideal.

(2) Consider the filter $b^{\uparrow}$ and the order ideal $a^{\downarrow}$ which is clearly a $\vee$-closed order ideal.
By assumption, $b^{\uparrow} \cap a^{\downarrow} = \emptyset$.
By Lemma~3.9, we may find a $\vee$-closed order ideal $J$ such that $a^{\downarrow} \subseteq J$ and $J \cap b^{\uparrow} = \emptyset$
and maximal with respect to these properties by Lemma~3.9.
By (1) above, $J$ is a prime $\vee$-closed order ideal.
Thus by Lemma~3.10, $S \setminus J$ is a prime filter in $S$.
By construction this prime filter contains $b$ and omits $a$, as required.
\end{proof}

The above lemma enables us to prove the following important result.

\begin{proposition}\label{coherent-implies-spatial}
Every coherent pseudogroup is spatial.
\end{proposition}
\begin{proof}
Let $S$ be a coherent pseudogroup.
By Lemma~3.5, we may assume that $S = \mbox{Idl}(T)$ where $T$ is a distributive inverse semigroup.
Let $A,B \in \mbox{Idl}(T)$ be distinct elements.
We shall construct a completely prime filters that contains one of these elements but not the other.
Without loss of generality, we may assume that there is $b \in B$ such that $b \notin A$.
It follows that
$$b^{\uparrow} \cap B = \emptyset.$$
Clearly $B$ is a $\vee$-closed order ideal.
Thus by Lemma~3.9, there exists a $\vee$-closed order ideal $P$ that contains $B$,
is disjoint from $b^{\uparrow}$,
and is a maximal $\vee$-closed order ideal with respect to these two conditions.
By Lemma~3.11, $P$ is a prime $\vee$-closed order ideal.
Thus by Lemma~3.10, the set $F = T \setminus P$ is a prime filter in $S$ that contains $b$ and is disjoint from $B$.
By Lemma~3.2, we have that $F^{u}$ is a completely prime filter in $S = \mbox{Idl}(T)$.
By definition $B \in F^{u}$ and $A \notin F^{u}$, as required.
\end{proof}

An \'etale groupoid $G$ is said to be {\em coherent} if it satisfies three conditions:
\begin{description}

\item[{\rm (C1)}] The set $\mathsf{KB}(G)$ of compact-open bisections forms a basis for the topology on $G$.

\item[{\rm (C2)}]  The set $\mathsf{KB}(G)$ of compact-open bisections is closed under subset multiplication.

\item[{\rm (C3)}] The \'etale groupoid $G$ is sober.

\end{description}

Whether or not an \'etale groupoid $G$ is coherent largely depends on the properties of $G_{o}$, as we now show.

\begin{lemma} Let $G$ be an \'etale groupoid.
\begin{enumerate}

\item Then $G$ has a basis of compact-open bisections if and only if $G_{o}$ has a basis of compact-open sets.

\item $G$ satisfies (C1) and (C2) if and only if $G_{o}$ satisfies (C1) and (C2).

\item $G$ is sober as an \'etale groupoid if and only if the space $G_{0}$ is sober.

\item $G$ is coherent if and only if $G_{o}$ is coherent.

\end{enumerate}
\end{lemma}
\begin{proof} (1) We suppose first that $G_{o}$ has a basis of compact-open sets.
We show that $G$ has a basis of compact-open bisections.
Let $U$ be any non-empty open bisection in $G$ and let $g \in U$.
Since $G$ is \'etale there is an open bisection $V$ containing $g$
such that $\mathbf{d}$ restricted to $V$ is a homeomorphism onto its image.
It follows that $\mathbf{d}$ restricted to $U \cap V$ is a homeomorphism onto its image and $g \in U \cap V$.
By assumption we may find a compact-open set (and therefore bisection) $B$ in $G_{o}$ containing $g^{-1}g$ and contained
in the image of $U \cap V$.
It follows that there is a compact-open bisection $A$ containing $g$ such that $\mathbf{d}$ maps $A$ to $B$
and which is contained in $U \cap V$.
It follows that every open bisection in $G$ is a union of compact-open bisections.
Thus the compact-open bisections form a basis for the topology.

Suppose now that $G$ has a basis of compact-open bisections.
We prove that $G_{o}$ has a basis of compact-open sets.
Let $U$ be an open set in $G_{o}$ and let $e \in U$.
There is therefore an open set $V$ in $G$ such that $U = G_{o} \cap V$.
Thus $U$ is also an open set in $G$.
Therefore there exists a compact-open bisection $W$ such that $e \in W \subseteq U$.
It follows that $W$ is a subset of $G_{o}$.

(2) We suppose first that $G_{o}$ satisfies (C1) and (C2).
This means that we assume that $G_{o}$ has a basis of compact-open sets and the intersection
of any two compact-open sets is again compact-open.

We prove first that if $A$ is a compact-open bisection then so too is $A^{-1}A$.
We need only prove that it is compact.
Let $A^{-1}A \subseteq \bigcup_{i} O_{i}$ be a covering by open bisections.
Then $A \subseteq \bigcup_{i} AO_{i}$ is also a covering by open bisections.
By assumption, $A$ is compact and so we may find a finite number $AO_{1}, \ldots, OA_{m}$ that cover $A$.
Thus $A \subseteq \bigcup_{i=1}^{m} AO_{i}$.
Hence $A^{-1}A \subseteq \bigcup_{i=1}^{m} A^{-1}AO_{i}$.
But $A^{-1}AO_{i} \subseteq O_{i}$ and so $A^{-1}A \subseteq \bigcup_{i=1}^{m} O_{i}$.
Thus $A^{-1}A$ is compact.

Let $A$ and $B$ be two compact-open bisections.
The product $AB$ is an open bisection so it only remains to show that it is compact.
Let $AB \subseteq \bigcup_{i} C_{i}$ where the $C_{i}$ are open bisections.
Then $A^{-1}ABB^{-1} \subseteq \bigcup_{i} A^{-1}C_{i}B^{-1}$.
Now $A^{-1}ABB^{-1} = A^{-1}A \cap BB^{-1}$ and so is compact.
Thus we may write
$A^{-1}ABB^{-1} \subseteq \bigcup_{i=1}^{m} A^{-1}C_{i}B^{-1}$
and so
$AB \subseteq \bigcup_{i=1}^{m} AA^{-1}C_{i}B^{-1}B \subseteq \bigcup_{i=1}^{m} C_{i}$, as required.

The proof of the converse is straightforward.

(3) Observe first that a covering functor $\eta \colon G \rightarrow H$ is bijective if and only if
the function $\eta \mid G_{o} \colon G_{o} \rightarrow H_{o}$ is bijective.
It follows that $\eta \colon G \rightarrow \mathsf{G}(\mathsf{B}(G))$ is bijective
if and only if $\eta \mid G_{o} \colon G_{o} \rightarrow \mathsf{G}(\mathsf{B}(G)_{o})$ is a bijective.
However, filters that are identities are determined by their idempotent elements,
and the idempotents in $\mathsf{B}(G)$ are the open subsets of $G_{o}$.
It follows that $\eta \colon G \rightarrow \mathsf{G}(\mathsf{B}(G))$ is bijective
if and only if 
$\eta \colon G_{o} \rightarrow \mathsf{G}(\mathsf{B}(G_{o}))$ is bijective.
Now observe that if $b$ is an open bisection in $G$ then $\eta (b) = X_{b}$.
It follows that $\eta$ is always an open map.
We have therefore proved that $G$ is sober if and only if $G_{o}$ is sober.

(4) This is immediate by (1), (2) and (3) above.

\end{proof}

Our definition of a {\em coherent space} is more general than the one given in \cite{J},
since we do not require the space itself to be compact.
It is a sober space, in our sense, in which the compact-open sets form a basis that is closed under binary meets.

The result by Exel \cite{Exel3} can be viewed within this setting as a proof of sobriety.

\begin{lemma} Let $G$ be a coherent \'etale groupoid.
Then $\mathsf{KB}(G)$ is a distributive inverse semigroup and $\mathsf{B}(G)$ is a coherent pseudogroup
\end{lemma}
\begin{proof} The inversion map on an \'etale groupoid $G$ is a homeomorphism of $G$ to itself.
Thus $\mathsf{KB}(G)$ is an inverse subsemigroup of $\mathsf{B}(G)$.
If $A$ and $B$ are compatible compact-open bisections then their join is an open bisection
and compact because the union of finitely many compact subsets is compact.
Because $\mathsf{B}(G)$ is infinitely complete and infinitely distributive it follows that
$\mathsf{KB}(B)$ is finitely complete and finitely distributive.
We have used here (C2) above in the definition.

The finite elements of $\mathsf{B}(G)$ are the compact ones and so precisely the compact-open bisections.
Let $A \in \mathsf{B}(G)$ be an arbitrary open bisection.
By (C1) above in the definition, $A$ can be written as a union of compact-open bisections.
Thus every element of $\mathsf{B}(G)$ is a join of finite elements.
We have shown that $\mathsf{B}(G)$ is coherent.
\end{proof}

\begin{lemma} Let $S$ be a coherent pseudogroup.
Then $\mathsf{G}(S)$ is a coherent \'etale groupoid. 
Moreover, the isomorphism  $\varepsilon$ establishes a bijection 
between  the finite elements of $S$ and the compact-open bisections.
\end{lemma}
\begin{proof} By Proposition~2.12(2), the groupoid  $\mathsf{G}(S)$ is sober.
Now $S$ is isomorphic to $\mathsf{B} (\mathsf{G}(S))$ via $\varepsilon$ since every coherent pseudogroup is spatial by Proposition \ref{coherent-implies-spatial}.
Under this isomorphism, finite elements of $S$ are mapped to the compact-open bisections of $\mathsf{G}(S)$.
By the coherence of $S$, every $s\in S$ is the joint of finite elements and hence every $X_s$ is a union of compact-open bisections 
which stem from the  range of $\varepsilon$.
Thus every open bisection is a union of compact-open bisections.  
It follows that $\mathsf{G}(S)$ satisfies (C1). Moreover, every compact-open bisection comes comes from a finite element of $S$ 
since the finite elements are closed under finite joins by coherence.
The condition (C2) holds because the finite elements of $S$ are closed under multiplication. 
\end{proof}

\begin{lemma} Let $\theta \colon G \rightarrow H$ be a continuous covering functor between coherent \'etale groupoids
with the property that the inverse image of every compact-open bisection is a compact-open bisection.
Then the inverse image of every compact-open set is a compact-open set.
\end{lemma}
\begin{proof} Let $X$ be a compact-open subset of $H$.
Since the groupoid $H$ is coherent, the compact-open bisections form a basis.
Thus we may write $X$ as a union of compact-open bisections
and so by compactness, we may write it as a finite union of compact-open bisections.
It follows that the inverse image of $X$ under $\theta$ can be written as finite union of compact-open bisections.
Thus since $\theta^{-1}(X)$ is a finite union of compact sets it is compact.
\end{proof}

We shall say that a continuous map between topological spaces is {\em coherent} if the inverse image
under this map of any compact-open set is a compact-open set.

\begin{theorem}[Duality for distributive inverse semigroups]\label{ddis}
The category of distributive inverse semigroups and their callitic morphisms is dually equivalent to the category
of coherent \'etale groupoids and coherent continuous covering functors.
\end{theorem}
\begin{proof} We refine the duality of Theorem~2.23.

Let $\theta \colon \: G \rightarrow H$ be a coherent continous covering functor between coherent \'etale groupoids.
Then $\mathsf{B}(\theta) \colon \mathsf{B}(H) \rightarrow \mathsf{B}(G)$ is a callitic morphism of pseudogroups.
Since $G$ and $H$ are both coherent, both  $\mathsf{B}(G)$ and $\mathsf{B}(H)$ are coherent by Lemma~3.14.
Since $\theta$ is coherent  $\mathsf{B}(\theta)$ maps finite elements to finite elements by Lemma~3.16.
Thus $\mathsf{B}(\theta)$ is coherent.

Let $\theta \colon S \rightarrow T$ be a callitic coherent morphism between coherent pseudogroups.
Then $\mathsf{G}(\theta) \colon \mathsf{G}(T) \rightarrow \mathsf{G}(S)$ is a continuous covering functor.
By Lemma~3.15 both $\mathsf{G}(S)$ and $\mathsf{G}(T)$ are coherent \'etale groupoids.
Since $\theta$ is also coherent, the inverse image under $\mathsf{G}(\theta)$ of every compact-open bisection is a compact-open bisection.
It follows by Lemma~3.16 that $\mathsf{G}(\theta)$ is coherent.
Finally, the category of  distributive inverse semigroups and their callitic morphisms is equivalent to
the category of coherent pseudogroups and their callitic morphisms by Proposition~3.7.
\end{proof}

We are now going to explicitly compute  the functor giving the  dual equivalence in the previous theorem using  the prime filters. 
Let $S$ be a distributive inverse semigroup.
Define $\mathsf{G_{P}}(S)$ to be the set of prime filters of $S$.
For each $s \in S$ define $Y_{s}$ to be the set of all prime filters that contains $s$.
Put $\pi = \{Y_{s} \colon s \in S \}$.

\begin{lemma} Let $S$ be a distributive inverse semigroup.
Then  $\mathsf{G_{P}}(S)$ is a groupoid and
$\pi$ is a basis for a topology that makes $\mathsf{G_{P}}(S)$
a topological groupoid.
\end{lemma}
\begin{proof}
Let $P \in Y_{s} \cap Y_{t}$.
Then $s,t \in P$.
But $P$ is a filter and so there exists $p \in P$ such that $p \leq s,t$.
Thus $P \in Y_{p} \subseteq Y_{s} \cap Y_{t}$.
Thus $\pi$ is a basis.
Next we have to check that the product of prime filters is a prime filter
so that $\mathsf{G_{P}}(S)$ is a groupoid.
This uses similar arguments to Lemmas~2.2 and 2.4.
The proof that $\mathsf{G_{P}}(S)$ is a topological groupoid is similar to the proof of Lemma~2.7.
\end{proof}

\begin{proposition} \label{gpt-equal-gs} Let $S$ be a coherent pseudogroup and let $T$ be  a distributive inverse semigroup with $S = \mbox{\rm Idl} (T)$ and $T  = K(S)$.
Then the groupoid $\mathsf{G}_{P} (T)$ is homeomorphic to the groupoid $\mathsf{G}(S)$.
\end{proposition}
\begin{proof} Define a map $\mathsf{G}_{P} (T) \rightarrow \mathsf{G}(S)$ by $P \mapsto P^{u}$.
By Lemma~3.2 this is a bijection.
It is routine to check that this is a functor.
Finally, we need to check that the map is continuous and open.
A basic open set in $\mathsf{G}(S)$ has the form $X_{s}$.
By assumption, $s = \bigvee_{i \in I} s_{i}$ where the $s_{i}$ are finite elements from $T$.
Thus by Lemma~2.6(5), we have that
$X_{s} = \bigcup_{i \in I} X_{s_{i}}$ where the $s_{i}$ are finite.
The inverse image of $X_{t}$ where $t$ is finite is precisely $Y_{t}$.
Thus the map is continuous.
The image of $Y_{t}$ is just $X_{t}$ and so the map is open.
\end{proof}

The groupoid $\mathsf{G}_{P} (S)$, where $S$ is a distributive inverse semigroup, is called the {\em prime spectrum} of $S$.
In Theorem~3.17, the functor from distributive inverse semigroups to \'etale groupoids can be 
replaced by $\mathsf{G}_{P}$ as by the proposition above $\mathsf{G} (\mbox{Idl} (T))$ is homeomorphic to $\mathsf{G}_{P} (T)$.

\subsection{Non-commutative Stone duality for boolean inverse semigroups}

We shall now specialize Theorem~3.17 to obtain a new proof of the duality proved directly in \cite{Law4}
from which the monoid case, first proved in \cite{Law3}, follows as a corollary. 
It is important to be clear about the definitions we shall give so we highlight them:
\begin{itemize}

\item A distributive inverse semigroup whose semilattices of idempotents is a boolean algebra 
is called a {\em weakly boolean inverse semigroup}.\footnote{This is not an ideal term. Perhaps {\em pre-boolean inverse semigroup} would be better.}

\item A weakly boolean inverse semigroup that is also an inverse $\wedge$-semigroup is called a {\em boolean inverse semigroup}.
A {\em morphism of boolean inverse semigroups} is an inverse semigroup homomorphism that preserves binary meets and binary compatible joins.

\end{itemize}

Let $S$ be a boolean inverse semigroup.
Let $a,b \in S$ such that $b \leq a$.
Then we may construct a unique element, denoted by $a \setminus b$ such that $b$ and $a \setminus b$ are orthogonal
and $a = b \vee (a \setminus b)$.
See Lemma~3.27 for a proof in a slightly more general setting.
We call $a \setminus b$ the {\em (relative) complement of $b$ in $a$}.

\begin{lemma} \mbox{}
\begin{enumerate}

\item Let $S$ be a distributive inverse semigroup.
Then every ultrafilter is a prime filter. 

\item  Let $S$ be a distributive inverse semigroup.
It is weakly boolean if and only if every prime filter is an ultrafilter.

\item  Let $S$ be a distributive inverse $\wedge$-semigroup.
It is boolean if and only if every prime filter is an ultrafilter.

\end{enumerate}
\end{lemma}
\begin{proof}
(1) Observe that $F$ is a prime filter (respectively, ultrafilter) in $S$ if and only if $F^{-1} \cdot F$ is a prime idempotent filter
(respectively, idempotent ultrafilter).
Next observe that $G$ is an idempotent prime filter (respectively, ultrafilter) in $S$ if and only if $E(G)$ is a prime filter in $E(S)$ (respectively, ultrafilter).
We now apply Proposition~1.6(1).

(2) This follows by the argument in (1) above combined with Proposition~1.6(2).

The proof of (3) is immediate by (2).
\end{proof}

It follows by the above result that the {\em callitic} morphisms of boolean inverse semigroups
are just the morphisms under which the inverse images of ultrafilters are ultrafilters.

A {\em boolean groupoid} is a hausdorff \'etale topological groupoid with a basis of compact-open bisections whose space of identities is a boolean space.
A {\em morphism of boolean groupoids}  is a proper continuous covering functor.

It is useful to deconstruct the definition of a boolean groupoid.
The following is proved as Lemma~2.37 of \cite{Law4}.

\begin{lemma} Let $G$ be a hausdorff \'etale topological groupoid.
\begin{enumerate}

\item $G$ has a basis of compact-open bisections if and only if $G_{o}$ has a basis of compact-open bisections.

\item The product of two compact-open subsets is compact-open.

\end{enumerate}
\end{lemma}

The following is Lemma~2.42(1) of \cite{Law4}.

\begin{lemma}
Boolean groupoids are sober.
\end{lemma}

\begin{proposition} \mbox{}
\begin{enumerate}

\item The hausdorff coherent \'etale groupoids are precisely the boolean groupoids.

\item If $G$ is a hausdorff coherent \'etale groupoid then $\mathsf{KB}(G)$ is a boolean inverse semigroup.

\item If $S$ is a coherent pseudogroup whose finite elements form a boolean inverse semigroup then $\mathsf{G}(S)$ is hausdorff. 
In particular, $\mathsf{G}_P (T)$ is hausdorff for all boolean inverse semigroups $T$. 

\item Let $S$ be a distributive inverse semigroup.
Then $\mathsf{G}_P (S)$ is hausdorff if and only if $S$ is a boolean inverse semigroup.

\end{enumerate}
\end{proposition}
\begin{proof} (1) Immediate.

(2) Observe that the semilattice of idempotents of  $\mathsf{KB}(G)$ is given by the compact-open subsets of the boolean space $G_{o}$ and so forms a boolean algebra.
It remains to show that $\mathsf{KB}(G)$ is an inverse $\wedge$-semigroup.
Let $A$ and $B$ be two compact-open bisections.
Clearly $A \cap B$ is an open bisection so it only remains to  show that it is compact.
Compact subsets of hausdorff spaces are closed.
Thus $A \cap B$ is closed.
But $A \cap B$ is a closed subset of the compact set $A$ and so $A \cap B$ is compact.

(3) We show the last statement. The first statement the follows from Lemma \ref{gpt-equal-gs}.
We may write $S = \mbox{Idl}(T)$ where $T$ is a boolean inverse semigroup.
Let $F$ and $G$ be distinct elements of $\mathsf{G}(S)$.
Thus they are distinct completely prime filters in $S$.
It follows that $F^{d}$ and $G^{d}$ are distinct prime filters in $T$ and so by Lemma~3.20 they are distinct ultrafilters.
It follows that there exists $a \in A$ such that $b \notin B$.
But $T$ is an inverse $\wedge$-semigroup and $B$ is an ultrafilter
and so by Lemma~2.7(2) of \cite{Law4}, there exists $c \in B$ such that $b \wedge c = 0$.
By definition, $X_{a^{\downarrow}}$ is the set of all completely prime filters that contain $a^{\downarrow}$
and $X_{c^{\downarrow}}$ is the set of all completely prime filters that contain $c^{\downarrow}$.
Observe that $F \in X_{a^{\downarrow}}$  and $G \in X_{b^{\downarrow}}$;
both sets are open sets in $\mathsf{G}(S)$;
and their interesection is empty because $b \wedge c = 0$.

(4) We have that  $\mathsf{G}_P (S)$ is hausdorff if $S$ is a boolean inverse semigroup.
Suppose now that $\mathsf{G}_P (S)$ is hausdorff.
Then  we have that $\mathsf{KB}(\mathsf{G}_P (S))$ is a boolean inverse semigroup.
But $S$ is isomorphic to $\mathsf{KB}(\mathsf{G}_P (S))$ by Theorem~3.17 and so is a boolean inverse semigroup.
\end{proof}

\begin{lemma}
Let $\theta \colon G \rightarrow H$ be a coherent continuous covering functor between boolean groupoids.
Then $\theta$ is proper.
\end{lemma}
\begin{proof} Let $X$ be a compact subset of $H$.
By assumption $H$ has a basis of compact-open bisections and so $H$ are thus $X$ is covered by a family of compact-open bisections.
It follows that $X$ is covered by a finite set of compact-open bisections.
Thus we may write $X \subseteq \bigcup_{i=1}^{m} B_{i}$ where the $B_{i}$ are compact-open bisections.
It follows that $\theta^{-1} (X) \subseteq  \bigcup_{i=1}^{m} \theta^{-1}(B_{i})$.
The union is a finite union of compact-open bisections.
Now $X$ is a compact subset of a hausdorff space and so $X$ is closed.
It follows that $\theta^{-1}(X)$ is a closed subset of $G$.
But $\theta^{-1}(X)$ is a closed subset of a compact set and so is itself compact.
Thus $X$ compact implies that $\theta^{-1}(X)$ is compact.
\end{proof}

Using the preceding lemmas we can now specialize Theorem \ref{ddis} to boolean inverse semigroups and boolean groupoids.  
This gives immediately the following result.
It was first proved by direct means as Theorem~2.40 of \cite{Law4}.

\begin{theorem}[Duality for boolean inverse semigroups]
The category of boolean inverse semigroups and their callitic morphisms is dual to the category of boolean groupoids
and their proper continuous covering functors.
\end{theorem}

If $S$ is an inverse semigroup, we denote by $\mathsf{G}_{M}(S)$ the set of ultrafilters on $S$.
In Theorem~3.25, the functor from boolean inverse semigroups to boolean groupoids can be replaced by $\mathsf{G}_{M}$ by Lemma~3.20.

\subsection{Weakly boolean inverse semigroups}

Under the duality of Theorem~3.17, weakly boolean inverse semigroups correspond to those coherent \'etale groupoid whose space of identities is hausdorff.
This is weaker than assuming that the whole groupoid is hausdorff which by Proposition~3.23 implies that the associated inverse semigroup is in fact boolean.
Accordingly, a {\em weakly boolean groupoid} is a coherent groupoid whose space of identities is hausdorff.
By Lemma~3.13, this is equivalent to an \'etale groupoid whose space of identities is a boolean space.

\begin{theorem}[Duality for weakly boolean inverse semigroups]
The category of weakly boolean inverse semigroups and their callitic morphisms is dually equivalent to the category
of weakly boolean groupoids and coherent continuous covering functors.
\end{theorem}

Weakly boolean inverse semigroups turn out to be important in understanding Paterson's universal groupoid.
We shall describe how in Section~5.1. 
For this reason, it is convenient to prove here some simple results about such semigroups.

\begin{lemma} Let $S$ be a weakly boolean inverse semigroup.
Let $a,b \in S$ such that $b \leq a$.
Then we may construct a unique element, denoted by $a \setminus b$, such that $b$ and $a \setminus b$ are orthogonal and $a = b \vee (a \setminus b)$.
\end{lemma}
\begin{proof}
We have that $\mathbf{d}(b) \leq \mathbf{d}(a)$.
But the semilattice of idempotents of $S$ is a boolean algebra.
Thus there exists $e \leq \mathbf{d}(a)$ such that $\mathbf{d}(a) = e \vee \mathbf{d}(b)$ and $e \wedge \mathbf{d}(b) = 0$.
Since we are working in a distributive inverse semigroup,
it follows that $a = b \vee ae$ and $b \wedge ae = 0$.
Suppose that $x \leq a$ is such that $a = b \vee x$ and $b \wedge x = 0$.
Then because we are working inside a principal order ideal we have that $\mathbf{d}(a) = \mathbf{d}(b) \vee \mathbf{d}(x)$ and
$\mathbf{d}(b) \wedge \mathbf{d}(x) = 0$.
But by uniqueness of relative complements in booleans algebras we have that $e = \mathbf{d}(x)$ and so $x = ae$.
\end{proof}

We denote the element $ae$ by $a \setminus b$, the {\em relative complement of $b$ in $a$}.

\begin{lemma} Let $S$ be a weakly boolean inverse semigroup.
Suppose that $t \leq s$ and $v \leq u$.
Then 
$(s \setminus t)(u \setminus v) = su \setminus (sv \vee tu \vee tv )$.
\end{lemma}
\begin{proof}
We have that
$$su = (s \setminus t)(u \setminus v)
\vee
tv
\vee
t(u \setminus v)
\vee
(s \setminus t)v.$$
But
$$tu = tv \vee t(u \setminus v) \mbox{ and } sv = tv \vee (s \setminus t)v.$$
Thus
$$su = (s \setminus t)(u \setminus v)
\vee
tv
\vee
sv
\vee
tu.$$
It remains to show that
$(s \setminus t)(u \setminus v)$
and
$tv \vee sv \vee tu$
have only zero as a common lower bound.
But this follows from the fact that 
$tv$, 
$t(u \setminus v)$
and
$(s \setminus t)v$ are pairwise orthogonal.
\end{proof}

The following result will be used in the context of weakly boolean inverse semigroups although we prove it in a more general case.
Recall that $Y_{s}$ denotes the set of all prime filters containing $s$.

\begin{lemma} Let $S$ be a distributive inverse semigroup.
If every prime filter in $Y_{s}$ is idempotent then $s$ is idempotent.
\end{lemma}
\begin{proof} Let $P$ be any prime filter containing $s$.
By assumption, $P$ is an idempotent filter and so by Lemma~2.3 it is an inverse subsemigroup of $S$.
It follows that $Y_{s} = Y_{s^{-1}}$.
Thus $s = s^{-1}$ by Lemma~3.11(2).
Similarly $Y_{s} \subseteq Y_{s^{2}}$.
Both $Y_{s} \cdot Y_{s}$ and $Y_{s^{2}} \cdot Y_{s}$ are defined and so $Y_{s^{2}} \subseteq Y_{s}$.
It follows that $Y_{s} = Y_{s^{2}}$ and so by Lemma~3.11(2), we have that $s$ is an idempotent.
\end{proof}

\section{Dense and tight coverages}

In this section, we shall describe a way of constructing pseudogroups from inverse semigroups and of constructing \'etale groupoids from families of filters.

\subsection{Groupoids determined by families of filters}

We show how the different constructions of \'etale groupoids from filters in Section~3 may be unified. 
Our approach was motivated by the notion of a coverage in the theory of frames \cite{J} but modified to deal with the greater complexity of the inverse case.
We observe first that the different kinds of filters were defined with respect to the join operation.
We shall therefore abstract the properties of the join that are necessary to obtain an \'etale groupoid.

A {\em coverage} $\mathcal{C}$ on an inverse semigroup $S$ is defined by the following data.
For each $a \in S$, there is a set $\mathcal{C}(a)$ of subsets of $a^{\downarrow}$, whose elements are called {\em coverings},
satisfying the following axioms: 
\begin{description}

\item[{\rm (R)}] $\{a \} \in \mathcal{C}(a)$ for all $a \in S$.

\item[{\rm (I)}] If $X \in \mathcal{C}(a)$ then $X^{-1} \in \mathcal{C}(a^{-1})$.

\item[{\rm (MS)}] $X \in \mathcal{C}(a)$ and $Y \in \mathcal{C}(b)$ imply that $XY \in \mathcal{C}(ab)$.

\item[{\rm (T)}] If $X \in \mathcal{C}(a)$ and  $X_{i} \in \mathcal{C}(x_{i})$ for each $x_{i} \in X$ then $\bigcup_{i} X_{i} \in \mathcal{C}(a)$.

\end{description}
For motivation, work in a pseudogroup and interpret $X \in \mathcal{C}(a)$ to mean that $a = \bigvee_{x \in X} x$ and observe that each of the above axioms
is true. 

A filter $A$ on $S$ is called a {\em $\mathcal{C}$-filter} if $x \in A$ and $X \in \mathcal{C}(x)$ then $y \in A$ for some $y \in X$.
We shall use the word {\em family} to describe the set of all $\mathcal{C}$-filters for a given coverage $\mathcal{C}$.
We now show how these definitions unify what we have discussed so far and take the opportunity to introduce two further examples.

\begin{examples} {\em Let $S$ be an inverse semigroup.
\begin{enumerate}

\item The simplest weak coverage is defined by putting $\mathcal{C}(x) =\{ \{x \} \}$.
We call this the {\em trivial coverage}.
The $\mathcal{C}$-filters are just the {\em filters}.

\item Let $S$ be a distributive inverse semigroup. Define $\mathcal{C}(x)$ to be those finite subsets of $x^{\downarrow}$
whose joins are $x$. This defines a coverage.
The $\mathcal{C}$-filters are just the {\em prime filters}.

\item Let $S$ be a pseudogroup.  Define $\mathcal{C}(x)$ to be those subsets of $x^{\downarrow}$
whose joins are $x$. This defines a coverage.
The $\mathcal{C}$-filters are just the {\em completely prime filters}.

\item Let $S$ be an inverse semigroup.
To define our next coverage we need some notation. 
Let $a \in A$ and $B \subseteq S$.
Define $a \rightarrow B$ to mean that for each $0 \neq x \leq a$ there exists $b \in B$ such that $x^{\downarrow} \cap b^{\downarrow} \neq 0$.
We call this the {\em arrow relation} and it was first defined in \cite{L}.
For each $a \in S$ define $\mathcal{T}(a)$ to consist of those subsets
$B \subseteq a^{\downarrow}$ such that $a \rightarrow B$.
We claim that $\mathcal{D}$ defines a coverage on $S$. 
It is immediate that (R) and (I) hold.
Suppose that $X \in \mathcal{D}(a)$ and $Y \in \mathcal{D}(b)$.
Then since $X \subseteq a^{\downarrow}$ and $Y \subseteq b^{\downarrow}$ we have that $XY \subseteq (ab)^{\downarrow}$.
Let $0 \neq z \leq ab$.
Then $aa^{-1}z = z$ and so $a^{-1}z \neq 0$.
It follows that $0 \neq a^{-1}z \leq a^{-1}ab \leq b$.
Thus there exists $y \in Y$ and a $u$ such that $u \leq y, a^{-1}z$.
Observe that $a^{-1}au = u$ and so $au \neq 0$.
Thus $0 \neq au \leq ay, z$.
We now carry out a similar calculation starting from $au \leq ay$.
Then $auy^{-1} \leq a$ and so there exists $x \in X$ and a $v$ such that $v \leq x, auy^{-1}$.
Observe that $vyy^{-1} = v$ and so $vy \neq 0$.
Thus $0 \neq vy \leq xy, auy^{-1}y = au \leq z$.
It follows that (MS) holds.
Finally, we check that (T) holds.
Let $X \in \mathcal{D}(a)$ and suppose that for each $x_{i} \in X$ we have that $X_{i} \in \mathcal{D}(x_{i})$.
We prove that $\bigcup_{i} X_{i} \in \mathcal{D}(x)$.
Let $0 \neq z \leq a$.
Then there exists $0 \neq u \leq z,x_{i}$ for some $x_{i} \in X$.
But $0 \neq u \leq x_{i}$ implies that there exists $y \in X_{i}$ and a $v$ such that
$0 \neq v \leq y,u$.
Thus there exists $0 \neq u \leq z, y$ where $y \in X_{i}$, as required.
The $\mathcal{D}$-filters are called {\em dense filters}.

\item This is the same as the above example but we only consider the finite subsets of $\mathcal{D}$.
The resulting coverage is denoted by $\mathcal{T}$ and is called the {\em tight coverage}.
The $\mathcal{T}$-filters are called {\em tight filters}.

\end{enumerate}
}
\end{examples}

Throughout the remainder of this section, $\mathcal{C}$ will be a coverage.
If $X \in \mathcal{C}(a)$ define $\mathbf{d}(X) = \{x^{-1}x \colon x \in X \}$.

\begin{lemma} Let $\mathcal{C}$ be a coverage on $S$.
\begin{enumerate}

\item $X \in \mathcal{C}(a)$ implies that $\mathbf{d}(X) \in \mathcal{C}(a^{-1}a)$.

\item If $X \subseteq a^{\downarrow}$ then $X \in \mathcal{C}(a)$ if and only if $\mathbf{d}(X) \in \mathcal{C}(a^{-1}a)$.

\item Let $X,Y \in \mathcal{C}(a)$.
Then $X \wedge Y = \{x \wedge y \colon x \in X, y \in Y\} \in \mathcal{C}(a)$
and $X \wedge Y = X\mathbf{d}(Y) = Y\mathbf{d}(X)$.

\item If $X \in \mathcal{C}(b)$ and $X,Y \in \mathcal{C}(a)$ then $X \wedge Y \in \mathcal{C}(b)$.

\end{enumerate}
\end{lemma}
\begin{proof} (1) We have that $x \in X$ implies that $x \leq a$ and so $x = ax^{-1}x$.
Thus $a^{-1}x = x^{-1}x$.
By (R), we have that $\{a^{-1} \} \in \mathcal{C}(a^{-1})$ and so by (MS) we have that $a^{-1}X \in \mathcal{C}(a^{-1}a)$
but $a^{-1}X = \{x^{-1}x \colon x \in X   \}$ and the claim follows.

(2) By (1), only one direction needs proving. Suppose that $X \subseteq a^{\downarrow}$ and  $\mathbf{d}(X) \in \mathcal{C}(a^{-1}a)$.
Then by (MS), we have that $a\mathbf{d}(X) \in \mathcal{C}(a)$.
But $X = a\mathbf{d}(X)$ and the result follows.

(3) Observe first that since $x,y \leq a$ the meet $x \wedge y$ is defined.
Since $x$ and $y$ are compatible, $x \wedge y = xy^{-1}y = yx^{-1}x$.
Thus $X \wedge Y = X\{y^{-1}y \colon y \in Y\}$ which belongs to $\mathcal{C}(a)$ by (1) and (MS).

(4) It remains to show that $X \wedge Y \in \mathcal{C}(b)$.
For each $x \in X$ we have that $x \leq a$ and so $x = xa^{-1}a$.
Thus for each $x \in X$ we have that $x\mathbf{d}(Y) \in \mathcal{C}(x)$.
But $X \in \mathcal{C}(b)$ and so by (T), we have that $X\mathbf{d}(Y) \in \mathcal{C}(b)$.

\end{proof}

Our goal now is to show that the set of all $\mathcal{C}$-filters forms an \'etale groupoid.

\begin{lemma} Let $A$ be a filter on the inverse semigroup $S$.
\begin{enumerate}

\item $A$ is a $\mathcal{C}$-filter if and only if $A^{-1}$ is a $\mathcal{C}$-filter.

\item $A$ is a $\mathcal{C}$-filter if and only if $A^{-1} \cdot A$ is a $\mathcal{C}$-filter.

\end{enumerate}
\end{lemma}
\begin{proof} (1) Suppose that $A$ is a $\mathcal{C}$-filter.
Let $x \in A^{-1}$ where $X \in \mathcal{C}(x)$.
Then $x^{-1} \in A$ and $X^{-1} \in \mathcal{C}(x^{-1})$ by axiom (I).
By assumption, there exists $y \in X^{-1}$ such that $y \in A$.
But then $y^{-1} \in A^{-1}$ where $y^{-1} \in X$, as required.
 
(2) Suppose that $A$ is a $\mathcal{C}$-filter.
Let $x \in A^{-1} \cdot A$ where $X \in \mathcal{C}(x)$.
Then $a^{-1}b \leq x$ where $a,b \in A$.
It follows that $ax \in A$ where $aX \in \mathcal{C}(ax)$ by axiom (MS).
By assumption, $ay \in A$ for some $y \in X$.
Thus $a^{-1}ay \in A^{-1} \cdot A$ and so $y \in  A^{-1} \cdot A$, as required.

Suppose now that $A^{-1} \cdot A$ is a $\mathcal{C}$-filter.
Let $x \in A$ where $X \in \mathcal{C}(x)$.
Then $x^{-1}x \in A^{-1} \cdot A$ where $x^{-1}X \in \mathcal{C}(x^{-1}X)$.
By assumption, $x^{-1}y \in A^{-1} \cdot A$ where $x^{-1}X \in \mathcal{C}(x^{-1}x)$.
Thus $x^{-1}y \in A^{-1} \cdot A$ for some $y \in A$.
Thus $x^{-1}xy \in A$ and so $y \in A$, as required.
\end{proof}

\begin{lemma} 
If $A$ and $B$ are $\mathcal{C}$-filters and if $A \cdot B$ exists then $A \cdot B$ is a $\mathcal{C}$-filter.
\end{lemma}
\begin{proof} As we explained in Section~1, $A \cdot B$ is a filter and 
$\mathbf{d}(A \cdot B) = \mathbf{d}(B)$.
Thus the result follows from the lemma above.
\end{proof}

It follows that we may define the groupoid $\mathsf{G}_{\mathcal{C}}(S)$ of $\mathcal{C}$-filters of $S$.

For each $s \in S$, define $Z_{s}$ to be the set of all $\mathcal{C}$-filters that contain $s$.
Define $\xi$ to be the set of all such sets.

\begin{lemma} \mbox{}
\begin{enumerate}

 \item $Z_{s}$ is a bisection.

\item $Z_{s}^{-1} = Z_{s^{-1}}$.

\item $Z_{s}Z_{t} = Z_{st}$.

\item $Z_{s} \cap Z_{t}$ is a union of elements of $\xi$.

\end{enumerate}
\end{lemma}
\begin{proof} (1) This follows by Lemma~2.11 of \cite{Law3}.

(2) This follows by Lemma~4.3(1).

(3) The inclusion $Z_{s}Z_{t} \subseteq Z_{st}$ follows by Lemma~4.4.
The proof of the reverse inclusion uses the same argument as Lemma~2.21(4) combined with Lemma~4.3(2).

(4) Let $A \in Z_{s} \cap Z_{t}$.
Then $s,t \in A$.
Since $A$ is a filter there exists $a \in A$ such that $a \leq s,t$.
Observe that $Z_{a} \subseteq Z_{s} \cap Z_{t}$ and that $A \in Z_{a}$. 
\end{proof}

It follows that $\xi$ is a basis for a topology on  $\mathsf{G}_{\mathcal{C}}(S)$.

\begin{proposition} For each coverage $\mathcal{C}$ on the inverse semigroup $S$,  
the groupoid $\mathsf{G}_{\mathcal{C}}(S)$ is an \'etale topological groupoid.
\end{proposition}
\begin{proof} 
The proof follows the same approach as used in Lemma~2.7 and Proposition~2.8 by virtue of Lemma~4.5.
\end{proof}

\subsection{Universal pseudogroups determined by coverages}

Let $S$ be an inverse semigroup equipped with a coverage $\mathcal{C}$.
A semigroup homomorphism $\theta \colon S \rightarrow T$ to a pseudogroup
is said to be a {\em $\mathcal{C}$-cover-to-join} map
if for each element $a \in S$ and $\mathcal{C}$-cover $A$ of $a$ we have that $\theta (a) = \bigvee_{a_{i} \in A} \theta (a_{i})$.
We single out two cases for special terminology.
\begin{itemize}

\item In the case $\mathcal{C} = \mathcal{D}$, we shall refer instead to {\em dense maps}.

\item In the case $\mathcal{C} = \mathcal{T}$, we shall refer instead to {\em tight maps}.

\end{itemize}
Suppose that $\pi \colon S \rightarrow \mathsf{P}_{\mathcal{C}}(S)$ is a $\mathcal{C}$-cover-to-join map to a pseudogroup
such that if
$\theta \colon S \rightarrow T$ is any  $\mathcal{C}$-cover-to-join map to a pseudogroup
then there is a unique morphism of pseudogroups 
$\bar{\theta} \colon  \mathsf{P}_{\mathcal{C}}(S) \rightarrow T$ such that $\theta = \bar{\theta} \pi$.
Then we call $\mathsf{P}_{\mathcal{C}}(S)$ the {\em universal pseudogroup} of $(S,\mathcal{C})$.

\begin{examples}\mbox{}
\begin{enumerate}
{\em

\item When $\mathcal{C}$ is the trivial coverage of Examples~4.1(1), the universal pseudogroup is just $C(S)$.

\item When $S$ is a distributive inverse semigroup equipped with the coverage coming from finite joins, Examples~4.1(2), 
the universal pseudogroup is ${\rm Idl}(S)$ by Proposition~3.1.

}
\end{enumerate}
\end{examples}

We do not know how to construct universal pseudogroups in general, but we shall show how to construct them in all the cases of interest to us in this paper.
To do this, we examine Examples~4.1 in more detail.
The first three examples share a common property.
In each case, $\mathcal{C}(a) \cap \mathcal{C}(b) \neq \emptyset$ implies that $a = b$.
We say that such semigroups are {\em separative (with respect to the coverage $\mathcal{C}$)}
and that the coverage is {\em separated}.

The last two examples are not separated.
For the tight coverage, this is the origin of Section~5 of \cite{L}.
It will turn out that constructing universal pseudogroups in the case where the semigroup is separative with respect to the coverage is easy.
In the case where the semigroup is not separative, we try to construct a homomorphic image which is separative with respect to the `induced coverage'.
The meaning of this latter phrase is unclear in general, but is not problematic in the case of the dense and tight coverages.
We deal with this case first and show that we can reduce it to the separative case.

Let $S$ be an inverse semigroup equipped with a coverage $\mathcal{C}$.
Define the relation $\equiv$ on $S$ by
$$a \equiv b \Leftrightarrow \mathcal{C}(a) \cap \mathcal{C}(b) \neq \emptyset.$$

\begin{lemma} 
The relation $\equiv$ is a congruence on $S$.
\end{lemma}
\begin{proof} We show first that $\equiv$ is an equivalence relation.
We have that $\{a \} \in \mathcal{C}(a)$ and so $a \equiv a$.
It is immediate that $a \equiv b$ implies that $b \equiv a$.
Suppose that $a \equiv b$ and $b \equiv c$.
Let $X \in \mathcal{C}(a) \cap \mathcal{C}(b)$ and $Y \in \mathcal{C}(b) \cap \mathcal{C}(c)$.
By Lemma~4.2(4), we have that $X \wedge Y \in \mathcal{C}(a) \cap \mathcal{C}(c)$
and so $a \equiv c$.
Thus $\equiv$ is an equivalence relation and it is a congruence by (MS).
\end{proof}

We denote by $\mathbf{S}$ the quotient of $S$ by $\equiv$,
and the $\equiv$-congruence class containing $s$ by $\mathbf{s}$.
There is a homomorphism $\sigma \colon S \rightarrow \mathbf{S}$ given by $s \mapsto \mathbf{s}$.

\begin{lemma} Let $\theta \colon S \rightarrow T$ be a $\mathcal{C}$-cover-to-join map to a pseudogroup.
Then there is a unique homomorphism $\bar{\theta} \colon \mathbf{S} \rightarrow T$ such that
$\bar{\theta} \sigma = \theta$.
\end{lemma}
\begin{proof} Suppose that $a \equiv b$.
Then there exists $X \in \mathcal{C}(a) \cap \mathcal{C}(b)$.
But $\theta$ is a $\mathcal{C}$-cover-to-join map and so 
$\theta (a) = \bigvee_{x \in X} \theta (x)$ 
and
$\theta (b) = \bigvee_{x \in X} \theta (x)$.
Thus $\theta (a) = \theta (b)$.
We may therefore define $\bar{\theta} (\mathbf{a}) = \theta (a)$.
\end{proof}

The above lemma shows that $\mathcal{C}$-cover-to-join maps factor through $\mathbf{S}$.
We now turn to the specific cases of interest to us.

\begin{lemma} Let $\mathcal{D}$ be the dense coverage on $S$.
\begin{enumerate}

\item Let $\mathbf{X} \in \mathcal{D}(\mathbf{a})$.
Then there exists $A \in \mathcal{D}(a)$ such that $\sigma (A) = \mathbf{X}$.
In addition, if $\mathbf{X}$ is finite then $A$ can be chosen to be finite. 

\item If $\mathcal{D}(\mathbf{a}) \cap \mathcal{D}(\mathbf{b}) \neq \emptyset$ 
then 
$\mathcal{D}(a) \cap \mathcal{D}(b) \neq \emptyset$.

\item Let $X \in \mathcal{D}(a)$. Then $\mathbf{X} \in \mathcal{D}(\mathbf{a})$.

\end{enumerate}
\end{lemma}
\begin{proof} Observe first that in this case the homomorphism $\sigma$ is $0$-restricted.

(1) By Lemma~4.2, we have that $\{\mathbf{x}^{-1}\mathbf{x} \colon \mathbf{x} \in \mathbf{X} \} \in \mathcal{D}(\mathbf{a}^{-1}\mathbf{a})$.
For each $\mathbf{x} \in \mathbf{X}$, choose an idempotent $e_{x}$ such that $\sigma (e_{x}) = \mathbf{x}^{-1}\mathbf{x}$.
Put $A = \{ae_{x} \colon \mathbf{x} \in \mathbf{X} \} \subseteq a^{\downarrow}$.
Observe that $\sigma (ae_{x}) = \mathbf{a}\mathbf{x}^{-1}\mathbf{x} = \mathbf{x}$.
Thus $\sigma (A) = \mathbf{X}$.
Observe that if $\mathbf{X}$ is finite then $A$ is finite.
We prove that $a \rightarrow A$.
Let $0 \neq z \leq a$.
Then $z = ak$ for some idempotent $k$.
Thus $0 \neq \sigma (z) \leq \sigma (a)$ since $\sigma$ is $0$-restricted.
Observe that $\sigma (z) = \sigma (a)\sigma (k)$.
Thus there exists $0 \neq \mathbf{u}$ and $\mathbf{x} \in \mathbf{X}$ such that $\mathbf{u} \leq \mathbf{z}, \mathbf{x}$.
Choose any idempotent $f$ such that $\sigma (f) = \mathbf{u}^{-1}\mathbf{u}$ and put $u = ae_{x}fk$.
Then $\sigma (u) = \mathbf{u}$, using the fact that $\mathbf{u}\sigma (k) = \mathbf{u}$, and so in particular $u \neq 0$.
By construction $u \leq ae_{x}, z$.
We have therefore proved that $a \rightarrow A$.

(2) By definition, there exists $\mathbf{X} \in \mathcal{D}(\mathbf{a}) \cap \mathcal{D}(\mathbf{b})$.
By (1) above, we may $A \in \mathcal{D}(a)$ and $B \in \mathcal{D}(b)$ such that $\sigma (A) = \mathbf{X} = \sigma (B)$.
Each element of $A$ has the form $ae_{x}$ 
and each element of $B$ has the form $bf_{x}$
where $\sigma (ae_{x}) = \mathbf{x} = \sigma (be_{x})$.
Thus $ae_{x} \equiv be_{x}$.
Choose $C_{x} \in \mathcal{D}(ae_{x}) \cap \mathcal{D}(be_{x})$ and put $C = \bigcup_{x} C_{x}$.
Then by axiom (T), we have that $C \in \mathcal{D}(a) \cap \mathcal{D}(b)$.

(3) Let $0 \neq \sigma (b) \leq \sigma (a)$.
Then $\sigma (b) = \sigma (ab^{-1}b)$.
Thus $b \equiv ab^{-1}b$.
Therefore there exists $Y \in \mathcal{D}(b) \cap \mathcal{D}(ab^{-1}b)$.
But $X \in \mathcal{D}(a)$ implies by (MS) that $Xb^{-1}b \in \mathcal{D}(ab^{-1}b)$.
Thus by Lemma~4.2, we have that $Y \wedge Xb^{-1}b \in \mathcal{D}(ab^{-1}b)$.
Now $0 \neq ab^{-1}b$ and so there exists $0 \neq z$ such that $z \leq ab^{-1}b$ and $z \leq y \wedge xb^{-1}b$ for some
$y \wedge xb^{-1}b \in Y \wedge Xb^{-1}b$.
But then $0 \neq \sigma (z) \leq \sigma (b), \sigma (x)$, as required.
\end{proof}

The following is immediate by the above lemma.

\begin{corollary} 
With respect to either the dense or tight coverage  on $S$, the quotient semigroup $\mathbf{S}$ is separative.
\end{corollary}

We now prove two key propositions.
The first provides more information about the factorization arising from Lemma~4.9.

\begin{proposition} Let $S$ be an inverse semigroup equipped with the dense coverage $\mathcal{D}$ and
let $\theta \colon S \rightarrow T$ be a $\mathcal{D}$-cover-to-join map to a pseudogroup.
Then the induced homomorphism $\bar{\theta} \colon \mathbf{S} \rightarrow T$ is a $\mathcal{D}$-cover-to-join map.
An analogous result holds for the tight coverage.
\end{proposition}
\begin{proof} Let $\mathbf{X} \in \mathcal{D}(\mathbf{a})$ in $\mathbf{S}$.
By Lemma~4.10, there exists $A \in \mathcal{D}(a)$ such that $\sigma (A) = \mathbf{X}$.
By assumption, $\theta (a) = \bigvee_{x \in A} \theta (x)$.
But $\bar{\theta}(\sigma (a)) = \theta (a)$ and
$\bigvee_{x \in X} \bar{\theta} (\sigma{x}) = \bigvee_{x \in A} \theta (x)$ and the result follows.
\end{proof}

The second proposition tells us that $S$ and $\mathbf{S}$ are essentially the same when it comes to constructing \'etale groupoids.

\begin{proposition} Let $S$ be an inverse semigroup equipped with the dense coverage $\mathcal{D}$.
Then the posets of dense filters on $S$ and those on $\mathbf{S}$ are order-isomorphic and this induces
a homeomorphism between the groupoids $\mathsf{G}_{\mathcal{D}}(S)$ and $\mathsf{G}_{\mathcal{D}}(\mathbf{S})$.
An analogous result holds for the tight coverage.
\end{proposition}
\begin{proof} We shall prove that the map $A \mapsto \sigma (A)$ is a homeomorphism.

Observe first that if $A$ is a dense filter in $S$ then $\sigma (x) \in \sigma (A)$ if and only if $x \in A$.
Suppose that $\sigma (x) \in \sigma (A)$.
Then $\sigma (x) = \sigma (a)$ for some $a \in A$.
Thus $x \equiv a$.
It follows that there exists $X \in \mathcal{D}(x) \cap \mathcal{D}(a)$.
But $a \in A$ and $A$ is a dense filter thus there exists $y \in X \cap A$.
But $y \leq x$ and so $x \in A$, as required.

Let $A$ and $B$ be dense filters.
If $A \subseteq B$ then clearly $\sigma (A) \subseteq \sigma (B)$.
Conversely, suppose that $\sigma (A) \subseteq \sigma (B)$.
Let $a \in A$.
Then $\sigma (a) \in \sigma (A) =\sigma (B)$.
Thus $\sigma (a) \in \sigma (B)$.
It follows by our observation above that $a \in B$ and so $A \subseteq B$.

We prove that if $A$ is a dense filter then $\sigma (A)$ is a dense filter.
By the observation above, it is clear that $\sigma (A)$ is a directed set.
Suppose that $\sigma (a) \in \sigma (A)$ where $a \in A$ and $\sigma (a) \leq \sigma (b)$.
Then  $\sigma (a) = \sigma (ba^{-1}a)$.
Thus again by the observation above, we have that $ba^{-1}a \in A$ and so $b \in A$ giving $\sigma (b) \in \sigma (A)$.
Let $\mathbf{X} \in \mathcal{D}(\mathbf{a})$ where $\mathbf{a} \in \sigma (A)$.
By lemma~4.10, there exists $X \in \mathcal{D}(a)$ such that $\sigma (X) = \mathbf{X}$.
But $a \in A$ and so there exists $b \in A \cap X$.
Thus $\sigma (b) \in \mathbf{X} \cap \sigma (A)$, as required.

The map $A \mapsto \mathbf{A}$ is a bijection.
Suppose that $A$ and $B$ are dense filters such that $\sigma (A) = \sigma (B)$.
Let $a \in A$.
Then $\sigma (a) \in \sigma (A)$ and so there exists $b \in B$ such that $\sigma (a) = \sigma (b)$.
It follows by the observation above that $b \in B$.
By symmetry it follows that $A = B$ and so the map is injective.
We now prove that this map is surjective.
Let $\mathbf{A}$ be a dense filter in $\mathbf{S}$.
Put $A = \sigma^{-1}(\mathbf{A})$.
It is clear that $A$ is closed upwards.
Let $a,b \in A$.
Then $\sigma (a), \sigma (b) \in \mathbf{A}$.
Thus there exists $\sigma (c) \in \mathbf{A}$ such that $\sigma (c) \leq \sigma (a), \sigma (b)$.
It follows that $\sigma (c) = \sigma (ac^{-1}c) = \sigma (bc^{-1}c)$.
Hence $ac^{-1}c \equiv bc^{-1}c$.
Thus there exists $X \in \mathcal{D}(ac^{-1}c) \cap \mathcal{D}(bc^{-1}c)$.
By Lemma~4.10, $\sigma (X) \in \mathcal{D}(\sigma (c))$.
Thus $\sigma (x) \in \mathbf{A}$, for some $x \in X$, since the filter is dense.
But then $x \in A$ and $x \leq a,b$. as required.
Finally, let $X \in \mathcal{D}(a)$ where $a \in A$.
Then by Lemma~4.10, we have that $\mathbf{X} \in \mathcal{D}(\mathbf{a})$.
Thus $\sigma (x) \in \mathbf{A}$ for some $x \in X$ and so $x \in A$, as required.
It follows that we have shown that the map is a bijection.

It remains to show that it induces a functor between the groupoids and that it is a homeomorphism.
Let $A$ be a dense filter.
We prove that $\sigma (A^{-1} \cdot A) = \sigma (A)^{-1} \cdot \sigma (A)$.
Let $\sigma (x) \in \sigma (A^{-1} \cdot A)$.
Then $\sigma (a^{-1}b) \leq \sigma (x)$ for some $a,b \in A$.
Thus $\sigma (a^{-1}b) = \sigma (x \mathbf{d}(a^{-1}b))$.
But $A^{-1} \cdot A$ is a dense filter and so $x \mathbf{d}(a^{-1}b)) \in A^{-1} \cdot A$.
It follows that $x \in A^{-1} \cdot A$ and so we may find $c,d \in A$ such that $c^{-1}d \leq x$.
But then $\sigma (c)^{-1} \sigma (d) \leq \sigma (x)$.
We have therefore proved that  $\sigma (A^{-1} \cdot A) \subseteq \sigma (A)^{-1} \cdot \sigma (A)$.
We now prove the reverse inclusion.
Let $\sigma (x) \in \sigma (A)^{-1} \cdot \sigma (A)$.
Then $\sigma (a)^{-1}\sigma (b) \leq \sigma (x)$ for some $a,b \in A$.
Thus $\sigma (a^{-1}b) = \sigma (x \mathbf{d}(a^{-1}b))$.
But $a^{-1}b \in A^{-1} \cdot A$ a dense filter and so $x \mathbf{d}(a^{-1}b)) \in A^{-1} \cdot A$ giving $x \in A^{-1} \cdot A$. 
Thus $\sigma (x) \in \sigma (A^{-1} \cdot A)$.

Suppose that $A$ and $B$ are dense filters such that $A \cdot B$ exists.
Then it follows by the above that $\sigma (A) \cdot \sigma (B)$ exists.
The proof that $\sigma (A \cdot B) = \sigma (A) \cdot \sigma (B)$ is similar to the above proof.

It remains to show that our bijection is a homeomorphism.
Let $s \in S$.
Then $Z_{s}$ consists of all dense filters that contain $s$.
The fact that $\sigma (Z_{s}) = Z_{\sigma (s)}$ is immediate in one direction,
and the converse follows since if $\sigma (A)$ contains $\sigma (s)$ then $s \in A$ by our
observation at the head of the proof.  
It follows that our map is an open map.
Finally, the inverse image of $Z_{\sigma (s)}$ under our map is precisely $Z_{s}$ and so our map is continuous.
\end{proof}

\subsection{Universal pseudogroups: the idempotent-pure case}

We begin by generalizing the classical notion of a nucleus on a frame \cite{J} to pseudogroups.
Let $S$ be an inverse semigroup.
A function $\nu \colon S \rightarrow S$ is called a {\em nucleus} if it satisfies the following four conditions:

\begin{description}

\item[{\rm (N1)}] $a \leq \nu (a)$ for all $a \in S$.

\item[{\rm (N2)}] $a \leq b$ implies that $\nu (a) \leq \nu (b)$.

\item[{\rm (N3)}] $\nu^{2}(a) = \nu (a)$ for all $a \in S$.

\item[{\rm (N4)}] $\nu (a) \nu (b) \leq \nu (ab)$ for all $a,b \in S$.

\end{description}

The following is a routine derivation from the axioms.

\begin{lemma} Let $\nu$ be a nucleus on an inverse semigroup $S$.
Then
$$\nu (ab) = \nu (a \nu (b)) = \nu ( \nu (a) b  ) = \nu ( \nu (a) \nu (b) ).$$
\end{lemma}

Let $S$ be an inverse semigroup equipped with a nucleus $\nu$.
Define
$$S_{\nu} = \{a \in S \colon \nu (a) = a    \},$$ 
the set of {\em $\nu$-closed} elements of $S$.
On the set $S_{\nu}$ define
$$a \cdot b = \nu (ab).$$
A homomorphism $\theta \colon S \rightarrow T$ is said to be {\em idempotent-pure} if $\theta (s)$ an idempotent implies that $s$ is an idempotent.

\begin{lemma} The structure $(S_{\nu},\cdot)$ is a semigroup and the map $S \rightarrow S_{\nu}$ given
by $a \mapsto \nu (a)$ is a surjective idempotent-pure semigroup homomorphism.
In particular, $(S_{v},\cdot)$ is an inverse semigroup whose natural partial order coincides with the one in $S$.
\end{lemma} 
\begin{proof} The proof that the operation yields a semigroup follows from Lemma~4.14 as does the proof that the map is a semigroup map.
The image of an inverse semigroup under a homomorphism is an inverse semigroup.  
Observe that if $\nu (s) = \nu (t)$ then $s$ and $t$ are bounded above and so are compatible.
Thus the kernel of $\nu$ is a subset of the compatibility relation and so idempotent-pure by \cite{Law2}.

We denote the natural partial order in $S_{\nu}$ temporarily by $\preceq$.
Let $a,b \in S_{\nu}$.
Suppose first that $a \preceq b$.
Then $a = b \cdot a^{-1} \cdot a$.
Thus $a = \nu (ba^{-1}a)$,
It follows that $ba^{-1}a \leq a$.
Observe that $bb^{-1}a = a$ and so $aa^{-1} \leq bb^{-1}$.
Also $ba^{-1}a = a(ba^{-1}a)^{-1}ba^{-1}a$.
Thus $ba^{-1}a = ab^{-1}b$ and so $ba^{-1}$ is an idempotent.
By symmetry $b^{-1}a$ is also an idempotent and so $a$ and $b$ are compatible.
But $aa^{-1} \leq bb^{-1}$ and so $a \leq bb^{-1}a \leq b$, as required.
On the other hand, if $a \leq b$ then $a = ba^{-1}a$.
Thus $a = \nu (a) = \nu (ba^{-1}a)$ and so $a \preceq b$, as required. 
\end{proof}

\begin{proposition} Let $\nu$ be a nucleus on a pseudogroup $S$.
Then $S_{\nu}$ is a pseudogroup and the natural map from $S$ to $S_{\nu}$ is an idempotent-pure pseudogroup homomorphism.
\end{proposition}
\begin{proof} 
Let $\{a_{i} \colon i \in I \}$ be a compatible subset of $S_{\nu}$.
Since the kernel of $\nu$ is contained in the compatibility relation, 
it is obviously a compatible subset of $S$ and so by assumption has a join $a$ in $S$.
Put $a' = \nu (a) \in S_{\nu}$.
We claim that $a'$ is the join of the $a_{i}$ in $S_{\nu}$.
First $a_{i} \leq a \leq \nu (a) = a'$ and so it is an upper bound of the $a_{i}$.
Suppose that $b \in S_{\nu}$ and $a_{i} \leq b$ for all $i$.
Then $a \leq b$ and so $a' = \nu (a) \leq \nu (b) = b$.  

For clarity we shall denote the join operation on $S_{\nu}$ by $\bigsqcup$.
We shall prove that if $\bigsqcup a_{i}$ exists then $\bigsqcup b \cdot a_{i}$ exists and that
$b \cdot \bigsqcup a_{i} = \bigsqcup ba_{i}$ where $b$ and the $a_{i}$ are all $\nu$-closed elements.
In the semigroup $S$, the existence of $\bigvee a_{i}$ implies the existence of $\bigvee ba_{i}$.
Now $ba_{i} \leq \bigvee ba_{i}$ implies that $\nu (ba_{i}) \leq \nu (\bigvee ba_{i})$.
Thus the  $\nu (ba_{i})$ are pairwise compatible and so $\bigvee \nu (ba_{i})$ exists.
It follows that $\bigsqcup b \cdot a_{i}$ exists.
Now
$$b \cdot \sqcup a_{i} = \nu (\bigvee ba_{i})$$
and
$$\bigsqcup b \cdot a_{i} = \nu (\bigvee \nu (ba_{i})).$$
Since $ba_{i} \leq \nu (ba_{i})$ it is immediate that
$\nu (\bigvee ba_{i}) \leq \nu (\bigvee \nu (ba_{i}))$.
To prove the reverse inequality, we start with
$ba_{i} \leq \bigvee ba_{i}$ and so
$\nu (ba_{i}) \leq \nu (\bigvee ba_{i} )$ from which the inequality readily follows.

It remains to show that the map $S \rightarrow S_{\nu}$ given by $a \mapsto \nu (a)$ is a pseudogroup map.
Suppose that $a_{i}$ is a compatible set of elements in $S$.
We need to prove that $\nu (\bigvee a_{i} ) = \bigsqcup \nu (a_{i})$.
Since $a_{i} \leq \nu (a_{i})$ we have that $\bigvee a_{i} \leq \bigvee \nu (a_{i})$
and so $\nu (\bigvee a_{i} ) \leq \bigsqcup \nu (a_{i})$.
The proof of the reverse inequality starts with $a_{i} \leq \bigvee_{i} a_{i}$,
and the desired inequality then follows readily.
\end{proof}

We may characterize the pseudogroup morphisms that may be described by means of nuclei.

\begin{theorem} 
Surjective idempotent-pure pseudogroup morphisms may be described by means of nuclei.
\end{theorem}
\begin{proof}
Let $\theta \colon S \rightarrow T$ be a surjective idempotent-pure pseudogroup morphism.
For each $t \in T$, the inverse image $\theta^{-1} (t)$ is a compatible subset of $S$ since $\theta$ is idempotent-pure.
Define $\theta_{\ast} \colon T \rightarrow S$ by 
$$\theta_{\ast} (t) = \bigvee \{ s \in S \colon  \theta (s) \leq t \}.$$
Define $\nu \colon S \rightarrow S$ by $\nu (s) = \theta_{\ast} (\theta (s))$.
Observe that $\theta_{\ast}$ is an order-preserving map and that
$s \leq \theta (\theta_{\ast} (s))$ for all $s \in S$ 
and
$\theta (\theta_{\ast}(t)) = t$ for all $t \in T$ since $\theta$ is assumed surjective. 
It therefore follows that $\theta = \theta \theta_{\ast} \theta$ 
and 
$\theta_{\ast} = \theta_{\ast} \theta \theta_{\ast}$.
We claim that $\nu$ is a nucleus on $S$.
The proofs that (N1), (N2) and (N3) hold are straightforward.
The proof of (N4) follows from the fact that multiplication distributes over compatible joins.

It remains to show that $(S_{\nu},\cdot)$ is isomorphic to $T$.
Let $s,t \in S_{\nu}$.
Then
$$\theta (s \cdot t) = \theta (\nu (st)) = \theta \theta_{\ast} \theta (st) = \theta (st) = \theta (s) \theta (t).$$
If $s,t \in S_{\nu}$ and $\theta (s) = \theta (t)$.
Then $\theta_{\ast}\theta (s) = \theta_{\ast} \theta (t)$ and so $\nu (s) = \nu (t)$ giving $s = t$.
Finally, let $t \in T$.
Then there exists $s \in S$ such that $\theta (s) = t$.
Then $\theta \theta_{\ast} \theta (s) = \theta \theta_{\ast} (t)$ giving $\theta (\nu (s)) = t$.
\end{proof}



A coverage $\mathcal{C}$ on an inverse semigroup $S$ is said to be {\em idempotent-pure} if
$X \in \mathcal{C}(a)$ and $X \subseteq E(S)$ implies that $a$ is an idempotent.
We shall now show how to construct nuclei on the pseudogroup $C(S)$ using idempotent-pure coverages.

Let $A$ be a subset of the inverse semigroup $S$.
The subset $A$ is said to be {\em $\mathcal{C}$-closed} if $X \subseteq A$ and $X \in \mathcal{C}(x)$ implies that $x \in A$.
Define $\overline{A}$ by $x \in \overline{A}$ if and only if there exists $X \subseteq A$ such that $X \in \mathcal{C}(x)$.

\begin{lemma} \mbox{}
\begin{description}

\item[{\rm (a)}] If $\mathcal{C}$ is separated then it is idempotent-pure, and $\overline{s^{\downarrow}} = s^{\downarrow}$ for all $s \in S$.

\item[{\rm (b)}] Let $\mathcal{C}$ be an arbitrary idempotent-pure coverage on the inverse semigroup $S$. 
\begin{enumerate}

\item Let $A$ be a compatible order ideal. Then  $\overline{A}$ is a $\mathcal{C}$-closed compatible order ideal.

\item $\overline{A}$ is equal to the intersection of all $\mathcal{C}$-closed compatible order ideals that contain $A$.

\item If $E,F \subseteq E(S)$ are $\mathcal{C}$-closed order ideals of the semilattice of idempotents then so too is $EF$.


\end{enumerate}

\end{description}
\end{lemma}
\begin{proof} (a) Let $E \in \mathcal{C}(x)$ where $E \subseteq E(S)$.
By (R) and (MS) we have that $Ex^{-1} \in \mathcal{C}(xx^{-1})$.
But if $e \in E$ then $e \leq x$ and so $e = ex = xe$.
Thus $ex^{-1} = exx^{-1} = e$.
It follows that $Ex^{-1} = E$ and so by assumption $x = xx^{-1}$.

Let $X \in \mathcal{C}(a)$ where $X \subseteq s^{\downarrow}$.
We prove that $\mathcal{C}(sa^{-1}a) \cap \mathcal{C}(a) \neq \emptyset$ from which we get $a = sa^{-1}a$ and so $a \leq s$.
Now $X \subseteq s^{\downarrow}$ implies that $X = s\mathbf{d}(X)$.
But $\mathbf{d}(X) \in \mathcal{C}(a^{-1}a)$ and so $s\mathbf{d}(X) \in \mathcal{C}(Sa^{-1}a)$.

(b1) We show first that $\overline{A}$ is $\mathcal{C}$-closed.
Let $X \subseteq \overline{A}$ be such that $X \in \mathcal{C}(a)$.
Let $x \in X$.
Then either $x \in A$ or $x \in \overline{A}$.
If the latter then there exists $A_{x} \subseteq A$ such that $A_{x} \in \mathcal{C}(x)$.
If the former then put $A_{x} = \{x \} \in \mathcal{C}(x)$ by (R).
Put $B = \bigcup_{x \in X} A_{x} \subseteq A$.
By (T), we have that $B \in \mathcal{C}(x)$.
Thus $x \in \overline{A}$.

Next we show that $\overline{A}$ is an order ideal.
Let $x \in \overline{A}$ where $X \in \mathcal{C}(x)$ and suppose that $y \leq x$.
Then $y = xy^{-1}y$.
But by (MS), we have that $Xy^{-1}y \in \mathcal{C}(y)$.
But $A$ is an order ideal and so $Xy^{-1}y \subseteq A$.
Thus $y \in \overline{A}$, as required.

Finally, we show that $\overline{A}$ is a compatible subset.
Let $a,b \in \overline{A}$ where $X,Y \subseteq A$ are such that $X \in \mathcal{C}(a)$ and $Y \in \mathcal{C}(b)$.
Then by (I) and (MS), we have that $X^{-1}Y \in \mathcal{C}(a^{-1}b)$
But $A$ is a compatible set and so $X^{-1}Y$ consists entirely of idempotents.
It follows by our assumption on the coverage that $a^{-1}b$ is an idempotent.
Similarly $ab^{-1}$ is an idempotent.
Thus $a$ and $b$ are compatible, as required.

(b2) This is immediate.

(b3) Let $g \in \overline{EF}$. By assumption, $g$ is an idempotent.
There exists $X \subseteq EF$ such that $X \in \mathcal{C}(g)$.
But $EF \subseteq E,F$.
Thus $g \in E$ and $g \in F$ and so $g \in EF$, as required.


\end{proof}

Let $A$ and $B$ be subsets of $S$.
Define the following sets
$$A^{-1}B = \{s \in S \colon As \subseteq B \}
\text{ and }
BA^{-1} = \{s \in S \colon sA \subseteq B \}.$$

\begin{lemma} Let $B$ be a $\mathcal{C}$-closed order ideal.
Then for any $A$, we have that $A^{-1}B$ is a $\mathcal{C}$-closed order ideal, and dually.
\end{lemma}
\begin{proof}
Let $s \in A^{-1}B$.
Then by definition $As \subseteq B$.
Let $t \leq s$ and let $a \in A$.
Then $at \leq as$. 
But $as \in B$ and $B$ is an order ideal and so $at \in B$.
It follows that $t \in A^{-1}B$ and so $A^{-1}B$ is an order ideal.
Let $X \subseteq A^{-1}B$ where  $X \in \mathcal{C}(x)$.
Then for each $x_{i} \in X$ we have that $Ax_{i} \subseteq B$.
It follows that for each $a \in A$, 
we have that 
$aX \in \mathcal{C}(ax)$ and $aX \subseteq B$.
But $B$ is $\mathcal{C}$-closed and so $ax \in B$ for every $a \in A$.
It follows by definition that $x \in A^{-1}B$, as required. 
\end{proof}

Let $\mathcal{C}$ be an idempotent-pure coverage on the inverse semigroup $S$.
Denote by $C(S,\mathcal{C})$ the set of $\mathcal{C}$-closed elements of $C(S)$.
Define $\iota \colon S \rightarrow C(S,\mathcal{C})$ by $\iota (s) = \overline{(s^{\downarrow})}$.

\begin{theorem} Let $\mathcal{C}$ be an idempotent-pure coverage on the inverse semigroup $S$.
Then the map $A \mapsto \overline{A}$ defines a nucleus on $C(S)$,
and $\iota \colon S \rightarrow C(S,\mathcal{C})$ is a $\mathcal{C}$-cover-to-join map 
which is universal amongst such maps to pseudogroups.
\end{theorem}
\begin{proof} It is clear that axioms (N1), (N2) and (N3) hold.
It remains to show that (N4) holds.
Let $A,B \in C(S)$.
We prove that $\overline{A} \, \overline{B} \subseteq \overline{AB}$.
Let $C$ be any  $\mathcal{C}$-closed compatible order ideal containing $AB$.
Thus $AB \subseteq C$.
It follows that $B \subseteq A^{-1}C$.
By the above lemma, $A^{-1}C$ is a $\mathcal{C}$-closed order ideal.
Now $B \subseteq \overline{B}$ and so $B = B \cap \overline{B} \subseteq A^{-1}C \cap \overline{B}$.
Observe that the intersection of a $\mathcal{C}$-closed compatible order ideal
and a $\mathcal{C}$-closed order ideal is a $\mathcal{C}$-closed compatible order ideal.
Thus $A^{-1}C \cap \overline{B}$ is a $\mathcal{C}$-closed compatible order ideal containing $B$.
It follows that it must contain $\overline{B}$.
Hence $\overline{B} \subseteq A^{-1}C$. 
Thus $A\overline{B} \subseteq C$.
A dual argument shows that $\overline{A} \, \overline{B} \subseteq C$, as required.

We now come to the proof of the second claim.
We show first that $\iota$ is a $\mathcal{C}$-cover-to-join map.
Let $X \in \mathcal{C}(a)$.
Clearly $\iota (b) \leq \iota (a)$ for all $b \in X$.
Suppose that $A \in C(S,\mathcal{C})$ such that $\iota (b) \leq A$ for all $b \in X$.
Then $X \subseteq A$.
But $A$ is a $\mathcal{C}$-closed subset and so $a \in A$ from which it follows that $\iota (a) \leq A$.
We have therefore proved that $\iota (a) = \bigvee_{b \in X} \iota (b)$, as required.

Let $\theta \colon S \rightarrow T$ be any $\mathcal{C}$-cover-to-join map to a pseudogroup.
Define $\bar{\theta} \colon  C(S,\mathcal{C}) \rightarrow T$ by
$\bar{\theta}(A) = \bigvee_{a \in A} \theta (a)$.
Given $s \in S$ we show that
$$\theta (s) = \bigvee_{x \in \iota (s)} \theta (x).$$
Clearly
$\theta (s) \leq \bigvee_{x \in \iota (s)} \theta (x)$.
To prove the reverse inequality we argue as follows.
Let $x \in \overline{s^{\downarrow}}$.
Then there exists $X \in \mathcal{C}(x)$ where $X \subseteq s^{\downarrow}$.
But $\theta$ is a $\mathcal{C}$-cover-to-join map and so $\theta (x) = \bigvee_{y \in X} \theta (y)$.
However $y \leq s$ and so $\theta (y) \leq \theta (s)$.
It follows that $\theta (x) \leq \theta (s)$, as required. 
\end{proof}

The above theorem tells us how to construct pseudogroups under certain circumstances, we now turn to the related groupoids.

\begin{theorem} Let $\mathcal{C}$ be an idempotent-pure coverage on the inverse semigroup $S$.
Then the poset of $\mathcal{C}$-filters on $S$ is order isomorphic to the poset of completely prime filters on $C(S,\mathcal{C})$
and this induces a homeomorphism between the groupoids $\mathsf{G}_{\mathcal{C}}(S)$ and $\mathsf{G}(C(S,\mathcal{C}))$.
\end{theorem}
\begin{proof} We begin by setting up the bijection.

Let $F$ be a $\mathcal{C}$-filter in $S$.
Define 
$$F^{u} = \{A \in C(S,\mathcal{C}) \colon A \cap F \neq \emptyset \}.$$
We prove that $F^{u}$ is a completely prime filter in $C(S,\mathcal{C})$.
Let $A,B \in F^{u}$.
Then $A \cap F \neq \emptyset$ and $B \cap F \neq \emptyset$.
Thus we may find elements $f_{1} \in F \cap A$ and $f_{2} \in F \cap B$.
But $f_{1},f_{2} \in F$ implies that there exists $f \in F$ such that $f \leq f_{1},f_{2}$.
 Furthermore, $A$ and $B$ are order ideals.
Thus $f \in A \cap B$.
It follows that $f \in F \cap A \cap B$ and so $A \cap B \in F^{u}$.
Now let $A \in F^{u}$ and $A \leq B$.
Then $A \cap F \neq \emptyset$ and $A \subseteq B$ and so $B \cap F \neq \emptyset$ giving $B \in F^{u}$.
Finally, let $\bigsqcup A_{i} \in F^{u}$.
Then $\overline{\bigcup A_{i} } \in F^{u}$.
Thus  $\overline{\bigcup A_{i} } \cap  F \neq \emptyset$.
It follows that there exists $a \in  \overline{\bigcup A_{i}}$ such that $a \in F$.
By definition, there exists $X \in \mathcal{C}(a)$ such that $X \subseteq \bigcup A_{i}$.
But $a \in F$ and  $X \in \mathcal{C}(a)$ implies that there exists $x \in F \cap X$ since $F$ is a $\mathcal{C}$-filter.
But $x \in X \subseteq \bigcup A_{i}$ and so $x \in A_{i}$.
It follows that $A_{i} \in F^{u}$, as required.

Let $P$ be a completely prime filter in  $C(S,\mathcal{C})$.
Define
$$P^{d} = \{s \in S \colon \overline{s^{\downarrow}} \in P \}.$$
We prove that $P$ is a $\mathcal{C}$-filter in $S$.
We show first that $P^{d}$ is a directed set. 
Let $s,t \in P^{d}$.
Then $\overline{s^{\downarrow}}, \overline{t^{\downarrow}} \in P$.
But $P$ is a filter in a pseudogroup and so  $\overline{s^{\downarrow}} \wedge \overline{t^{\downarrow}} \in P$.
Now in a pseudogroup we have that
$$\overline{s^{\downarrow}} \wedge \overline{t^{\downarrow}}
=
\bigsqcup \overline{u^{\downarrow}}$$
where $\overline{u^{\downarrow}} \leq \overline{s^{\downarrow}}, \overline{t^{\downarrow}}$.
But $P$ is a completely prime filter and so $\overline{u^{\downarrow}} \in P$ for some such $u$.
Now sadly, we cannot deduce that $u \leq s,t$ so we need to do some more work.
We have that $u \in \overline{s^{\downarrow}}$ and $u \in \overline{t^{\downarrow}}$.
Thus there is a subset $\{s_{i} \colon i \in I \} \subseteq s^{\downarrow}$ such that $\{s_{i} \colon i \in I \} \in \mathcal{C}(u)$
and a subset $\{t_{j} \colon j \in J \} \subseteq t^{\downarrow}$ such that $\{t_{j} \colon j \in J \} \in \mathcal{C}(u)$.
We therefore have that
$$\overline{u^{\downarrow}} = \bigsqcup_{i} \overline{s_{i}^{\downarrow}}$$
and
$$\overline{u^{\downarrow}} = \bigsqcup_{j} \overline{t_{j}^{\downarrow}}.$$
Observe that it immediately follows that each $s_{i}$ is compatible with each $t_{j}$.
By \cite{R3}, we may write
$$\overline{u^{\downarrow}} = \bigsqcup_{i,j} \overline{s_{i}^{\downarrow}} \wedge \overline{t_{j}^{\downarrow}}.$$
But we now use again the fact that $P$ is completely prime to deduce that 
$\overline{s_{i}^{\downarrow}} \wedge \overline{t_{j}^{\downarrow}} \in P$ for some $i$ and $j$.
However, we ascertained above that $s_{i}$ is compatible with $t_{j}$.
It follows that $s_{i} \wedge t_{j} \leq s,t$ exists and that 
$\overline{(s_{i} \wedge t_{j})^{\downarrow}} = \overline{s_{i}^{\downarrow}} \wedge \overline{t_{j}^{\downarrow}}$.
Thus $s_{i} \wedge t_{j} \in P^{d}$ and is below both $s$ and $t$.
It is immediate that $P^{d}$ is closed upwards.
It remains to show that it is a $\mathcal{C}$-filter. 
Let $x \in P^{d}$ where $X \in \mathcal{C}(x)$.
Then $\overline{x^{\downarrow}} = \bigsqcup_{y \in X} \overline{y^{\downarrow}} \in P$ and so $\overline{y^{\downarrow}} \in P$,
since $P$ is completely prime, and so $y \in P^{d}$, as required.

It is immediate that $P^{d}$ is closed upwards.
It remains to show that it is a $\mathcal{C}$-filter. 
Let $x \in P^{d}$ where $X \in \mathcal{C}(x)$.
Then $\overline{x^{\downarrow}} = \bigsqcup_{y \in X} \overline{y^{\downarrow}} \in P$ and so $\overline{y^{\downarrow}} \in P$,
since $P$ is completely prime, and so $y \in P^{d}$, as required.

It remains to show that the above two operations are mutually inverse.
We begin by showing that $F = (F^{u})^{d}$.
If $a \in F$ then $\overline{a^{\downarrow}} \in F^{u}$ and so $a \in (F^{u})^{d}$.
Thus $F \subseteq (F^{u})^{d}$.
Let $s \in (F^{u})^{d}$.
Then $\overline{s^{\downarrow}} \in F^{u}$.
Thus $\overline{s^{\downarrow}} \cap F \neq \emptyset$.
It follows that there is $X \in \mathcal{C}(x)$ such that $X \subseteq s^{\downarrow}$.
But $F$ is a $\mathcal{C}$-filter and so there exists $y \in X$ such that $y \in F$.
But $y \leq s$ and so $s \in F$, as required.

Next we show that $P = (P^{d})^{u}$.
Let $A \in (P^{d})^{u}$.
Then $A \cap P^{d} \neq \emptyset$.
Let $a \in A \cap P^{d}$.
Then $\overline{a^{\downarrow}} \in P$ and $\overline{a^{\downarrow}} \leq A$.
Thus $A \in P$.
We have shown that $(P^{d})^{u} \subseteq P$.
To prove the reverse inclusion let $A \in P$.
We have that $A = \bigsqcup_{a \in A} \overline{a^{\downarrow}}$.
But $P$ is completely prime and so $\overline{a^{\downarrow}} \in P$ for some $a \in A$.
Thus $a \in P^{d}$ and $A \in (P^{d})^{u}$, as required.

We have therefore set up a bijection between the two groupoids.
We show now that this bijection is a functor.
We show first that  $(F^{-1} \cdot F)^{u} = (F^{u})^{-1} \cdot F^{u}$.
Let $A \in (F^{-1} \cdot F)^{u}$.
Then there is $f \in F$ such that $f^{-1}f \in A$ and $f \in F$.
But $X = \overline{f^{\downarrow}} \in F^{u}$ and $X^{-1}X \leq A$, as required.
The proof of the reverse inclusion is straightforward.

Let $F$ and $G$ be two $\mathcal{C}$-filters in $S$ such that the product $F \cdot G$ is defined.
Let $A \in (F \cdot G)^{u}$.
Then $fg \in A$ for some $f \in F$ and $g \in G$.
Observe that $\overline{(f^{\downarrow})} \cdot \overline{(g^{\downarrow})} = \overline{(fg)^{\downarrow}}$.
It follows that $A \in F^{u} \cdot G^{u}$.
To prove the reverse inclusion let $A \in F^{u} \cdot G^{u}$.
Then $XY \leq A$ where $X \in F^{u}$ and $Y \in G^{u}$.
Thus $X \cap F \neq \emptyset$ and $Y \cap G \neq \emptyset$.
Let $f \in X \cap F$ and $g \in Y \cap G$.
But then $fg \in A$ and $fg \in F \cdot G$.
Thus $A \in (F \cdot G)^{u}$.

It follows that the two groupoids are isomorphic.
It remains to show that this isomorphism induces a homeomorphism.
The basic open sets in $\mathsf{G}_{\mathcal{C}}(S)$ have the form $Z_{s}$ where $s \in S$.
We claim that the image of this set under the map $F \mapsto F^{u}$ is the set $X_{t}$ where $t = \overline{s^{\downarrow}}$.
Let $s \in F$ where $F$ is a $\mathcal{C}$-filter.
Then $F^{u}$ is, as we have seen, a completely prime filter.
But $\overline{s^{\downarrow}} \cap F \neq \emptyset$ and so $t \in F^{u}$.
Conversely, if $A \in X^{t}$ then $s \in A^{d}$ and $(A^{d})^{u} = A$.
Thus our isomorphism is an open map.
Finally, consider the open subset $X_{A}$ where $A \in C(S,\mathcal{C})$.
Then $X_{A} = \bigcup X_{t}$ where $t = \overline{s^{\downarrow}}$ and $s \in A$.
Thus we need only determine the inverse images of the set $X_{t}$ where $t = \overline{s^{\downarrow}}$ for some $s \in S$.
But this is just the set $Z_{s}$.
It follows that our mapping is continuous and open and so it is a homeomorphism.
\end{proof}

A coverage $\mathcal{C}$ is said to be of {\em finite type} if the elements of $\mathcal{C}(a)$ are finite sets for all $a$.

\begin{lemma} Let $\mathcal{C}$ be an idempotent-pure coverage on $S$ of finite type.

\begin{enumerate}

\item The finite elements of  $C(S,\mathcal{C})$ are the elements of the form $\overline{A}$
where $A$ is a compatible order ideal of the form $A = \{a_{1}, \ldots, a_{m} \}^{\downarrow}$.

\item The product of finite idempotents in  $C(S,\mathcal{C})$ is a finite idempotent.

\item Every idempotent is a join of finite idempotents.

\item $C(S,\mathcal{C})$ is coherent.

\end{enumerate}
\end{lemma}
\begin{proof} (1) Let  $A = \{a_{1}, \ldots, a_{m} \}^{\downarrow} \in C(S)$
and let $\overline{A} \leq \bigsqcup_{j} B_{j}$ where the $B_{j}$ are $\mathcal{C}$-closed elements of $C(S)$.
Then $A$ is a subset of the $\mathcal{C}$-closure of $\bigcup_{j} B_{j}$.
It follows that each $a_{i}$ is in the closure of a finite union of some of $B_{j}$'s.
Thus $A$ is contained in the closure of a finite union of the $B_{j}$'s under our assumption that the coverage is of finite type.
Thus $\overline{A} \leq \bigsqcup_{j=1}^{p} B_{j}$ for some finite $p$ and suitable relabelling.

Let $A = \{a_{i} \colon i \in I \}^{\downarrow}$ be an element of $C(S)$ such that $\overline{A}$ is finite.  
Then 
$\overline{A} = \bigsqcup_{i} \overline{a_{i}^{\downarrow}} 
= 
\bigsqcup_{i=1}^{p} \overline{a_{i}^{\downarrow}} 
$
for some finite number of $a_{i}$ and suitable relabelling.

(2) This follows from Lemma~4.18(3).

(3) If $E$ is an idempotent in  $C(S,\mathcal{C})$ then $E = \bigsqcup_{e \in E} \overline{e^{\downarrow}}$.

(4) This follows by the results (2) and (3) above and Lemma~3.4.
\end{proof}

If $\mathcal{C}$ is a coverage of finite type, then we can clearly define what it means for homomorphisms
$\theta \colon S \rightarrow T$ to distributive inverse semigroups to be {\em $\mathcal{C}$-cover-to-join maps}. 
We denote the inverse subsemigroup of finite elements in  $C(S,\mathcal{C})$ by $K(S,\mathcal{C})$.
By the above lemma, we have now proved the following.

\begin{theorem} Let $\mathcal{C}$ be an idempotent-pure coverage of finite type on the inverse semigroup $S$.
Then the pseudogroup $C(S,\mathcal{C})$ is coherent.
The map $\iota \colon S \rightarrow K(S,\mathcal{C})$
is a $\mathcal{C}$-cover-to-join map and is universal amongst such maps to distributive inverse semigroups.
\end{theorem}

\subsection{Main results}

We shall now apply the theory we have developed to three coverages: the trivial, the dense and the tight.

We begin with the trivial coverage of Examples~4.1(1).
This is a coverage of finite-type and every inverse semigroup is separative with respect to this coverage.
The theorem below is a direct application of Theorems~4.20, 4.21 and 4.24 as well as Proposition~3.19.

\begin{theorem} Let $S$ be an inverse semigroup equipped with the trivial coverage.
\begin{enumerate}

\item The universal pseudogroup in this case is just $\mathsf{P}(S) = C(S)$
along with the map $\pi \colon S \rightarrow \mathsf{P}(S)$
where universality is with respect to arbitrary semigroup homomorphisms to pseudogroups.

\item  The pseudogroup $\mathsf{P}(S)$ is coherent. 
Put $\mathsf{D}(S) = K(\mathsf{P}(S))$ along with the map $\delta \colon S \rightarrow \mathsf{D}(S)$. 
This is the universal distributive inverse semigroup where universality is with respect to arbitrary 
semigroup homomorphisms to distributive inverse semigroups.

\item The groupoid of all filters on $S$ is homeomorphic to $\mathsf{G}(\mathsf{P}(S))$.

\item The groupoid of all filters on $S$ is homeomorphic to $\mathsf{G}_{P}(\mathsf{D}(S))$.

\end{enumerate}
\end{theorem}

One application of the semigroup $\mathsf{D}(S)$ is in connection with the weak meet condition.

\begin{proposition} Let $S$ be an inverse semigroup.
Then $S$ satisfies the weak meet condition if and only if $\mathsf{D}(S)$ is an inverse $\wedge$-semigroup.
\end{proposition}
\begin{proof} The natural partial order in $\mathsf{D}(S)$ is subset-inclusion.
Observe that
$s^{\downarrow} \cap t^{\downarrow}$ is a compatible order ideal.
It is finitely-generated if and only if the weak meet condition holds.
Suppose that $\mathsf{D}(S)$ is an inverse $\wedge$-semigroup.
Then in particular $s^{\downarrow} \wedge t^{\downarrow}$ exists and must equal a finitely-generated compatible order ideal. 
It follows that $s^{\downarrow} \cap t^{\downarrow}$ is a finitely generated compatible order ideal.
Conversely, suppose that $S$ satisfies the weak meet condition.
Then $s^{\downarrow} \wedge t^{\downarrow}$ always exists.
An arbitrary element of $\mathsf{D}(S)$ can be written $A = s_{1}^{\downarrow} \vee \ldots \vee s_{m}^{\downarrow}$.
If $B = t_{1}^{\downarrow} \vee \ldots \vee t_{n}^{\downarrow}$ is any element of  $\mathsf{D}(S)$
then it can be verified that $A \wedge B = \bigvee_{i,j} (s_{i}^{\downarrow} \wedge t_{j}^{\downarrow})$.
\end{proof}

We now turn to the dense coverage and its properties.

\begin{theorem} Let $S$ be an inverse semigroup equipped with the dense coverage.
Then the universal pseudogroup $\pi \colon S \rightarrow \mathsf{P}_{\mathcal{D}}(S)$ exists.
\end{theorem}
\begin{proof} We put  $\mathsf{P}_{\mathcal{D}}(S) = C(\mathbf{S},\mathcal{D})$ and define $\pi (s) =  \iota (\mathbf{s})$. 
We now show that this has the required universal properties.
Let $\theta \colon S \rightarrow T$ be a dense map to a pseudogroup.
By Lemma~4.9 and Proposition~4.12, there is a dense map $\phi \colon \mathbf{S} \rightarrow T$ such that  $\theta = \phi \sigma$.
By Corollary~4.11, the semigroup $\mathbf{S}$ is separative with respect to the dense coverage.
By Theorem~4.20, there is a pseudogroup morphism $\bar{\phi} \colon \mathbf{S} \rightarrow C(\mathbf{S},\mathcal{D})$ such that $\phi = \bar{\phi} \iota$.
If we put $\pi = \iota \sigma$ and rechristen $\bar{\phi}$ as $\bar{\theta}$,
we have that $\theta = \bar{\theta} \pi$.
The uniqueness of $\bar{\theta}$ follows from the fact that each element of $\mathsf{P}_{\mathcal{D}}(S)$ is a join of elements of the form $\pi (s)$.
\end{proof}

We shall denote the universal pseudogroup with respect to the dense coverage by $\mathsf{P}_{d}(S)$
and call it the {\em dense pseudogroup} of $S$.

\begin{theorem} Let $S$ be an inverse semigroup equipped with the dense coverage.
Then $\mathsf{G}_{\mathcal{D}}(S)$ is homeomorphic to the groupoid $\mathsf{G}(\mathsf{P}_{d}(S))$.
\end{theorem}
\begin{proof} 
This is immediate by Proposition~3.13 and Theorem~4.21.
\end{proof}

We shall now briefly outline why we chose the name {\em dense} for this coverage.
With each inverse semigroup $S$, we may associate a left cancellative category 
$\mathsf{C}(S)$, that we call its {\em associated Leech category}.
The elements of this category are of the form $(e,s)$ where $e$ is a non-zero idempotent
and $s$ a non-zero element such that $ss^{-1} \leq e$.
The product $(e,s)(f,t)$ is defined iff $s^{-1}s = f$ in which case its product defined to be $(e,st)$.
In \cite{LS}, we proved that there as a bijective correspondence between what we called
{\em Ehresmann topologies} on $S$ and {\em Grothendieck topologies} on $\mathsf{C}(S)$.
One significant Grothendieck topology is the {\em dense topology} also referred to
as the {\em $\neg \neg$-topology}; see pp~115, 273 of \cite{MM}.
The dense coverage $\mathcal{D}$ can be used to construct an Ehresmann topology $\mathcal{D}^{\downarrow}$:
for each idempotent $e \in S$ if $A \in \mathcal{D}(e)$ then $A^{\downarrow} \in \mathcal{D}^{\downarrow}(e)$.
We call this the {\em dense Ehresmann topology}. 
Using the correspondence proved as Theorem~3.8 of \cite{LS}, we may easily 
show that the dense coverage on $S$ gives rise to the dense Ehresmann topology on $S$ 
which is associated with the dense topology on $\mathsf{C}(S)$.
This motivates our final result about the dense pseudogroup.
First we need a simple but illuminating lemma.

\begin{lemma} Let $E$ be a meet semilattice.
Let $F \subseteq E$ be an order ideal.
If $e \in (F^{\perp})^{\perp}$ is non-zero then $e \rightarrow eF$.
\end{lemma}
\begin{proof}
Clearly $eF \subseteq F$ and $eF \subseteq e^{\downarrow}$.
Let $0 \neq f \leq e$.
Suppose that $fF = 0$.
Then $f \in F^{\perp}$.
But $e \in (F^{\perp})^{\perp}$ and so $ef = 0$ but this is a contradiction.
Thus $fF \neq 0$, as required.
\end{proof}

Let $e$ be an idempotent in a pseudogroup $S$.
Define
$$e^{\ast} = \bigvee_{f \wedge e = 0} f.$$
By construction we have that $e \wedge e^{\ast} = 0$.

\begin{theorem} 
The dense pseudogroup is boolean.
\end{theorem}
\begin{proof} We may assume that we are working with a semigroup $S$ separated with respect to the dense coverage.
To show that the pseudogroup is boolean, it is enough to show that $E^{\ast \ast} = E$ for any densely-closed order ideal of the semilattice of idempotents of $S$.
Since $E \subseteq E^{\ast \ast}$, it is enough to prove that $E^{\ast \ast} \subseteq \overline{E} = E$.
Let $f \in E^{\ast \ast}$.
Then $\overline{f^{\downarrow}} \wedge E^{\ast} = 0$.
We shall prove that $f \in (E^{\perp})^{\perp}$ giving $f \in \overline{E}$ by Lemma~4.26.
Let $e \in E^{\perp}$.
Then $e \in E^{\ast}$ and so $\overline{f^{\downarrow}} \wedge \overline{e^{\downarrow}} = 0$.
It follows that $fe = 0$.
\end{proof}

We now turn to the tight coverage.
A similar argument to that used in the proof of Theorem~4.26 shows that there is a universal pseudogroup
$\mathsf{P}_{\mathcal{T}}(S)$ with respect to the tight coverage.
However, by Lemma~4.22 this pseudogroup is coherent.
Put $\mathsf{D}_{t}(S) = K(\mathsf{P}_{\mathcal{T}}(S))$ and 
denote by $\delta \colon S \rightarrow \mathsf{D}_{t}(S)$ the corresponding map.
We have therefore proved result (1) below, a generalization of Theorem~2.1 the key result of \cite{Law4}.
The proof of (2) below follows from Proposition~4.13, Theorem~4.21 and Proposition~3.19.

\begin{theorem} Let $S$ be an inverse semigroup.
\begin{enumerate}

\item There is a distributive inverse semigroup $\mathsf{D}_{t}(S)$ 
and a tight map $\delta \colon S \rightarrow \mathsf{D}_{t}(S)$ which is universal for tight maps from $S$
to distributive inverse semigroups.

\item The \'etale groupoids $\mathsf{G}_{\mathcal{T}}(S)$ and $\mathsf{G}_{P}(\mathsf{D}_{t}(S))$ are homeomorphic.

\item If $S$ satisfies the weak meet condition then $\mathsf{D}_{t}(S)$ is an inverse $\wedge$-semigroup.

\end{enumerate}
\end{theorem}
\begin{proof} It remains to prove (3).
Let $s,t \in S$ such that $s^{\downarrow} \cap t^{\downarrow} = \{a_{1}, \ldots, a_{m} \}^{\downarrow}$.
We prove that 
$\mathbf{s}^{\downarrow} \wedge \mathbf{t}^{\downarrow} = \overline{\{\mathbf{a_{1}}, \ldots, \mathbf{a_{m}} \}^{\downarrow}} = A$.
We have that
$\{a_{1}, \ldots, a_{m}\}^{\downarrow} \subseteq s^{\downarrow}$.
Thus 
$\{\mathbf{a_{1}}, \ldots, \mathbf{a_{m}} \}^{\downarrow} \subseteq \mathbf{s}^{\downarrow}$
and so
$A \subseteq \mathbf{s}^{\downarrow}$ by Lemma~4.18(a).
Similarly $A \subseteq \mathbf{t}^{\downarrow}$.
Let $\mathbf{a}^{\downarrow} \subseteq \mathbf{s}^{\downarrow}, \mathbf{t}^{\downarrow}$.
Then $\mathbf{a} \leq \mathbf{s}, \mathbf{t}$.
Thus $\sigma (a) = \sigma (sa^{-1}a) = \sigma (ta^{-1}a)$. 
Let 
$X \in \mathcal{T}(ta^{-1}a) \cap \mathcal{T}(a)$
and
$Y \in \mathcal{T}(a) \cap \mathcal{T}(sa^{-1}a)$.
By Lemma~4.2(3),(4), we have that $X \wedge Y \in \mathcal{T}(a), \mathcal{T}(ta^{-1}a), \mathcal{T}(sa^{-1}a)$.
Observe that $X \wedge Y \subseteq \{a_{1}, \ldots, a_{m} \}^{\downarrow}$.
By Lemma~4.10(3), we have that $\sigma (a) = \bigvee \sigma (X) \wedge \sigma (Y) \leq A$.
The result now follows because every element of $\mathsf{D}_{t}(S)$ 
is a finite join of elements of the form $\mathbf{a}^{\downarrow}$.
\end{proof}

We call the distributive inverse semigroup $\mathsf{D}_{t}(S)$ the {\em tight completion} of $S$.

\section{Applications}

\subsection{Paterson's universal groupoid}

This groupoid was introduced by Paterson \cite{P} in functional-analytic terms, 
and described by means of filters by the second author \cite{L} and in a slightly revised form by the first author \cite{LMS}
in collaboration with Stuart Margolis and Ben Steinberg.
Let $S$ be an inverse semigroup.
If we equip it with the simplest coverage of Examples~4.12(1), then the associated filters are just all filters.
We denote by $U_{s}$ the set of all filters containing $s$ and by $\omega$ the set of all such sets.
The groupoid of all filters equipped with this topology will be denoted by $\mathsf{G}_{\omega}(S)$.
This is the groupoid and topology that one would expect on the basis of the theory developed in this paper so far.
See Theorem~4.24.

We shall now define a different topology on the same groupoid.
For $x$, $x_1,\ldots, x_n \in S$ with $x_1,\ldots, x_n \leq x$, 
the set $U_{x; x_1, \ldots, x_n} $ is defined by 
$$ U_{x; x_1, \ldots, x_n} = U_x \cap U_{x_1}^c \cap\ldots \cap U_{x_n}^c$$
where $U_{x}^c$ is the complement of $U_x$ in the groupoid.
Let $\Omega$ be the set of all such subsets.
We call this topology the {\em patch topology}.
With respect to this topology, the groupoid is called the {\em universal groupoid} and is denoted by means of $\mathsf{G}_{u}(S)$ where `u' stands for `universal'.
In \cite{L}, the second author proves that $\Omega$ consists of compact-open sets and that the space of identities is hausdorff.
It therefore follows by Lemma~3.13 that we have the following,  using the terminology of Section~3.5

\begin{theorem} 
The universal groupoid of an inverse semigroup is a weakly boolean groupoid. 
\end{theorem}

It is of interest to know when the universal groupoid is actually boolean as opposed to being just weakly boolean.
To do this, we shall use another description of the topology of   $\mathsf{G}_{u}(S)$   based on an alternative way of regarding filters.
With each filter $X$, we may associate a function $j(X) \colon S \rightarrow \{0,1 \}$ 
defined by 
$$\mbox{$j(X)(x)=1$ if $x \in X$ and $j(X) (x)=0$ otherwise}.$$
If we regard $\Gamma$ as a poset with respect to its natural partial order and $\mathbf{2} = \{0,1\}$ as a meet-semilattice then $j(X)$ is an order-preserving map
with the additional property that if $j(X)(a) = 1$ and $j(X)(b) = 1$ there exists $c \leq a,b$ such that $j(X)(c) = 1$.
Conversely, if $\alpha \colon S \rightarrow \mathbf{2}$ is an order-preserving map satisfying the additional property then
the set $X$ of all elements $a \in \Gamma$ such that $\alpha (a) = 1$ is a filter.
In this way, we set up a bijective correspondence between filters on $S$
and certain kinds of order-preserving maps $S \rightarrow \mathbf{2}$.
In any event, the map $j$ is injective. 
Thus we can consider   $\mathsf{G}_{u}(S)$  as a subset of $\{0,1\}^{S}$. 
The space $\{0,1\}^{S}$ carries a canonical topology: namely, the  product topology where $\{0,1\}$ is given the discrete topology.
The equivalence of (1) and (5) below was first proved by Steinberg \cite{Stei}.

\begin{proposition} \label{theorem-characterization-hausdorff} Let $S$ be an inverse semigroup. Then the following assertions are equivalent:

\begin{enumerate}

\item The groupoid $\mathsf{G}_{u}(S)$ is hausdorff.

\item  Every compact set in  $\mathsf{G}_{u}(S)$  is closed.

\item The map $j \colon \mathsf{G}_{u}(S) \longrightarrow \{0,1\}^{S}$ is continuous.

\item The topology of $\mathsf{G}_{u}(S)$ is induced from the map $j \colon \mathsf{G}_{u}(S) \longrightarrow \{0,1\}^{S}$.   

\item The semigroup satisfies the weak meet condition.

\end{enumerate}
\end{proposition}
\begin{proof} The topology of $\mathsf{G}_{u}(S)$
obviously comes from the restriction of certain open subsets of $\{0,1\}^{S}$ to   $\mathsf{G}_{u}(S)$ . 
Thus the equivalence of (3) and (4) is clear.
We are now going to show $ (4) \Longrightarrow (1) \Longrightarrow (2)\Longrightarrow (5) \Longrightarrow (4)$.
Observe that the implications $(4)\Longrightarrow (1)$ and $(1)\Longrightarrow (2)$ are immediate.

$(2)\Longrightarrow (5)$.  Let $x_1,\ldots, x_n \in  S$ and assume that $U = \bigcap_{j=1}^n U_{x_j}$ is non-empty. 
The sets $U_{x_{j}}$ are compact and therefore by assumption they are closed and so their union is closed.
But $U \subseteq U_{x_{1}}$ is a closed subset of a compact set and so is itself compact.
Let $X\in U$.
Then there exists $x \in X$ such that $x \leq x_1,\ldots, x_n$
and $X \in U_x \subseteq U$.
As $X\in U$ was arbitrary, we can express $U$ as the union of all these $U_x$. 
But $U$ is compact, and so we deduce (5).

$(5)\Longrightarrow  (4)$.  The proof is similar to the proof given in the case where the semigroup is an inverse $\wedge$-semigroup
described in \cite{L}. 
To show that the topology induced by $j$  agrees with the original topology, we have
to show the following: for arbitrary $X\in  \mathsf{G}_{u}(S) $ and $x_1,\ldots, x_n,
y_1,\ldots, y_m$ with  $X\in U_{x_1}\cap\ldots\cap U_{x_n} \cap
U_{y_1}^c\cap \ldots \cap U_{y_m}^c$, there exist $z_1,\ldots, z_k \leq z$ with
\begin{equation} \label{stern}X\in U_{z;z_1,\ldots,z_k}\subseteq U_{x_1}\cap\ldots\cap U_{x_n} \cap
U_{y_1}^c\cap \ldots \cap U_{y_m}^c.
\end{equation}
By the weak meet condition we have
$$U_{x_1}\cap\ldots\cap U_{x_n} = U_{a_1}\cup \ldots \cup U_{a_m}$$
for suitable $a_1,\ldots, a_l \in \Gamma$ with $a_j < x_k$ for all $j$ and $k$.

Let  $X\in U_{x_1}\cap\ldots\cap U_{x_n}$.
Then there exists then a $k\in \{1,\ldots, m\}$
with $X\in U_{a_k}$. Without loss of generality we can assume $k=1$. 
 This gives
$$X \in U_{a_1} \subseteq U_{x_1} \cap \ldots \cap U_{x_n}.$$ 
We now have to deal  with the $U_{y_j}^c$.  For each $j$ we consider $U_{y_{j}} \cap U_{a_1}$. By the weak meet condition, 
we have
$$ U_{y_{j}} \cap U_{a_1} = U_{b_1^{(j)}} \cup \ldots U_{b_{l(j)}^{j}}$$
with suitable $b_1^{(j)},\ldots, b_{l(j)}^{(j)} \leq y_j, a_1$. As $X$ does not belong to $U_{y_j}$, $j=1,\ldots, m$, we obtain
$$ X\notin U_{b_k^{(j)}}$$
for  $k=1,\ldots, l(j)$. Putting this together we arrive at
$$X \in U_{a_1;b_1^{(1)},\ldots b_{l(1)}^{(1)}, \ldots, b_1^{(m)},\ldots, b_{l(m)}^{(m)}}\subseteq
U_{x_1}\cap\ldots\cap U_{x_n} \cap
U_{y_1}^c\cap \ldots \cap U_{y_m}^c$$
 and we are done.
\end{proof}

We therefore have the following result.

\begin{corollary} The universal groupoid of an inverse semigroup is boolean if and only if the inverse semigroup
satisfies the weak meet condition.
\end{corollary}

But what exactly is the universal groupoid?
Paterson gave one answer in terms of topological groupoids \cite{P}.
We shall give another in terms of inverse semigroups.

Let $S$ be a distributive inverse semigroup and $\mathsf{G}_{P}(S)$ its associated coherent groupoid of prime filters.
Recall that a basis is given by $\pi = \{Y_{s} \colon s \in S \}$ where $Y_{s}$ is the set of all prime filters containing $s$.
Define $\Pi = \{ Y_{s} \cap Y_{t}^{c} \colon s,t \in S, \,  t \leq s \}$.
It is convenient to define $Y_{s;t} = Y_{s} \cap Y_{t}^{c}$ where $t \leq s$.

\begin{lemma} Let $S$ be a distributive inverse semigroup.
\begin{enumerate}

\item With the above definition, $\Pi$ is the basis for a topology on the groupoid $\mathsf{G}_{P}(S)$.

\item If $S$ is also an inverse $\wedge$-semigroup, 
then the set $Y_{s} \cap Y_{t}^{c}$, where $s$ and $t$ are arbitrary, is equal to an element of  $\Pi$.

\item If $S$ is weakly boolean then the topologies generated by $\pi$ and $\Pi$ are the same.

\item $Y_{s \vee t ; u \vee v} = Y_{s; (u \vee v)s^{-1}s} \cup Y_{t ; (u \vee v)t^{-1}t}$.

\end{enumerate}
\end{lemma}
\begin{proof} (1) Suppose that $(Y_{s} \cap Y_{t}^{c}) \cap (Y_{u} \cap Y_{v}^{c}) \neq \emptyset$ where $t \leq s$ and $v \leq u$.
Let $P$ be any prime filter belonging to this set.
Then $s,u \in P$ and $t,v \notin P$.
It follows that there exists $z \in P$ such that $z \leq s,u$.
Put $a = zt^{-1}t \leq t$ and $b = zv^{-1}v \leq v$.
Since $a,b \leq z$ we have that $c = a \vee b$ exists.
Consider now the set $Y_{z} \cap Y_{c}^{c}$.
Suppose that $c \in P$.
Then since $P$ is a prime filter either $a \in P$ or $b \in P$.
Suppose that $a \in P$.
Then $t \in P$ which is a contradiction.
It follows that $P \in Y_{z} \cap Y_{c}^{c}$.
Let $Q \in Y_{z} \cap Y_{c}^{c}$.
Then $a,b \in Q$.
Suppose that $t \in Q$.
Then since $a = zt^{-1}t$ we would have $a \in Q$ which is a contradiction since $Q$ omits $c$.
Thus $t \notin Q$.
Similarly, $v \notin Q$.
It follows that $P \in Y_{z} \cap Y_{c}^{c} \subseteq (Y_{s} \cap Y_{t}^{c}) \cap (Y_{u} \cap Y_{v}^{c})$.
Thus $\Pi$ is the basis for a topology on $\mathsf{G}_{P}(S)$. 

(2) It is easy to check that
$Y_{s} \cap Y_{t}^{c}
=
Y_{s} \cap Y_{s \wedge t}^{c}$.

(3) Let $t \leq s$. Then $\mathbf{d}(t) \leq \mathbf{d}(s)$.
By assumption, the semilattice of idempotents of $S$ is a boolean algebra and so there exists an idempotent $e$ such that
$e \leq \mathbf{d}(s)$,
$e \wedge \mathbf{d}(t) = 0$.
and $\mathbf{s} = e \vee \mathbf{d}(t)$.
By distributivity we have that $s = se \vee t$ and it is easy to check that $se \wedge t = 0$.
Put $u = se$.
We prove that $Y_{s} \cap Y_{t}^{c} = Y_{u}$.
Let $X$ be a prime filter in  $Y_{s} \cap Y_{t}^{c}$.
Then $s \in X$.
But $X$ is a prime filter and so either $u \in X$ or $t \in X$.
We cannot have the latter and so $u \in X$ and $X \in Y_{u}$.
Conversely, suppose that $X \in Y_{u}$.
Then $s \in X$ and we cannot have $t \in X$ because $u \wedge t = 0$.

(4) Observe first that $(u \vee v)s^{-1}s \leq s$.
We have that $(u \vee v)s^{-1}s \leq (s \vee t)s^{-1}s$.
But $(s \vee t)s^{-1}s = s \vee ts^{-1}s$ and since $s$ and $t$ are compatible $ts^{-1}$ is an idempotent.
Thus $ts^{-1}s \leq s$.
It follows that $s \vee ts^{-1}s = s$, as required.

Let $P \in Y_{s \vee t ; u \vee v}$.
Thus $s \vee t \in P$.
Since $P$ is a prime filter either $u \in P$ or $v \in P$.
Suppose that $u \in P$ and $(u \vee v)s^{-1}s \in P$.
Then $u \vee v \in P$ which is a contradiction.
It is clear that the lefthand-side is contain in the righthand-side.
Now let $P \in Y_{s; (u \vee v)s^{-1}s} \cup Y_{t ; (u \vee v)t^{-1}t}$.
Suppose that $P \in Y_{s; (u \vee v)s^{-1}s}$.
Then $s \in P$ implies that $s \vee t \in P$.
Suppose that $u \vee v \in P$.
Then $(u \vee v)s^{-1}s \in P$ which is a contradiction.
It follows that the righthand-side is contained in the lefthand-side. 
\end{proof}

It follows from (2) above that the topology generated by $\Pi$ is a generalization of the {\em patch topology} described on page~72 of \cite{J}.

\begin{proposition} Let $S$ be a distributive inverse semigroup.
\begin{enumerate}

\item $Y_{s;t}Y_{u;v} = Y_{su;sv \vee tu \vee tv}$.

\item $Y_{s;t}^{-1} = Y_{s^{-1} ; t^{-1}}$.

\item The groupoid $\mathsf{G}_{P}(S)$ equipped with the patch topology is weakly boolean.

\end{enumerate}
\end{proposition}
\begin{proof} (1) Observe that since $t \leq s$ and $v \leq u$ we have that $sv, tu, tv \leq su$.
Thus the join $sv \vee tu \vee tv$ exists.
Let $X \in Y_{s;t}$v and $Y \in  Y_{u;v}$.
Then $su \in XY$.
Suppose that $sv \vee tu \vee tv \in XY$.
But $XY$ is a prime filter and so either $sv \in XY$ or $tu \in XY$ or $tv \in XY$ but each of these is ruled out.
Thus $Y_{s;t}Y_{u;v} \subseteq Y_{su;sv \vee tu \vee tv}$.
We now prove the reverse inclusion.
Put $w = sv \vee tu \vee tv$. 
Let $Z \in  Y_{su;w}$.
Put $X = (s (uZ^{-1} \cdot Z u^{-1} )^{\uparrow})^{\uparrow}$ 
and 
$Y = (uZ^{-1} \cdot Z)^{\uparrow}$.
Then $X \cdot Y = Z$ and by construction both $X$ and $Y$ are prime filters.
Clearly $s \in X$ and $u \in Y$.
Suppose $t \in X$.
Then we may write $sueu^{-1} \leq t$ for some $e \in E(Z^{-1} \cdot Z)$.
Thus $(su)e(u^{-1}s^{-1})(su) \leq ts^{-1}su$.
But $(su)e(u^{-1}s^{-1})(su) \in Z$ and since $t \leq s$ we have that $ts^{-1}su = tu$.
Thus we deduce that $tu \in Z$, which is a contradiction.
Suppose $v \in Y$.
Then $ue \leq v$ for some $e \in E(Z^{-1} \cdot Z)$.
Thus $sue \leq sv$
But $(su)e \in ZZ^{-1}Z = Z$.
Thus $sv \in Z$, which is a contradiction.
It follows that $X \in Y_{s;t}$ and $Y \in Y_{u;v}$ as required.

(2) We have that $X \in Y_{s;t}$ if and only if $X^{-1} \in Y_{s^{-1};t^{-1}}$
and we know that $X$ is a prime filter if and only if $X^{-1}$ is a prime filter.

(3) We prove first that $\mathsf{G}_{P}(S)$ is also an etale groupoid with respect to the patch topology.
By (2) above, it is clear that the inversion map is continuous with respect to the patch topology.

We now show that the multiplication map is continuous with respect to the patch topology.
We use the same idea as in the proof of Proposition~4.3 of \cite{L}.
We prove that $m^{-1} (Y_{s;t})$ is open.
Let $U\cdot V = X \in Y_{s;t}$.
Then $s \in U\cdot V$ and so we may find $u \in U$ and $v \in V$ such that $uv \leq s$.
In addition, we may assume that $u^{-1}u = vv^{-1}$ since $U^{-1} \cdot U = V \cdot V^{-1}$.
Consider now $Y_{u;tt^{-1}u}$ and $Y_{v;vt^{-1}t}$.
We claim that $U \in Y_{u;tt^{-1}u}$ and $V \in Y_{v;vt^{-1}t}$.
Suppose that $tt^{-1}u \in U$.
Then $tt^{-1}uv \in Z$.
But $tt^{-1}uv \leq tt^{-1}s = t$ which implies that $t \in Z$, which is a contradiction.
Suppose that $vt^{-1}t \in V$.
Then $uvt^{-1}t \in Z$ and $uvt^{-1}t \leq st^{-1}t = t$, which is a contradiction.
It remains to show that if $A \in Y_{u;tt^{-1}u}$ and $B \in Y_{v;vt^{-1}t}$ and $A \cdot B$ exists then $A \cdot B \in Y_{s;t}$.
Since $u \in A$ and $v \in B$ we have that $uv \in A \cdot B$ and $uv \leq s$ gives $s \in A \cdot B$.
Suppose that $t \in A \cdot B$.
Then $ab \leq t$ where $a \in A$ and $b \in B$ and we may assume that $a^{-1}a = bb^{-1}$ since $A \cdot B$ exists.
From the fact that $A$ is a filter and $u \in A$ we may also assume that $a \leq u$.
Similarly we may assume that $b \leq v$.
We now calculate
$$a = au^{-1}u = a(bb^{-1})u^{-1}u = (ab)b^{-1}u^{-1}u \leq t(b^{-1}u^{-1})u \leq ts^{-1}u = tt^{-1}ss^{-1}u$$
which is equal to $tt^{-1}u$.
But this contradicts the fact that $tt^{-1}u \notin A$.

We need to show that the set of identities is an open subspace with respect to the patch topology.
Let $P$ be an identity in the groupoid $\mathsf{G}_{P}(S)$. 
Then it is an inverse subsemigroup by Lemma~2.3 and so contains an idempotent $e$, say.
Then $P \in Y_{e}$ which consists entirely of idempotent prime filters.
By (1) above and an argument similar to the one employed in the proof of Proposition~2.8,
we deduce that the product of open sets is open.
We have therefore proved that our groupoid is \'etale.

To show that the groupoid with the patch topology is weakly boolean it is enough,
by Section~3.4, to show that the set of identities of $\mathsf{G}_{P}(S)$,
with respect to the subspace topology, is a boolean space.
The case where the distributive lattice has a top element is dealt with in Proposition~II.4.5 of \cite{J}.
We deal with the general case here and use the same idea as that used in the proof of Proposition~5.2.
The idea for this proof goes back to Section~4.3 of Paterson \cite{P}.
Observe first that we can restrict our attention to the distributive lattice $E(S)$ since idempotent filters are determined by the idempotents they contain.
We give the set $\mathbf{2}^{E(S)}$ the discrete topology.
By the Axiom of Choice this is a compact space.
A subbase for this topology is given by sets of the form $U_{e}$ and $U_{e}^{c}$ where $e \in E(S)$ and
$$U_{e} = \{\theta \colon E(S) \rightarrow \mathbf{2} \colon \theta (e)(1) = 1\}
\mbox{ and }
U_{e}^{e} = \{\theta \colon E(S) \rightarrow \mathbf{2} \colon \theta (e)(1) = 1\}.$$
With each filter, we may associate an element $j(X)$ of $\mathbf{2}^{E(S)}$.
This is an injective function.
If we restrict our attention to the prime filters on $E(S)$ then the function $j$ restricts to a bijection with a closed subset of $\mathbf{2}^{E(S)}$.
We denote this closed subset by $\mathbf{P}$.
The simple argument to prove that it really is closed is made explicit in \cite{Law3a}.
 The restriction of the product topology to $\mathbf{P}$ is hausdorff and has a basis of compact-open sets.
The topology generated by $\Pi$ restricted to the prime filters on $E(S)$ is easily seen to be 
homeomorphic to the restriction of the product topology to the set $\mathbf{P}$.
Thus we have shown that the space $\mathsf{G}_{P}(S)$ equipped with the patch topology is boolean.
\end{proof}

If $S$ is a distributive inverse semigroup then the groupoid $\mathsf{G}_{P}(S)$ equipped with the patch topology is called the 
{\em weak booleanization} of the groupoid $\mathsf{G}_{P}(S)$ with its usual topology.
If $S$ is a distributive $\wedge$-semigroup then $\mathsf{G}_{P}(S)$ equipped with the patch topology is called the {\em booleanization}.

Let $S$ be an arbitrary inverse semigroup.
Put $\mathsf{B}(S) = \mathsf{KB}(\mathsf{G}_{u}(S))$.

\begin{theorem}[First booleanization] \mbox{}
\begin{enumerate}

\item The universal groupoid $\mathsf{G}_{u}(S)$ of an inverse semigroup $S$ is homeomorphic to the weak booleanization of the groupoid $\mathsf{G}_{P}(\mathsf{D}(S))$.
In particular, the weakly boolean inverse semigroup $\mathsf{B}(S)$ has the property that the universal groupoid $\mathsf{G}_{u}(S)$ 
is homeomorphic to the groupoid $\mathsf{G}_{P}(\mathsf{B}(S))$.

\item If the inverse semigroup $S$ also satisfies the weak meet condition then this weak booleanization is in fact a booleanization.

\end{enumerate}
\end{theorem}
\begin{proof} (1) By Theorem~4.24(4), the groupoid of all filters of $S$ with the usual topology is homeomorphic to the groupoid $\mathsf{G}_{P}(\mathsf{D}(S))$,
where $\mathsf{D}(S)$ is a distributive inverse semigroup, with its usual topology.
Observe that the set $U_{x}$ is mapped to the set $Y_{s}$ where $s = x^{\downarrow}$
and that the inverse image of $Y_{s \vee t}$ where $s = x^{\downarrow}$ and $t = y^{\downarrow}$ is the set $U_{x} \cup U_{y}$.  
The groupoid  $\mathsf{G}_{P}(\mathsf{D}(S))$ equipped with the patch topology is homeomorphic to the universal groupoid $\mathsf{G}_{u}(S)$.
To see this start with the set $U_{x;x_{1}, \ldots, x_{n}}$ where $x_{1}, \ldots, x_{n} \leq x$.
Then if we put $s = x^{\downarrow}$ and $t = \{x_{1}, \ldots, x_{n}\}^{\downarrow}$ then in $\mathsf{D}(S)$ we have that $t \leq s$.
The set $U_{x;x_{1}, \ldots, x_{n}}$ is therefore mapped by the homeomorphism to the set $Y_{s;t}$.
By Lemma~5.4(4), the inverse image of arbitrary sets of the form $Y_{s;t}$ is a union of sets of the form $U_{x;x_{1}, \ldots, x_{n}}$.
It follows that the given homeomorphism is also a homeomorphism when each groupoid is regarded with respect to the new topology. 
The proof of the second claim follows by the extension of Theorem~3.25 or the restriction of Theorem~3.17 to weakly boolean inverse semigroups.
 
(2) This follows by (1) above, Proposition~4.25 and Theorem~3.25.
\end{proof}

Let $S$ be an inverse semigroup.
Define $\beta \colon S \rightarrow \mathsf{B}(S)$ by $\beta (s) = U_{s}$.
This homomorphism has an interesting property: if $P$ is a prime filter in $\mathsf{B}(S)$ then $\beta^{-1}(P)$ is a filter in $S$.
The key observation that proves this is the following.
Suppose that $\emptyset \neq U_{x;x_{1}, \ldots, x_{n}} \subseteq U_{y}$.
Then $x^{\uparrow} \in  U_{x;x_{1}, \ldots, x_{n}}$ giving $x^{\uparrow} \in U_{y}$.
Thus $x \leq y$ and so $U_{x} \subseteq U_{y}$.
It follows that $U_{x;x_{1}, \ldots, x_{n}} \subseteq U_{x} \subseteq U_{y}$.
We now describe the universal property enjoyed by this homomorphism.

\begin{theorem}[Second booleanization] Let $S$ be an inverse semigroup and let $\theta \colon S \rightarrow T$ be a homomorphism to a weakly boolean inverse semigroup
with the property that the inverse image under $\theta$ of each prime filter in $T$ is a filter in $S$.
Then there is a unique homomorphism of distributive inverse semigroups $\bar{\theta} \colon \mathsf{B}(S) \rightarrow T$ such that $\bar{\theta} \beta = \theta$.
\end{theorem}
\begin{proof} We prove uniqueness first.
Any morphism of distributive inverse semigroups $\gamma \colon \mathsf{B}(S) \rightarrow T$
is determined by its values on the elements of the form $U_{x;x_{1}, \ldots, x_{m}}$.
Observe now that the elements $U_{x_{1}}, \ldots, U_{x_{m}}$ are pairwise compatible in $\mathsf{B}(S)$ and so their union 
$\bigcup_{i=1}^{m} U_{x_{i}}$ is a well-defined element.
In $\mathsf{B}(S)$, we have that
$$U_{x} = U_{x;x_{1}, \ldots, x_{m}} \vee \left( \bigcup_{i=1}^{m} U_{x_{i}} \right)
\mbox{ and }
U_{x;x_{1}, \ldots, x_{m}} \wedge \left( \bigcup_{i=1}^{m} U_{x_{i}} \right) = 0.$$
The map $\gamma$ preserves this information, essentially by Lemma~3.27, and so $\gamma$ is determined by its values on
elements of the form $U_{y}$.
This is enough to force uniqueness.

We now turn to existence.
By Theorem~4.24(2), 
there is a homomorphism of distributive inverse semigroups $\phi \colon \mathsf{D}(S) \rightarrow T$ such that $\phi \delta = \theta$.
The inverse image under $\phi$ of a prime filter is a prime filter.
We equip $\mathsf{G}_{P}(\mathsf{D}(S))$ with the patch topology.
By Theorem~5.6(1), we may identify $\mathsf{B}(S)$ with the set of compact-open bisections of $\mathsf{G}_{P}(\mathsf{D}(S))$.
 
We may therefore assume that we are given a homomorphism $\phi \colon D \rightarrow T$
from a distributive inverse semigroup $D$ to $T$ such that the inverse images of prime filters under $\phi$ are prime filters.
Put $B(D)$ equal to the set of all compact-open bisections of $\mathsf{G}_{P}(D)$ with respect to the patch topology of Proposition~5.5.
Let $B(S)'$ be the subset of $B(D)$ consisting of those elements of the form $Y_{s;t}$.
By Proposition~5.5(1) and (2), it is an inverse subsemigroup.
Let $\upsilon \colon D \rightarrow B(D)$ be the map $s \mapsto Y_{s}$.
This is an injective map by Lemma~3.11(2).

Define 
$$\bar{\phi} (Y_{s;t}) = \phi (s) \setminus \phi (t).$$
To show this is well-defined, suppose that 
$Y_{s;t} = Y_{u;v}$ and that $\phi (s) \setminus \phi (t) \neq \phi (u) \setminus \phi (v)$.
By Lemma~3.11(2), we may find a prime filter $P$ in $T$ which contains, without loss of generality, 
$\phi (s) \setminus \phi (t)$ but not $\phi (u) \setminus \phi (v)$.
By assumption $Q = \phi^{-1}(P)$ is a prime filter in $D$.
Clearly $s \in Q$ and $t \notin Q$.
Thus $Q \in Y_{s;t}$.
By assumption $Q \in Y_{u;v}$.
But this implies that $\phi (u) \in P$ and $\phi (v) \notin P$ and so $\phi (u) \setminus \phi (v) \in P$, which is a contradiction.
Thus this map is well-defined.
By Lemma~3.28, it is a homomorphism.
It follows that we have defined a homomorphism from $B(D)'$ to the semigroup $T$.

It remains to extend the map $\bar{\phi}$ to the whole of $B(D)$.
The first step is to observe that if $Y_{s;t}$ and $Y_{u;v}$ are compatible in  $B(D)$
then $\phi (s)\setminus \phi (t)$ is compatible with $\phi (u) \setminus \phi (v)$
by Lemma~3.28 and Proposition~5.5(2).
Let $P$ be a prime filter containing $(\phi (s)\setminus \phi (t))^{-1}(\phi (u) \setminus \phi (v))$.
Then $\phi^{-1}(P)$ is a prime filter that contains $s^{-1}u$ and omits $s^{-1}v \vee t^{-1}u \vee t^{-1}v$.
Thus $\phi^{-1}(P) \in Y_{s;t}^{-1}Y_{u;v}$.
By assumption, this contains only idempotent filters.
Thus $\phi^{-1}(P)$ is an idempotent filter and so $P$ is an idempotent prime filter.
Thus all the prime filters containing $(\phi (s)\setminus \phi (t))^{-1}(\phi (u) \setminus \phi (v))$ are idempotent.
By Lemma~3.29, it follows that $(\phi (s)\setminus \phi (t))^{-1}(\phi (u) \setminus \phi (v))$ is an idempotent.
We deduce that $\phi (s)\setminus \phi (t)$ and $\phi (u) \setminus \phi (v)$ are compatible.

Given a finite union of compatible elements $\bigcup_{i=1}^{m} Y_{s_{i} ; t_{i}}$ in $B(D)$
define 
$$\bar{\phi}(\bigcup_{i=1}^{m} Y_{s_{i} ; t_{i}}) = \bigvee_{i=1}^{m} \phi (s_{i}) \setminus \phi (t_{i}).$$
We show that this map is well-defined.
Suppose that
$$\bigcup_{i=1}^{m} Y_{s_{i} ; t_{i}}
=
\bigcup_{j=1}^{n} Y_{u_{j} ; v_{j}}.$$
We need to prove that 
$$\bigvee_{i=1}^{m} \phi (s_{i}) \setminus \phi (t_{i})
=
\bigvee_{j=1}^{n} \phi (u_{j}) \setminus \phi (v_{j}).$$
By Lemma~3.11(2) it is enough to show that the set of prime filters containing the
lefthand-side is the same as the set of prime filters containing the righthand-side.
This is straightforward to prove given our assumption on $\phi$.
By construction, the map $\bar{\phi}$ is a morphism of distributive inverse semigroups.
\end{proof}

\begin{remark}
{\em The above theorem was inspired by the calculations on pp~190--191 of Paterson's book \cite{P}}.
\end{remark}

\subsection{The closure of the space of ultrafilters in the universal groupoid}

The goal of this section is to show that our notion of tight filters and tight maps coincides with that of Exel \cite{Exel1,Exel2}.

\begin{lemma} Let $S$ be a semigroup and let $F$ be a proper filter.
\begin{enumerate}

\item $A$ is a tight filter (respectively, ultrafilter) if and only if $A^{-1} \cdot A$ is a tight filter (respectively, ultrafilter).

\item  $A$ is an idempotent tight filter (respectively, ultrafilter) in $S$ if and only if $E(A)$ is a tight filter (respectively, ultrafilter) in $E(S)$.

\end{enumerate}
\end{lemma}
\begin{proof} (1) The claim concerning tight filters follows from Lemma~4.3.
The proof of our claim concerning ultrafilters follows from Proposition~2.13 
\cite{Law3} since the argument used there in fact works in any inverse semigroup.

(2) A filter $A$ is an idempotent filter if and only if $A = E(A)^{\uparrow}$.
It is immediate that if $A$ is a tight filter then $E(A)$ is a tight filter since $E(A) \subset A$.
Conversely, suppose that $E(A)$ is a tight filter.
Let $\{a_{1}, \ldots, a_{m} \} \in \mathcal{T}(a)$.
Then 
$\{\mathbf{d}(a_{1}), \ldots, \mathbf{d}(a_{m})\} \in \mathcal{T}(\mathbf{d}(a))$. 
By assumption $\mathbf{d}(a_{i}) \in E(A)$ for some $i$.
Now $a, \mathbf{d}(a_{i}) \in A$ and $A$ is an inverse subsemigroup and so $a_{i} = a \mathbf{d}(a_{i}) \in A$,
as required.

The proof of our claim concerning ultrafilters follows from Proposition~2.13 \cite{Law3} 
since the argument used there in fact works in any inverse semigroup.
\end{proof}

Our first important result about tight filters is the following.

\begin{proposition} 
Every ultrafilter is a tight filter.
\end{proposition}
\begin{proof} Let $A$ be an ultrafilter.
Then $A^{-1} \cdot A$ is an idempotent ultrafilter by Lemma~5.9.
If we can prove that all idempotent ultrafilters are tight then the result will follow from Lemma~5.9 again.
Thus we assume that $A$ is an idempotent ultrafilter and prove that it is tight.
But $A$ is an ultrafilter in the inverse semigroup if and only if $E(A)$ is an ultrafilter in the
semilattice of idempotents.
It follows from Lemma~2.33 of \cite{Law3} that every ultafilter in a semilattice is tight.
Thus $E(A)$ is tight and so $A$ is tight by Lemma~5.9.
\end{proof}

We denote the set of all tight filters of our inverse semigroup $S$ equipped with the restriction of the patch topology on the universal groupoid by $\mathsf{G}_{t}(S)$.
A basis of this topology is given by the  sets
$W_{x; x_1,\ldots x_n} = U_{x;x_1,\ldots x_n} \cap \mathsf{G}_{t}$
for $x_1,\ldots, x_n \leq x \in S$. 
The groupoid of ultrafilters equipped with the restriction of the patch topology on the universal groupoid is denoted by $\mathsf{G}_{m}(S)$.
A basis for this topology is given by the sets  
$V_{x; x_1, \ldots, x_n} = U_{x;x_1,\ldots x_n} \cap \mathsf{G}_{m}$.
By the proposition above, we have the following inclusions
$$\mathsf{G}_{m}(S) \subseteq \mathsf{G}_{t}(S) \subseteq \mathsf{G}_{u}(S).$$ 

\begin{lemma} Let $F$ be a tight filter.
Let $x_{1}, \ldots, x_{n} \leq x$ be such that 
$x \in F$ and $\{x_{1}, \ldots, x_{m} \} \cap F = \emptyset$.
Then $x^{\downarrow} \cap \{x_{1}, \ldots, x_{m} \}^{\perp} \neq 0$.
\end{lemma}
\begin{proof}
Suppose that $x \in F$ and $\{x_{1}, \ldots, x_{m} \} \cap F = \emptyset$
but that $x^{\downarrow} \cap \{x_{1}, \ldots, x_{m} \}^{\perp} = 0$.
Let $0 \neq y \leq x$.
Then by our assumption, $y$ cannot be orthogonal to all $x_{i}$ and so
there exists an $i$ such that $x_{i}^{\downarrow} \cap y^{\downarrow} \neq 0$.
It follows that $a \rightarrow \{x_{1}, \ldots, x_{m} \}$.
But $F$ is a tight filter and so $x_{i} \in F$ for some $i$ which is a contradiction.
\end{proof}

The next results provide the remaining key properties about tight filters.

\begin{proposition} \mbox{}
\begin{enumerate}

\item Let $F$ be a tight filter.
Then every open set in the universal groupoid containing $F$ also contains an ultrafilter.

\item Let $F$ be a filter that is not tight.
Then there is an open set in the universal groupoid of $S$ that contains $F$ but does not contain any ultrafilter.

\end{enumerate}
\end{proposition}
\begin{proof}
(1) Let $F$ be a tight filter and let $U$ be an open set containing $F$.
Then we may find $x$ and elements $x_{1}, \ldots, x_{m} \leq x$ such that
$F \in U_{x;x_{1}, \ldots, x_{m}} \subseteq U$.
Thus $x \in F$ and $\{x_{1}, \ldots, x_{m} \} \cap F = \emptyset$.
By Lemma~5.11, there is a non-zero element $z \leq x$ and orthogonal to all the $x_{i}$.
Let $G$ be an ultrafilter containing $z$.
Then $x \in G$ and $\{x_{1}, \ldots, x_{m} \} \cap G = \emptyset$.
Thus $G \in U_{x;x_{1}, \ldots, x_{m}}$ and so $G \in U$.

(2) Let $F$ be a filter that is not tight.
Then there exists $x \in F$ and a covering $x \rightarrow \{x_{1}, \ldots, x_{m} \} \in \mathcal{T}(x)$ 
such that $F \cap \{x_{1}, \ldots, x_{m}\} = \emptyset$.
Clearly $F \in U_{x;x_{1}, \ldots, x_{m}}$.
However if $G$ were any ultrafilter in this open set we would have $x \in G$ and so, by (1) above,
we would have to have $x_{i} \in G$ for some $i$ which is a contradiction.
We have therefore have found an open set containing $F$ that does not contain any ultrafilter.
\end{proof}

The proof of the following is now immediate by the above and can be viewed as the filter version of a result first proved by Exel in \cite{Exel2}.

\begin{theorem} Let $S$ be an inverse semigroup.
Then the closure of the set of ultrafilters in the universal groupoid of $S$ is precisely the set of tight filters.
\end{theorem}

\subsection{Inverse semigroups in which every tight filter is an ultrafilter }

The first goal of this section is to reconcile Exel's work \cite{Exel1,Exel2} with that of the second author \cite{L}.
This involves determining when $\mathsf{G}_{m}(S) = \mathsf{G}_{t}(S)$.
The following is Proposition~4.4(b) of \cite{L}.

\begin{lemma} 
The identities $\mathsf{G}_{u}(S)_{o}$ form a closed subset of $\mathsf{G}_{u}(S)$.
\end{lemma}

Our next theorem extends the results found in \cite{L}.

\begin{theorem} Let $S$ be an inverse semigroup.
Then the following are equivalent.
\begin{enumerate}

\item $\mathsf{G}_{m}(S)$ is closed in $\mathsf{G}_{u}(S)$.

\item Every tight filter is an ultrafilter.

\item Every idempotent tight filter is an idempotent ultrafilter.

\item $\mathsf{G}_{m}(S)_{o}$ is closed in $\mathsf{G}_{u}(S)_{o}$.

\item For arbitrary $x_1,\ldots, x_n \leq  x \in S$ the set $V_{x;x_1,\ldots, x_n}$ is compact.

\item The tight completion $\mathsf{D}_{t}(S)$ of $S$ is weakly boolean.

\end{enumerate}
\end{theorem}
\begin{proof} The equivalence of (1) and (2) is immediate by Theorem~5.13.
The equivalence of (2) and (3) follows by Lemma~5.4.

Suppose that (3) holds.
The closure of $G_{m}(S)^{(0)}$ is contained in $G_{u}(\Gamma)^{(0)}$
which by assumption is contained in $G_{m}(S)^{(0)}$ and so (4) holds.
Conversely, suppose that (4) holds.
Let $F$ be an idempotent tight filter.
Such a filter belongs to the closure of $G_{m}(S)^{(0)}$ which is assumed to be $G_{m}(S)^{(0)}$ itself.
Thus $F$ is an ultrafilter and so (3) holds.
The equivalence of (4) and (5) was proved in \cite{L}.

We now prove the equivalence of (2) and (6).
The distributive inverse semigroup $\mathsf{D}_{t}(S)$ is weakly boolean if and only if every prime filter is an ultrafilter by Proposition~1.6.
By Proposition~4.13, the poset of tight filters on $S$ is order-isomorphic to the poset of tight filters on $\mathbf{S}$.
By Theorem~4.21, the poset of tight filters on $\mathbf{S}$ is order-isomorphic to the poset of completely prime filters on $C(\mathbf{S},\mathcal{T})$
By Lemma~3.2, the poset of completely prime filters on $C(\mathbf{S},\mathcal{T})$ is order-isomorphic to the poset of 
prime filters on $K(\mathbf{S},\mathcal{T}) = \mathsf{D}_{t}(S)$.
It follows that the poset of tight filters on $S$ is order-isomorphic to the poset of prime filters on $\mathsf{D}_{t}(S)$.
Under this order-isomorphism maximal filters are mapped to maximal filters.
Thus every tight filter on $S$ is an ultrafilter if and only if every ultrafilter on $\mathsf{D}_{t}(S)$
is a prime filter if and only if $\mathsf{D}_{t}(S)$ is weakly boolean.
\end{proof}

An inverse semigroup is said to satisfy the {\em compactness condition} if any of the equivalent conditions of the previous theorem holds.

Our goal now is to find simple sufficient conditions on an inverse semigroup that imply it satisfies the compactness condition.
The following lemma will be useful.

\begin{lemma} Let $S$ be an inverse semigroup with zero and let $e$ and $f$ be idempotents with $f \leq e$.
\begin{enumerate} 

\item $V_{e;f} \neq \emptyset$ if and only if $e^{\downarrow} \cap f^{\perp} \neq 0$.

\item If $V_{e;f} = \bigcup_{i=1}^{m} V_{e_{i}}$ then the $e_{i}$ can be chosen such that $e_{i} \leq e$.

\item If $V_{e;f} = \bigcup_{i=1}^{m} V_{e_{i}}$, where $e_{1}, \ldots, e_{m} \leq e$, then $e \rightarrow \{e_{1}, \ldots, e_{m}, f \}$.

\item If $e \rightarrow \{e_{1}, \ldots, e_{m}, f \}$ where $e_{1}, \ldots, e_{m} \leq e$ then $V_{e;f} \subseteq \bigcup_{i=1}^{m} V_{e_{i}}$.

\end{enumerate}
\end{lemma}
\begin{proof} (1) Let $F \in  V_{e;f}$.
Then $F$ is an ultrafilter such that  $e \in F$ and $f \notin F$.
Every ultrafilter is a tight filter by Proposition~5.10,
and so by Lemma~5.11 there exists $0 \neq i \leq e$ such that $i \wedge f = 0$.
Hence $e^{\downarrow} \cap f^{\perp} \neq 0$.
We now prove the converse.
Let  $0 \neq i \leq e$ such that $i \wedge f = 0$.
By Zorn's Lemma, there is an ultrafilter $F$ containing $i$.
But then $e \in F$ and $f \notin F$ and so $F \in V_{e;f}$, as required.

(2) Suppose that $V_{e;f} = \bigcup_{i=1}^{m} V_{x_{i}}$.
Each ultrafilter in $V_{e;f}$ contains an idempotent and so is itself an idempotent ultrafilter.
We may therefore assume that the $x_{i}$ are all idempotents.
We claim that 
$V_{e;f} = \bigcup_{i=1}^{m} V_{x_{i} \wedge e}$.
Let $F \in V_{e;f}$.
Then by assumption $F \in V_{x_{i}}$ for some $i$.
But $e \in F$ and so $e \wedge x_{i} \in F$.
It follows that $F \in V_{x_{i} \wedge e}$.
Conversely let $F \in  V_{x_{i} \wedge e}$.
Then $e \in F$ and $x_{i} \in F$.
It follows that $F \in V_{e;f}$. 

(3) Suppose that $V_{e;f} = \bigcup_{i=1}^{m} V_{e_{i}}$. 
Let $0 \neq z \leq e$.
If $z \wedge f \neq 0$ then we are done so we may suppose that $z \wedge f = 0$.
Thus $z \in e^{\downarrow} \cap f^{\perp}$.
Let $G$ be any ultrafilter containing $z$.
Then $e \in G$ and $f \notin G$.
It follows that $e_{i} \in G$ for some $i$.
Thus $z \wedge e_{i} \neq 0$, as required.

(4) Suppose that $e \rightarrow \{e_{1}, \ldots, e_{m}, f \}$.
Let $F \in V_{e;f} \neq \emptyset$.
By Proposition~5.10, every ultrafilter is a tight filter and so either $f \in F$ or $e_{i} \in F$ for some $i$.
Since we have excluded the former, it follows that $e_{i} \in F$.
This shows that $V_{e;f} \subseteq \bigcup_{i=1}^{m} V_{e_{i}}$. 
\end{proof}

Condition (1) above is equivalent to saying that the semilattice of idempotents of the inverse semigroup is {\em $0$-disjunctive}.

The following gives a necessary condition for the compactness condition to hold.

\begin{proposition} Let $S$ be an inverse semigroup with zero.
Suppose that for arbitrary idempotents $e,f$ in $S$ with $f \leq e$ either $V_{e;f}$ is empty or
there exist idempotents $e_1,\ldots, e_m \leq e$ with
$$V_{e;f} = \bigcup_{k=1}^m V_{e_k}.$$
Then the compactness condition holds.
\end{proposition}
\begin{proof} We shall prove that every tight idempotent filter is an idempotent ultrafilter and the result will follow by Theorem~5.15.
Suppose to the contrary that $P$ is a tight idempotent filter that is not an ultrafilter.
Then $P \subseteq Q$ where $Q$ is an ultrafilter.

Let $f \in Q$ where $f \notin P$ where $f$ may be chosen without loss of generality as an idempotent.
Let $e \in P$ be any idempotent.
Then $e \wedge f \in Q$ and is not in $P$.
Thus we may find an idempotent $e \in P$ such that $f \leq e$.
Now $P$ is a tight filter that contains $e$ and omits $f$.
It follows by Lemma~5.6 that there is $0 \neq z \leq e$ and $z \wedge f = 0$.
By Lemma~5.16(1), we have that $V_{e;f}$ is non-empty.
Thus by assumption 
$$V_{e;f} = \bigcup_{j=1}^n V_{e_j}$$
where the $e_{1}, \ldots, e_{m} \leq e$.
By Lemma~5.16(3), we have that $e \rightarrow \{e_{1}, \ldots, e_{n}, f\}$.
Now $e \in P$ and $f \notin P$ and $P$ is a tight filter and so $e_{i} \in P$ for some $i$.
But then $e_{i} \in Q$ giving $f,e_{i} \in Q$ and $e_{i} \wedge f = 0$ which is a contradiction.
\end{proof}

\begin{proposition} 
Let $S$ be an inverse semigroup satisfying the weak meet condition.  
Then a subbasis for the topology on $\mathsf{G}_{m}(S)$ is given by the sets $V_x$ where $x\in S$.
\end{proposition}
\begin{proof}  It suffices to show that for arbitrary $X\in \mathsf{G}_{m}(S)$ and $z_1,\ldots,z_n \leq z\in  S$ with $X\in V_{z;z_1,\ldots,z_n}$, 
we have $X\in V_x\subseteq V_{z;z_1,\ldots,z_n}$ for a suitable $x$. 
To that end, let $X$ be an ultrafilter that contains $z$ and omits $z_{1}, \ldots, z_{n}$.
We now use the properties of ultrafilters in inverse semigroups satisfying the weak meet condition described in Section~1.
By Theorem~1.5, there is an elements $x \in X$ such that $x \leq z$ and $x^{\downarrow} \cap z_{i}^{\downarrow} = 0$ for $1 \leq i \leq n$.
Clearly $X \in V_{x}$ and $V_{x} \subset V_{z;z_1,\ldots,z_n}$.
\end{proof}

\begin{theorem}\label{theorem-characterization-compactness}
If the inverse semigroup $S$ satisfies the weak meet condition, then the following are equivalent:

\begin{enumerate}

\item $S$ satisfies the compactness condition.

\item For any  $x \in  S$ the set  $V_x$ is compact.

\item For any $x_{1},\ldots, x_{n} \leq x$ the set $V_{x;x_{1}, \ldots, x_{n}}$ is compact.

\item For arbitrary  idempotents  $e,f_1,\ldots, f_n$ in $S$ with $f_j \leq p$ for $j=1,\ldots, n$, there exist idempotents $e_1,\ldots, e_m \leq e$ with
$$V_{e;f_1,\ldots, f_n} = \bigcup_{k=1}^m V_{e_k}.$$

\item For arbitrary idempotents $e,f$ in $S$ with $f \leq e$, if $V_{e:f} \neq \emptyset$ then there exist idempotents $e_1,\ldots, e_m \leq e$ with
$$V_{e;f} = \bigcup_{k=1}^m V_{e_k}.$$

\item The tight completion $\mathsf{D}_{t}(S)$ of $S$ is boolean.

\end{enumerate}
\end{theorem}
\begin{proof} 

$(1)\Longrightarrow (2)$. It suffices to consider the case where $x = e$ is an idempotent in $S$. 
Then $V_e$ is just the intersection of $U_e$, which is compact by results of the previous section,
and the set of identities of $\mathsf{G}_{m}(S)$, which is closed by assumption. 
Hence $V_e$ is compact.

$(2)\Longrightarrow (3)$. By definition of the topology, the sets $V_x$ are open and hence their complements are closed. 
As closed subsets of compact sets are compact we obtain (3).

$(3)\Longrightarrow (4)$. By the weak meet condition, the topology is generated by the $V_x$ where $x\in S$.  
Hence any set $V = V_{e;f_1,\ldots, f_n}$ can be written as a union of sets of the form $V_x$ with suitable $x \leq e$. 
As $V$ is compact by (3), finitely many of such $x$ suffice.

$(4)\Longrightarrow (5)$. This is clear.

$(5)\Longrightarrow (1)$. This was proved as Proposition~5.17.

$(1)\Longleftrightarrow (6)$. This follows from Theorem~5.15(6) and the fact that if $S$ is a weak meet semigroup then $\mathsf{D}_{t}(S)$ 
is an inverse $\wedge$-semigroup by Theorem~4.30.
\end{proof}

Condition (5) above is a generalization of Theorem~2.22 of \cite{Law4}.

A variant of (4) and (5) was introduced in \cite{L} under the name of trapping condition and shown to imply the compactness condition.
It was then realized in \cite{Law4} that, in the context of suitable semilattices, the trapping condition is actually equivalent to the compactness condition.
Our result above generalizes the corresponding results of \cite{L,Law4}.

Let $S$ be an inverse semigroup.
If $e$ and $f$ are idempotents of $S$ such that $e \leq f$ and if $e \leq i \leq f$ then either $e = i$ or $f = i$,
then we say that $e$ is a {\em 1-step restriction} 
or {\em one-step restriction} of $f$.\footnote{We would usually say that {\em $f$ covers $e$} but in this paper that could cause confusion.}
Inductively, we define $f$ to be an {\em $n$-step restriction} of $e$ if there exists an $(n-1)$-step restriction $i$ of $e$ 
such that  $f$ is a $1$-step restriction of $i$.
Note that we make no assumption that $n$ is unique.
We say that the semilattice of idempotents of $S$ has {\em finite-depth} if whenever $0 \neq e \leq f$
then $e$ is an $n$-step restriction of $f$ for some finite $n$.

\begin{lemma} Let $S$ be an inverse semigroup whose semilattice of idempotents has finite-depth.
If $0 \neq f \leq e$ then $V_{e;f}$ can be written as a finite union of sets of the form $V_{i;j}$ where
$j \leq i$ and $j$ is a $1$-step restriction of $i$.
\end{lemma}
\begin{proof}
Let $f \leq e$ be an $n$-step restriction.
Then there is an idempotent $i$ such that $f \leq i \leq e$ where $f$ is a $1$-step restriction of $i$ and
$i$ is an $n-1$-step restriction of $e$.
Observe that $V_{e;f} = V_{e;i} \cup V_{i;f}$. 
The result follows by iteration.
\end{proof}

A one-step restriction $f$ of $e$ is said to be {\em weakly complemented} 
if there exist one-step restrictions $e_1,e_{2}, \ldots \in f^{\perp}$ of $e$ such that
$g \leq e$ and $g \wedge f = 0$ implies that $g \leq e_{i}$ for some $i$.

We say that the semilattice of idempotents of $S$ is {\em pseudofinite} or {\em locally finite}
if for each $e \in E(S)$ the set $e^{\downarrow} \setminus \{e\}$ is a finitely generated order ideal.
This is equivalent to saying that each idempotent has only finitely many one-step restrictions.

An inverse semigroup $S$ with zero is said to be {\em coarse-grained} 
if its semilattice of idempotents is 
locally finite, 
has finite-depth, 
and is such that any one-step restriction is weakly complemented.

\begin{theorem} 
A coarse-grained inverse semigroup satisfies the compactness condition. 
\end{theorem}
\begin{proof}  
We show that the condition of Proposition~5.17 holds.
By Lemma~5.20, it is enough to deal with the case where $f \leq e$ is a one-step restriction.
We assume that $V_{e;f} \neq \emptyset$.
By assumption, we may find finitely many one-step restrictions $e_{1}, \ldots, e_{m}$ such that
$g \leq e$ and $g \wedge f = 0$ implies that $g \leq e_{i}$ for some $1 \leq i \leq m$.
We will show that 
$$V_{e;f} = \bigcup_{k=1}^m V_{e_k}.$$
Let $X \in V_{e;f}$.
Then $e \in X$ and $f \notin X$.
We know that $E(X)$ is an ultrafilter in $E(S)$ by Lemma~5.9(2).
Thus there exists $i \in E(X)$ such that $i \wedge f = 0$.
Thus we may find $g \in X$, an idempotent, such that $g \leq e$ and $g \wedge f = 0$.
By assumption $g \leq e_{i}$ for some $i$ and so $X \in V_{e_{i}}$.
To prove the reverse inclusion, suppose that $X \in V_{e_{i}}$.  
Then $e_{i} \leq e$ implies that $e \in X$.
But $e_{i} \wedge f = 0$ and so $f \notin X$, as required. 
\end{proof}

\begin{example}{\em  The inverse semigroup constructed from a locally finite graph is coarse-grained.
The semilattice of idempotents of this semigroup is unambiguous \cite{Law4}
and so the set of one-step restrictions of an idempotent is an orthogonal subset.
The semilattice of idempotents has finite-depth by construction and is locally finite by fiat.}
\end{example}

\section{Concluding remarks}

\subsection{Ehresmann's legacy} We begin by expanding slightly our remarks in the Preliminaries concerning the origins of frame theory in pseudogroup theory.
Ehresmann's work on pseudogroups and local structures that Johnstone refers to can be found collected in Partie~II-1 of \cite{E} as paper 47 and, incidently,
is the only paper Ehresmann wrote in his native German.
Its title is {\em Gattungen von lokalen Strukturen}\footnote{`Species of local structures'.}.
The reader will recall from introductory courses in differential geometry that differential manifolds of various complexions
are defined by means of atlases whose changes of charts are required to satisfy certain conditions depending on the type of manifold
being defined. For the constructions to work, the changes of charts need to belong to a so-called {\em pseudogroup of transformations}.
Differential geometers, such as Ehresmann, were well aware in the 1950's that a whole range of local structures, important in differential geometry,
could be defined in the same way: for example, fibrations and foliations.
Furthermore, category theory, developed in the 1940's, pointed the way to describing classes of structures in general.
Ehresmann, in his {\em Gattungen}-paper, attempts to combine these two approaches to describing local structures.
His goal is to describe, in categorical terms, the process of constructing a category of local structures from a pseudogroup.
This goal in fact becomes the basis of the papers collected in Partie~II-1.
Although pseudogroups have frames of idempotents,
frame theory turned its back on pseudogroups and they do not occur at all in either \cite{J} or \cite{MM}.
Ehresmann's work on pseudogroups was generally neglected except in East Germany \cite{R}
and within inverse semigroup theory where his ideas
turned out to be extremely fertile: they form the basis of the book \cite{Law2}.

\subsection{Motivation for the theory} Inverse semigroups arise naturally in the theory of $C^{\ast}$-algebras.
This was first observed by Renault \cite{Ren} and was later developed more explicitly by Paterson \cite{P}. 
The salient idea in this context is that  $C^{\ast}$-algebras can be constructed from topological groupoids which in turn can be constructed from inverse semigroups.
Exel's work \cite{Exel1} is  a prime example of this fruitful approach to constructing $C^{\ast}$-algebras. 
In this way  many interesting $C^{\ast}$-algebras, such as graph algebras and tiling algebras, can be constructed from inverse semigroups.
This  raises the question of the nature of the relationship between inverse semigroups and topological groupoids.
The authors' interest in this question was aroused by Kellendonk's work \cite{Kel1,Kel2} on aperiodic tilings as models of quasi-crystals.
The second author's paper \cite{L} reanalysed Paterson's work in the light of Kellendonk's and presented a general order-based approach 
to the construction of groupoids from inverse semigroups. 
This approach yields  an alternative description of Paterson's universal groupoid and at the same  time  features a certain reduction of the universal groupoid. 
This reduction is the tiling groupoid in the case of tiling semigroups and the graph groupoid in the case of graphs. 
In this way a unified treatment of certain basic properties concerning e.g. ideal theory  of  tiling groupoids and graph groupoids becomes possible.
The order-based approach  to the construction of groupoids from inverse semigroups was then developed  further by the first author 
in collaboration with Stuart Margolis and Ben Steinberg in terms of filters \cite{LMS}.
Thus the topological groupoids arising in the theory of $C^{\ast}$-algebras were groupoids of filters.
This set the stage for \cite{Law3}, where the first author showed that the topological groupoids arising in the case of the Cuntz $C^{\ast}$-algebras
could be constructed from a class of inverse monoids, called boolean inverse monoids, in a way generalising the classical Stone duality
between unital boolean algebras and boolean spaces.
This was subsequently generalized to boolean inverse semigroup in \cite{Law4} where
the Thompson groups and the Cuntz-Krieger $C^{\ast}$-algebras were shown to arise naturally from
this more generality duality.
It is now that frames naturally reappear.
In Johnstone's book \cite{J}, Stone duality is  obtained as a consequence of the
adjunction that exists between the category of frames, or rather their dual the category of locales,
and the category of topological spaces.
Thus the question arises of obtaining the non-commutative Stone dualities described in \cite{Law3,Law4}
in terms of a generalization of frame theory in which frames are replaced by pseudogroups and topological spaces by \'etale topological groupoids.
This was the starting point for this paper.
For applications of our theory to groups and $C^{\ast}$-algebras we refer the reader to \cite{Law4}.

\subsection{Booleanizations} Given an inverse semigroup $S$ we can ask what the best weakly boolean inverse semigroup associated with $S$ might be.
The first answer is given via Paterson's universal groupoid and the weakly boolean inverse semigroup $\mathsf{B}(S)$ characterized in Theorem~5.7.
This semigroup can always be constructed and contains an isomorphic copy of $S$.
However, the construction of  $\mathsf{B}(S)$ does not take account of any boolean-like properties, whatever that might mean, 
that $S$ may already possess.
Alternatively, we may construct by Theorems~4.26 and 4.29  the dense pseudogroup $\mathsf{P}_{d}(S)$.
This is a boolean inverse monoid and takes account of the boolean-like properties of the semigroup $S$ on account 
of the fact that the map $\pi \colon S \rightarrow \mathsf{P}_{d}(S)$ is a dense map.
Howevever, $\mathsf{P}_{d}(S)$ is not coherent in general.
Intuitively, this provides the motivation for the construction of the distributive inverse semigroup $\mathsf{D}_{t}(S)$, the tight completion of $S$,
in Theorem~4.30.
It is not weakly boolean in general, but it is when $S$ satisfies the compactness condition.
In this case, the associated groupoid is constructed from ultrafilters on $S$ and the map $\delta \colon S \rightarrow \mathsf{D}_{t}(S)$ is a tight map.
Inverse semigroups satisfying the compactness condition, were the subject of the second author's papers
which began with an examination in \cite{Lawa,Lawb} of some ideas to be found in \cite{Birget} and then led via \cite{Law3,Law3a}
to \cite{JL,Law4} where special cases of some of the ideas in this paper were developed.


\end{document}